\documentclass{amsart}

\usepackage{amsmath,amsthm,amssymb}
\usepackage{enumerate}
\usepackage{graphicx}
\usepackage{float}
\usepackage[margin=1.5in]{geometry}
\numberwithin{equation}{section}

\theoremstyle{plain}
\newtheorem{thm}{Theorem}[section]
\newtheorem{cor}[thm]{Corollary}
\newtheorem{lem}[thm]{Lemma}
\newtheorem{prop}[thm]{Proposition}

\theoremstyle{definition}
\newtheorem{defn}[thm]{Definition}

\theoremstyle{remark}
\newtheorem{rem}[thm]{Remark}
\newtheorem{ex}[thm]{Example}

\usepackage{tikz}
\usetikzlibrary{arrows,shapes,trees}
\usepackage{graphicx}
\usepackage{hyperref}
\usepackage{amssymb}
\usepackage{amsfonts}
\usepackage{amsthm}
\usepackage{amsmath}
\usepackage{epstopdf}
\usepackage{mathrsfs}

\newcommand{\N}{\mathbb N}

\newcommand{\Z}{\mathbb Z}

\newcommand{\R}{\mathbb R}

\newcommand{\e}{\varepsilon}

\newcommand{\dist}{\mathrm{dist}}

\renewcommand{\t}{\tilde}
\renewcommand{\l}{\ell}
\newcommand{\la}{\lambda}

\newcommand{\p}{\partial}

\def\XXint#1#2#3{{\setbox0=\hbox{$#1{#2#3}{\int}$ }
\vcenter{\hbox{$#2#3$ }}\kern-.6\wd0}}

\title{Uniformizing Gromov hyperbolic spaces with Busemann functions}
\author{Clark Butler}

\begin{document}
\begin{abstract}
Given a complete Gromov hyperbolic space $X$ that is roughly starlike from a point $\omega$ in its Gromov boundary $\p_{G}X$, we use a Busemann function based at $\omega$ to construct an incomplete unbounded uniform metric space $X_{\e}$ whose boundary $\p X_{\e}$ can be canonically identified with the Gromov boundary $\p_{\omega}X$ of $X$ relative to $\omega$. This uniformization construction generalizes the procedure used to obtain the Euclidean upper half plane from the hyperbolic plane. Furthermore we show, for an arbitrary metric space $Z$, that there is a hyperbolic filling $X$ of $Z$ that can be uniformized in such a way that the boundary $\p X_{\e}$ has a biLipschitz identification with the completion $\bar{Z}$ of $Z$. We also prove that this uniformization procedure can be done at an exponent that is often optimal in the case of CAT$(-1)$ spaces. 
\end{abstract}
\maketitle

\section{Introduction}

The goal of this paper is to construct an unbounded analogue of the uniformizations of Gromov hyperbolic spaces built by Bonk, Heinonen and Koskela in their extensive study of a number of problems in conformal analysis \cite{BHK}. The most familiar special case of our procedure is the construction of the upper half-space $\{(x,y):y > 0\}$ in $\R^{2}$ from the hyperbolic plane $\mathbb{H}^{2}$, which is discussed in Example \ref{hyperbolic plane}. The guiding example in \cite{BHK}, by comparison, is the relationship between $\mathbb{H}^{2}$ and the Euclidean unit disk $\{(x,y):x^{2}+y^{2} < 1\}$. As can be seen from these examples, the input for uniformization is a geodesic Gromov hyperbolic space $X$ and the output is an incomplete metric space $\Omega$, obtained from a conformal deformation of $X$, that is \emph{uniform} in the sense of Definition \ref{def:uniform} below. The density used for uniformizing a Gromov hyperbolic space $X$ in \cite{BHK} is exponential in the distance to a fixed point $z$ of $X$. In contrast we will be using a density that is exponential in a \emph{Busemann function} associated to a particular point of the Gromov boundary of $X$. This choice of density is natural as Busemann functions are often interpreted as distance functions ``from infinity" and can themselves be used to define a boundary of the space $X$ \cite[$\S$3]{BGS}. Unlike in \cite{BHK}, we will not assume any local compactness properties on $X$, so specializing our results back to their setting yields a small generalization of their results as well. 

Our principal application of this uniformization construction will be to \emph{hyperbolic fillings} $X$ of a metric space $Z$, with a particular focus on the case in which $Z$ is unbounded. When $Z$ is bounded a hyperbolic filling $X$ of $Z$ can be thought of as a Gromov hyperbolic graph whose Gromov boundary can be canonically identified with $Z$; in the case that $Z$ is unbounded there are some additional subtleties to this notion owing to the fact that the Gromov boundary of a Gromov hyperbolic space is always bounded. We refer to the discussion prior to Theorem \ref{filling theorem} for further information on this, as well as the contents of Section \ref{sec:filling}. Our use of Busemann functions in this setting is inspired by the hyperbolic filling construction of Buyalo and Schroeder \cite[Chapter 6]{BS07} for arbitrary metric spaces $Z$. 

Our uniformization construction for hyperbolic fillings is used in a followup work \cite{Bu23} in order to establish a correspondence between Newton-Sobolev classes of functions on the hyperbolic filling of $Z$ and Besov classes of functions on $Z$ in the special case that $Z$ carries a doubling measure. This is heavily inspired by work of A. Bj\"orn, J. Bj\"orn, and Shanmugalingam \cite{BBS21} that establishes the corresponding result in the case that $Z$ is bounded. In a closely related work \cite{Bu22} we also  generalize to our setting their results \cite{BBS20} on how local Poincar\'e inequalities transform under the uniformization in \cite{BHK}. This yields some interesting new examples of uniform metric spaces satisfying Poincar\'e inequalities. There are a number of known variants on the correspondence between function spaces on the hyperbolic filling and function spaces on $Z$, see for instance \cite{BS18}, \cite{BSS18}, \cite{BP03}. Such correspondences were one of the original motivating factors in the use of hyperbolic fillings in analysis on metric spaces. For applications to trace theorems on Ahlfors regular metric spaces that demonstrate the power of these correspondences we refer to \cite{SS17}.

Lastly we remark that the idea of uniformizing Gromov hyperbolic spaces using Busemann functions has been developed independently by Zhou \cite{Z20} for the purpose of an entirely different set of applications, including a study of Teichmueller's displacement problem for quasi-isometries of Gromov hyperbolic spaces. The work \cite{Z20} in particular gives alternative proofs of the main uniformization theorems (Theorems \ref{unbounded uniformization} and \ref{identification theorem}) restricted to the case of proper Gromov hyperbolic spaces and the original range $0 < \e \leq \e_{0}$ of exponents considered by Bonk-Heinonen-Koskela (see Theorem \ref{Gehring-Hayman} below). Our applications to CAT$(-1)$ spaces and hyperbolic fillings require us to consider exponents outside this range however; this is a key point of departure from \cite{Z20}. 

Stating our main theorems require some preliminary definitions. We opt to give precise definitions in the corresponding sections throughout the paper, while mostly only outlining the necessary definitions here in the introduction. For a metric space $(X,d)$ and a curve $\gamma: I \rightarrow X$, $I \subset \R$ a subinterval, we write $\l(\gamma)$ for the length of $\gamma$ measured in $X$. We will follow the standard practice of using $\gamma$ to denote both the parametrization of the curve and the image of the curve in $X$. The curve $\gamma$ is a \emph{geodesic} if it is isometric as a mapping of $I$ into $X$. We say that $X$ is \emph{geodesic} if any two points can be joined by a geodesic. We will use the following distance notation for distance from a point $x$ to a set $E$ in any metric space $(X,d)$,
\[
\dist(x,E) = \inf_{y \in E} d(x,y),
\]
and in particular will write $\dist(x,\gamma)$ for the distance of a point $x \in X$ to (the image of) a curve $\gamma$.


We now define uniform metric spaces. We start with an \emph{incomplete} metric space $(\Omega,d)$. We denote the boundary of $\Omega$ in its completion $\bar{\Omega}$ by $\p \Omega = \bar{\Omega} \backslash \Omega$. For $x \in \Omega$ we write $d_{\Omega}(x):=\dist(x,\p \Omega)$ for the distance from $x$ to the boundary $\p \Omega$. For the definition below we allow $I = [a,b] \subset \R$ to be any closed interval, and for a curve $\gamma: I \rightarrow X$ we denote its endpoints by $\gamma_{-} := \gamma(a)$ and $\gamma_{+} := \gamma(b)$. For such an interval $I \subset \R$ we write $I_{\leq t} = \{s \in I: s\leq t\}$ and $I_{\geq t} = \{s \in I: s\geq t\}$.  

\begin{defn}\label{def:uniform}For a constant $A \geq 1$ and a closed interval $I \subset \R$, a curve $\gamma: I \rightarrow \Omega$ is \emph{$A$-uniform} if 
\begin{equation}\label{uniform one}
\l(\gamma) \leq Ad(\gamma_{-},\gamma_{+}),
\end{equation}
and if for every $t \in I$ we have
\begin{equation}\label{uniform two}
\min\{\l(\gamma|_{I\leq t}),\l(\gamma|_{I\geq t})\} \leq A d_{\Omega}(\gamma(t)). 
\end{equation}
We say that the metric space $\Omega$ is \emph{$A$-uniform} if any two points in $\Omega$ can be joined by an $A$-uniform curve. 
\end{defn}

Many reasonable domains in Euclidean space such as the unit ball or upper half-space provide natural examples of uniform metric spaces when they are equipped with the Euclidean metric. The first requirement \eqref{uniform one} implies that $A$-uniform curves minimize the distance between their endpoints up to the multiplicative constant $A$. The second requirement \eqref{uniform two} implies that if we cut $\gamma$ at any point $\gamma(t)$ then at least one of the two subcurves $\gamma|_{I_{\leq t}}$ or $\gamma|_{I_{\geq t}}$ must have length controlled by the distance $d_{\Omega}(\gamma(t))$ of $\gamma(t)$ to $\p \Omega$. We note that it is easily verified from the definitions that the property of a curve $\gamma$ being $A$-uniform is independent of the choice of parametrization of $\gamma$. For the purpose of formulating our theorems it is convenient to extend the definition of $A$-uniform curves to allow for arbitrary subintervals $I \subset \R$ and to allow the possibility $\l(\gamma) = \infty$; as this extension is somewhat technical we refer to Definition \ref{def:extend uniform} for the exact details.

\begin{rem}\label{no local compact} The definition of uniform metric spaces advanced in \cite{BHK} also requires local compactness. We follow V\"ais\"al\"a \cite{V99} in dropping the local compactness requirement, as the output of our uniformization procedure need not be locally compact in many cases of interest.  
\end{rem}

For a continuous function $\rho: X \rightarrow (0,\infty)$ we write
\[
\l_{\rho}(\gamma) = \int_{\gamma} \rho \, ds,
\]
for the line integral of $\rho$ along $\gamma$. We refer to \cite[Appendix]{BHK} for a detailed discussion of line integrals in our context. We will often refer to such a positive continuous function $\rho$ as a \emph{density} on $X$. The following definition plays a key role in the statement of our main theorems.

\begin{defn}\label{conformal factor}
Let $(X,d)$ be a geodesic metric space and let $\rho: X \rightarrow (0,\infty)$ be a density on $X$. The \emph{conformal deformation of $X$ with conformal factor $\rho$} is the metric space $X_{\rho} = (X,d_{\rho})$ with metric
\[
d_{\rho}(x,y) = \inf  \l_{\rho}(\gamma),
\]
with the infimum taken over all curves $\gamma$ joining $x$ to $y$. We say that the density $\rho$ is a \emph{Gehring-Hayman density} (abbreviated as a \emph{GH-density}) if there is a constant $M \geq 1$ such that for any $x,y \in X$ and any geodesic $\gamma$ joining $x$ to $y$ we have
\begin{equation}\label{first GH}
\l_{\rho}(\gamma) \leq M d_{\rho}(x,y). 
\end{equation}
\end{defn}

We will refer to the inequality \eqref{first GH} as the \emph{GH-inequality} and will sometimes refer to the constant $M$ as the \emph{$GH$-constant}. The terminology here is inspired by the work of Gehring-Hayman \cite{GH62}, which shows that in a simply connected hyperbolic domain $\Omega$ in the complex plane the hyperbolic geodesics minimize Euclidean length among all curves in the domain with the same end points, up to a universal multiplicative constant. Here the density $\rho$ is given by the conformal change of metric relating the Euclidean metric on $\Omega$ to the hyperbolic metric.  Note that if $X$ is a tree then the GH-inequality holds for any density $\rho$ with $M = 1$ since any path joining two points in a tree must contain the geodesic joining those points. 

We next discuss the notions we will need regarding Gromov hyperbolic spaces. Most formal definitions regarding Gromov hyperbolicity and the Gromov boundary are postponed to Section \ref{sec:hyperbolic}, as they can be found in any standard reference such as \cite{BS07}, \cite{GdH90}. A geodesic metric space $X$ is \emph{Gromov hyperbolic} if there is a $\delta \geq 0$ such that all geodesic triangles are \emph{$\delta$-thin}, meaning that for any geodesic triangle $\Delta$ each edge of $\Delta$ is contained in a $\delta$-neighborhood of the other two edges of $\Delta$. In this case we will also say that $X$ is \emph{$\delta$-hyperbolic}. We write $\p_{G}X$ for the Gromov boundary of $X$, to be defined in Section \ref{sec:hyperbolic}; for now we note that a geodesic ray $\gamma:[0,\infty) \rightarrow X$ can always be identified with an equivalence class $[\gamma] \in \p_{G}X$, but in general not every point in $\p_{G}X$ can be realized in this way. 

We consider a complete geodesic $\delta$-hyperbolic space $X$ and a geodesic ray $\gamma:[0,\infty) \rightarrow X$. The \emph{Busemann function} $b_{\gamma}: X \rightarrow \R$ associated to $\gamma$ is defined by the limit
\begin{equation}\label{first busemann definition}
b_{\gamma}(x) = \lim_{t \rightarrow \infty} d(\gamma(t),x)-t. 
\end{equation}
Using the triangle inequality and the fact that $d(\gamma(t),\gamma(0)) = t$, it's easy to check that the right side is nonincreasing in $t$ and bounded below by $-d(\gamma(0),x)$, so this limit exists. It's also easily verified that $b_{\gamma}$ is 1-Lipschitz, thus in particular is continuous. As is customary when considering Busemann functions, we will refer to any translate $b = b_{\gamma} + s$ of $b_{\gamma}$ for a constant $s \in \R$ as a Busemann function as well. We write
\begin{equation}\label{extension busemann definition}
\mathcal{B}(X) = \{b_{\gamma}+s: \text{$\gamma$ a geodesic ray in $X$, $s \in \R$}\},
\end{equation}
for the set of all Busemann functions on $X$. Given a Busemann function $b = b_{\gamma} + s$ we will write $\omega_{b} = [\gamma] \in \p_{G}X$ for the point in the Gromov boundary determined by the geodesic ray $\gamma$. We will refer to $\omega_{b}$ as the \emph{basepoint} of $b$ and say that $b$ is \emph{based at $\omega_{b}$}. 

To state our theorems in their appropriate generality it is useful to augment the set $\mathcal{B}(X)$ with the distance functions on $X$: for $z \in X$ we write $b_{z}(x) = d(x,z)$ for the distance function to $z$ and write
\begin{equation}\label{extension distance definition}
\mathcal{D}(X) = \{b_{z}+s: \text{$z \in X$, $s \in \R$}\},
\end{equation}
for the set of all translates of distance functions on $X$. We then write $\hat{\mathcal{B}}(X) = \mathcal{B}(X) \cup \mathcal{D}(X)$. In the case $b = b_{z} + s$ we write $\omega_{b} = z$ and refer to $\omega_{b}$ as the basepoint of $b$ as well. The defining formula \eqref{first busemann definition} for Busemann functions shows that any $b \in \mathcal{B}(X)$ can be realized as a pointwise limit of functions $b_{t} \in \mathcal{D}(X)$ defined by $b_{t}(x) = d(\gamma(t),x)-t$.

Given $b \in \hat{\mathcal{B}}(X)$ and $\e > 0$ we define a density $\rho_{\e,b}$ on $X$ by 
\[
\rho_{\e,b}(x) = e^{-\e b(x)}.
\]
We write $X_{\e,b} = (X,d_{\e,b})$ for the conformal deformation of $X$ with conformal factor $\rho_{\e,b}$. In the theorem below we will be assuming that $X$ is \emph{$K$-roughly starlike} from the basepoint $\omega_{b}$ of $b$. This is a technical condition on geodesics starting from $\omega_{b}$ that is described in Definition \ref{def:rough star}. The main purpose of this hypothesis is to rule out cases such as trees that have arbitrarily long finite branches.  This $K$-rough starlikeness condition will be satisfied with $K = \frac{1}{2}$ in our application of Theorem \ref{unbounded uniformization} to hyperbolic fillings in Theorem \ref{filling theorem}. 

\begin{thm}\label{unbounded uniformization}
Let $X$ be a complete geodesic $\delta$-hyperbolic space and let $b \in \hat{\mathcal{B}}(X)$ be given. We suppose that $X$ is $K$-roughly starlike from $\omega_{b}$ and that $\e > 0$ is given such that $\rho_{\e,b}$ is a GH-density with constant $M$. 

Then geodesics in $X$ are $A$-uniform curves in $X_{\e,b}$, with $A = A(\delta,K,\e,M)$. Consequently $X_{\e,b}$ is an $A$-uniform metric space. Furthermore $X_{\e,b}$ is bounded if and only if $b \in \mathcal{D}(X)$. 
\end{thm}

In this statement and all subsequent ones the notation $A = A(\delta,K,\e,M)$ is used to indicate that a particular constant depends on the indicated parameters. We refer to Definition \ref{def:extend uniform} for the extension of the definition of $A$-uniform curves that is necessary to cover the case of an arbitrary geodesic in $X$; Definition \ref{def:uniform} only covers the case of geodesics defined on closed intervals. The claim that $X_{\e,b}$ is bounded if and only if $b \in \mathcal{D}(X)$ does not require either the rough starlikeness hypothesis or the assumption that $\rho_{\e,b}$ is a GH-density, see Proposition \ref{bounded equivalence}.

We describe the motivating example for Theorem \ref{unbounded uniformization} in the case $b \in \mathcal{B}(X)$ below. 

\begin{ex}\label{hyperbolic plane}
Let $\mathbb{U}^{2} = \{(x,y) \in \R^{2}: y > 0\}$ be the upper half space in $\R^{2}$ equipped with the Euclidean metric, which is easily seen to be a uniform metric space. Let $\mathbb{H}^{2}$ denote the upper half plane model of the hyperbolic plane, which is $\mathbb{U}^{2}$ equipped with the Riemannian metric $ds^{2} = \frac{dx^{2}+dy^{2}}{y^{2}}$. Define $\gamma: [0,\infty) \rightarrow \mathbb{H}^{2}$ by $\gamma(t) = (0,e^{t})$. Then $\gamma$ is a geodesic ray in $\mathbb{H}^{2}$. 

From explicit formulas for the hyperbolic distance in this model (see for instance \cite[A.3]{BS07}) it is straightforward to calculate that the associated Busemann function is given by $b_{\gamma}(x,y) = -\log y$. Setting $\e = 1$, the density $\rho_{1,b_{\gamma}}$ is thus simply given by $\rho_{1,b_{\gamma}}(x,y) = y$. Therefore the uniformized metric space $\mathbb{H}^{2}_{1,b_{\gamma}}$ is isometric to $\mathbb{U}^{2}$. We also remark that the GH-inequality \eqref{first GH} for the density $\rho_{1,b_{\gamma}}$ can easily be verified using the standard representation of geodesics in the upper half-plane model for $\mathbb{H}^{2}$ as subsegments of semicircles or vertical lines orthogonal to the horizontal line $\{y = 0\}$ in $\R^{2}$.  
\end{ex}

A metric space $(X,d)$ is \emph{proper} if its closed balls are compact. In the case $b \in \mathcal{D}(X)$, Theorem \ref{unbounded uniformization} generalizes \cite[Proposition 4.5]{BHK} as it does not require $X$ to be proper and allows for a potentially larger range of values for the parameter $\e$; this larger range will be relevant to Theorem \ref{CAT theorem}. 


For a $\delta$-hyperbolic space $X$ and a point $\omega \in \p_{G}X$ we write $\p_{\omega} X = \p_{G}X \backslash \{\omega\}$ for the complement of $\omega$ in the Gromov boundary of $X$. We will refer to $\p_{\omega} X$ as the \emph{Gromov boundary relative to $\omega$} for reasons that will be explained prior to Proposition \ref{convergence Busemann}. We formally extend this definition to $\omega \in X$ by defining $\p_{\omega}X = \p_{G}X$; in this case we will still refer to $\p_{\omega} X$ as the Gromov boundary relative to $\omega$, with the understanding that this simply coincides with the standard Gromov boundary for $\omega \in X$. As part of the proof of Theorem \ref{unbounded uniformization}, we will show that there is a canonical identification $\varphi_{\e,b}: \p_{\omega_{b}} X \rightarrow \p X_{\e,b}$ between the Gromov boundary of $X$ relative to $\omega_{b}$ and the boundary of $X_{\e,b}$ in its completion. The most important property of this identification is summarized in Theorem \ref{identification theorem} below. 

A function $b \in \hat{\mathcal{B}}(X)$ can be used to define a natural class of metrics on $\p_{\omega_{b}}X$ known as \emph{visual metrics based at $\omega_{b}$} (see \cite[Chapter 3]{BS07} as well as Section \ref{subsec:visual}). These visual metrics have an associated parameter $\e > 0$ and a comparison constant $L$ to a specific model quasi-metric $\theta_{\e,b}$ on $\p_{\omega_{b}}X$ defined in \eqref{visual quasi}. We continue to write $d_{\e,b}$ for the canonical extension of the metric on the uniformization $X_{\e,b}$ to its completion $\bar{X}_{\e,b}$. 


\begin{thm}\label{identification theorem}
Let $X$ be a complete geodesic $\delta$-hyperbolic space and let $b \in \hat{\mathcal{B}}(X)$ be such that $X$ is $K$-roughly starlike from the basepoint $\omega$ of $b$. Let $\e > 0$ be given such that $\rho_{\e,b}$ is a GH-density with constant $M$. Then there is a canonical identification $\varphi_{\e,b}: \p_{\omega}X \rightarrow \p X_{\e,b}$ under which the restriction of $d_{\e,b}$ to $\p X_{\e,b}$ defines a visual metric on $\p_{\omega}X$ based at $\omega$ with parameter $\e$ and comparison constant $L = L(\delta,K,\e,M)$.  
\end{thm}

For a precise description of the identification $\varphi_{\e,b}$ we refer to the disucssion after \eqref{construct identify}. 

\begin{rem}\label{flexibility}
It is useful to allow some additional flexibility in the choice of function $b$ in Theorems \ref{unbounded uniformization} and \ref{identification theorem}. This flexibility will be used in Theorem \ref{filling theorem}. For a continuous function $b: X \rightarrow \R$ and a constant $\kappa \geq 0$ we write $b \in \hat{\mathcal{B}}_{\kappa}(X)$ if there is some $b' \in \hat{\mathcal{B}}(X)$ such that $|b(x)-b'(x)| \leq \kappa$ for all $x \in X$; we write $b \in \mathcal{D}_{\kappa}(x)$ if $b' \in \mathcal{D}(X)$ and $b \in \mathcal{B}_{\kappa}(x)$ if $b' \in \mathcal{B}(X)$. We define the basepoint of $b$ to be the basepoint of $b'$, $\omega_{b} = \omega_{b'}$; while this definition may be ambiguous in the case $b \in \mathcal{D}_{\kappa}(X)$, this ambiguity does not matter in the context of our theorems. Then $X_{\e,b}$ is $e^{\e \kappa}$-biLipschitz to $X_{\e,b'}$ via the identity map on $X$ and $\rho_{\e,b'}$ will be a GH-density with constant $e^{2\e \kappa}M$ if $\rho_{\e,b}$ is a GH-density with constant $M$. It then easily follows that Theorems \ref{unbounded uniformization} and \ref{identification theorem} hold for $b \in \hat{\mathcal{B}}_{\kappa}(X)$ as well, with the uniformity parameter $A$ in Theorem \ref{unbounded uniformization} and the comparison constant $L$ in Theorem \ref{identification theorem} depending additionally on $\kappa$. 
\end{rem}

\begin{rem}\label{generalize to Banach}
Since we do not assume that the Gromov hyperbolic space $X$ in our theorems is proper, it is an interesting question whether our theorems can be applied to the ``free quasiworld" considered by V\"ais\"al\"a \cite{V99}. The Gromov hyperbolic spaces that arise in this context are domains in Banach spaces equipped with hyperbolic metrics. However in this setting the hypothesis that $X$ is geodesic is too strong \cite[Remark 3.5]{V99}. The best one can assume is that $X$ is a \emph{length space}, i.e., that the distance between two points of $X$ is equal to the infimum of the lengths of all curves joining them. Thus one would need to generalize Theorems \ref{unbounded uniformization} and \ref{identification theorem} to Gromov hyperbolic spaces that are not necessarily geodesic, but are still length spaces. We believe that such a generalization is possible, but since it is unnecessary for our applications we will not pursue it here. 
\end{rem}

Let's now discuss when the hypotheses of Theorems \ref{unbounded uniformization} and \ref{identification theorem} are satisfied in practice. The two key hypotheses are the rough starlikeness hypothesis from the basepoint $\omega_{b}$ of $b$ and the assumption that $\rho_{\e,b}$ is a GH-density on $X$. The rough starlikeness hypothesis is always easily verified in applications of interest, so as a consequence it is typically not a concern when trying to apply these theorems. Thus the main hypothesis to verify is that of $\rho_{\e,b}$ being a GH-density. The most general result available regarding verifying this condition is the following theorem of Bonk-Heinonen-Koskela. 

\begin{thm}\label{Gehring-Hayman}\cite[Theorem 5.1]{BHK}
Let $(X,d)$ be a geodesic $\delta$-hyperbolic space. There is $\e_{0} = \e_{0}(\delta) > 0$ depending only on $\delta$ such that if a density $\rho:X \rightarrow (0,\infty)$ satisfies for all $x,y \in X$ and some fixed $0 < \e \leq \e_{0}$,
\begin{equation}\label{proto Harnack}
e^{-\e d(x,y)} \leq \frac{\rho(x)}{\rho(y)} \leq e^{\e d(x,y)},
\end{equation}
then $\rho$ is a GH-density with constant $M = 20$. 
\end{thm} 

The inequality \eqref{proto Harnack} is satisfied for $\rho_{\e,b}$ for any $b \in \hat{\mathcal{B}}(X)$ since all functions in $\hat{\mathcal{B}}(X)$ are $1$-Lipschitz. Theorem \ref{Gehring-Hayman} builds on a number of previous works that are summarized at the beginning of \cite[Chapter 5]{BHK}. Thus if one is not concerned about obtaining Theorems \ref{unbounded uniformization} and \ref{identification theorem} for a specific value of $\e$ then it is always possible to assume that $\rho_{\e,b}$ is a GH-density with constant $M = 20$ by taking $\e$ sufficiently small.

In general one wants to establish that $\rho_{\e,b}$ is a GH-density for as large a value of $\e$ as possible, as this property is then inherited for smaller values of $\e$ by Proposition \ref{inheritance}. This is particularly important for applications in which there is a preferred visual metric on $\p_{\omega}X$, such as Theorems \ref{CAT theorem} and \ref{filling theorem} below.

CAT$(-1)$ spaces are geodesic metric spaces in which geodesic triangles are thinner than corresponding comparison geodesic triangles in the hyperbolic plane $\mathbb{H}^{2}$. We refer to \cite[Definition 3.2.1]{DSU17} for a precise definition; since the proper definition is somewhat lengthy to state and we will only be using easily stated consequences of the CAT$(-1)$ property, we omit a full description of the definition here. These spaces encompass many natural examples such as trees and simply connected Riemannian manifolds with sectional curvatures $\leq -1$. A CAT$(-1)$ space is $\delta$-hyperbolic with the same hyperbolicity constant $\delta = \delta(\mathbb{H}^{2})$ as the hyperbolic plane, for which the optimal constant can be computed explicitly to be $\delta = \log(1+\sqrt{2})$.

For a CAT$(-1)$ space $X$ and a function $b \in \hat{\mathcal{B}}(X)$ with basepoint $\omega$ the model quasi-metric $\theta_{1,b}$ on $\p_{\omega}X$ in fact defines a distinguished choice of visual metric on $\p_{\omega}X$ with parameter $\e = 1$. This metric is known as the \emph{Bourdon metric} on $\p_{\omega}X = \p X$ when $b \in \mathcal{D}(X)$ and the \emph{Hamenst\"adt metric} on $\p_{\omega}X$ when $b \in \mathcal{B}(X)$. For further details we refer to Remark \ref{CAT visual}. Our next theorem applies Theorems \ref{unbounded uniformization} and \ref{identification theorem} to the special case of CAT$(-1)$ spaces at the special value $\e = 1$. 

\begin{thm}\label{CAT theorem}
Let $X$ be a complete CAT$(-1)$ space and let $b \in \hat{\mathcal{B}}(X)$ be given with basepoint $\omega$. Then there is a universal constant $M \geq 1$ such that $\rho_{1,b}$ is a GH-density with constant $M$. If, furthermore, $X$ is $K$-roughly starlike from the basepoint $\omega$ of $b$ for some $K \geq 0$ then the conclusions of Theorems \ref{unbounded uniformization} and \ref{identification theorem} hold for $X_{1,b}$ with constants $A = A(K)$ and $L = L(K)$ depending only on $K$. In particular the restriction of $d_{1,b}$ to $\p X_{1,b}$ is $L$-biLipschitz to $\theta_{1,b}$.
\end{thm}

The constant $M$ in Theorem \ref{CAT theorem} is universal in the sense that it is the same for any CAT$(-1)$ space $X$ and any $b \in \hat{\mathcal{B}}(X)$. The dependence of $A$ on $K$ can be removed when $b \in \mathcal{D}(X)$ by mimicking the arguments of \cite[Proposition 4.5]{BHK}; this same comment applies to Theorem \ref{unbounded uniformization} as well. The conclusions of Theorem \ref{CAT theorem} show in particular that the boundary $\p X_{1,b}$ of the uniformization has a canonical biLipschitz identification with the Gromov boundary $\p_{\omega}X$ relative to $\omega$ equipped with the distinguished visual metric $\theta_{1,b}$. 

\begin{rem}\label{sharp} Theorem \ref{identification theorem} produces an obstruction for $\rho_{\e,b}$ to be a GH-density: the Gromov boundary $\p_{\omega}X$ must admit a visual metric based at the basepoint $\omega$ of $b$ with parameter $\e$. In the case $b = b_{z}$ for some $z \in X$ (i.e., if $b$ is a distance function) then this shows in particular that $\e \leq -K_{u}(X)$, where $K_{u}(X)$ is the asymptotic upper curvature bound defined by Bonk and Foertsch (see \cite[Theorem 1.5]{BF06}). In the case of the $n$-dimensional hyperbolic space $\mathbb{H}^{n}$ of constant negative curvature $-1$ ($n \geq 2$) we have $K_{u}(\mathbb{H}^{n}) = 1$ by \cite[Proposition 1.4]{BF06}. Hence $\rho_{\e,b}$ cannot be a GH-density for any $\e > 1$ when $b \in \mathcal{D}(\mathbb{H}^{n})$. This shows in particular that the value $\e = 1$ for $\rho_{1,b}$ to be a GH-density in Theorem \ref{CAT theorem} is sharp in certain cases. We can in fact extend these conclusions to observe that $\rho_{\e,b}$ cannot be a GH-density for any $\e > 1$ when $b \in \mathcal{B}(\mathbb{H}^{n})$ as well by observing that, for a fixed $\omega \in \p\mathbb{H}^{n}$, any visual metric based at $\omega$ on $\p_{\omega}\mathbb{H}^{n}$ with parameter $\e > 1$ would give rise to a visual metric on $\p \mathbb{H}^{n}$ with parameter $\e > 1$ by a M\"obius inversion of $\p_{\omega}\mathbb{H}^{n}$ centered at the point $\omega$. 
\end{rem}

We will also apply our uniformization results to hyperbolic fillings of an arbitrary metric space $(Z,d)$. We briefly describe the construction of the hyperbolic filling  here, with further details in Section \ref{sec:filling}, including proofs for the claims made here. Our construction will depend in part on two parameters $\alpha > 1$ and $\tau > 1$. For an $r > 0$ we say that a subset $S \subset Z$ is \emph{$r$-separated} if for each $x,y \in S$ we have $d(x,y) \geq r$. Given a parameter $\alpha > 1$, we choose for each $n \in \Z$ a maximal $\alpha^{-n}$-separated subset $S_{n}$ of $Z$. For $n \in \Z$ we write $V_{n} = \{(z,n): z \in S_{n}\}$ and set $V = \bigcup_{n \in \Z} V_{n}$. The set $V$ will serve as the vertex set for $X$. We define the \emph{height function} $h: V \rightarrow \Z$ on this vertex set by $h(v) = n$ for $v = (z,n) \in V_{n}$.

We associate to each vertex $v = (z,n)\in V$ the ball $B(v) = B(z,\tau \alpha^{-n})$ of radius $\tau \alpha^{-n}$ centered at $z$. We place an edge between vertices $v,w \in V$ if and only if their heights satisfy $|h(v)-h(w)| \leq 1$ and their associated balls satisfy $B(v) \cap B(w) \neq \emptyset$. We write $X$ for the resulting graph and call this a \emph{hyperbolic filling} of $Z$. If $\tau$ is sufficiently large (see inequality \eqref{tau requirement}) then $X$ will be a connected graph by Proposition \ref{connected filling}. We make $X$ into a geodesic metric space by declaring all edges to have unit length. We extend the height function $h$ to a $1$-Lipschitz function $h: X \rightarrow \R$ by linearly interpolating the values of $h$ from the vertices to the edges of $X$. 

As a metric space $X$ is $\delta$-hyperbolic with $\delta = \delta(\alpha,\tau)$ depending only on the parameters $\alpha$ and $\tau$. There is a distinguished point $\omega \in \p_{G}X$ in the Gromov boundary that can be thought of as an ideal point at infinity for $Z$. We have an identification $\p_{\omega}X \cong \bar{Z}$ of the Gromov boundary relative to $\omega$ with the completion $\bar{Z}$ of $Z$. Under this identification the extension of the metric $d$ to $\bar{Z}$ defines a visual metric on $\p_{\omega}X$ with parameter $\e = \log \alpha$. All of the results mentioned here are proved in Section \ref{sec:filling}.

We define a density $\rho$ on $X$ by $\rho(x) = \alpha^{-h(x)}$ for $x \in X$. We write $X_{\rho}$ for the conformal deformation of $X$ with conformal factor $\rho$. By Lemma \ref{height busemann} there is a Busemann function $b$ based at $\omega$ such that $|h(x)-b(x)| \leq 3$ for all $x \in X$ and therefore $h \in \mathcal{B}_{3}(X)$ in the notation of Remark \ref{flexibility}. For such a Busemann function $b \in \mathcal{B}(X)$ we have that the density $\rho$ is uniformly comparable to the density $\rho_{\e,b}$ with $\e = \log \alpha$.

\begin{thm}\label{filling theorem}
Let $Z$ be a metric space and let $X$ be a hyperbolic filling of $Z$ with parameters $\alpha > 1$ and $\tau > \min\{3,\alpha/(\alpha-1)\}$. Then $X$ is $\frac{1}{2}$-roughly starlike from $\omega$ and $\rho$ is a GH-density with constant $M = M(\alpha,\tau)$.

Thus the conclusions of Theorems \ref{unbounded uniformization} and \ref{identification theorem} hold for $X_{\rho}$. In particular we have a canonical $L$-biLipschitz identification of $\p X_{\rho}$ and $\bar{Z}$, with $L = L(\alpha,\tau)$. 
\end{thm}

We compare our results to those of \cite{BBS21} in Remark \ref{bounded filling}. Another  notable predecessor to Theorem \ref{filling theorem} in the case that $Z$ is compact is the work of Piaggio \cite[Section 2]{P13}.

We provide here an outline of the contents of the rest of the paper. In Section \ref{sec:hyperbolic} we review several key notions in the setting of Gromov hyperbolic spaces. Section \ref{sec:tripod} establishes some basic properties of geodesic triangles in Gromov hyperbolic spaces with vertices on the Gromov boundary and gives a rough formula for evaluating certain distance functions and Busemann functions on their edges. We then use these results in Section \ref{sec:uniformize} to obtain estimates for the uniformized distance and prove Theorems \ref{unbounded uniformization}, \ref{identification theorem}, and \ref{CAT theorem}. In Section \ref{sec:filling} we construct the hyperbolic fillings of metric spaces that we use in Theorem \ref{filling theorem} and establish their basic properties. Lastly Theorem \ref{filling theorem} is proved in Section \ref{sec:uniform filling}. 

We are very grateful to Nageswari Shanmugalingam for providing multiple drafts of the work \cite{BBS21} on which a significant part of this paper is based. We also thank Tushar Das for making us aware of the results of \cite{DSU17} that are used to prove Theorem \ref{CAT theorem}.

\section{Hyperbolic metric spaces}\label{sec:hyperbolic}

\subsection{Definitions}\label{subsec:Def} Let $X$ be a set and let $f$, $g$ be real-valued functions defined on $X$. For $c \geq 0$ we will write $f \doteq_{c} g$ if
\[
|f(x)-g(x)| \leq c,
\] 
for all $x \in X$. If the exact value of the constant $c$ is not important or implied by context we will often just write $f \doteq g$. We will sometimes refer to the relation $f \doteq g$ as a \emph{rough equality} between $f$ and $g$. 

If $C \geq 1$ and $f$ and $g$ both take values in $(0,\infty)$, we will write $f \asymp_{C} g$ if  
\[
C^{-1}g(x) \leq f(x) \leq Cg(x).
\]
We will similarly write $f \asymp g$ if the value of $C$ is not important or implied by context. Note that if $f \doteq_{c} g$ then $e^{f} \asymp_{e^{c}} e^{g}$, and similarly if $f \asymp_{C} g$ then $\log f \doteq_{\log C} \log g$. We will stick to a convention of using lowercase $c$ for additive constants and uppercase $C$ for multiplicative constants. When this additive constant $c$ is determined by other parameters $\delta$, $K$, etc. under discussion we will write $c = c(\delta,K)$, while continuing to use the shorthand $c$ where it is not ambiguous (and the same for multiplicative constants $C$).

For a metric space $(X,d)$  we write $B(x,r) = \{y:d(x,y) < r\}$ for the open ball of radius $r > 0$ centered at $x$. A map $f:(X,d) \rightarrow (X',d')$ between metric spaces $X$ and $X'$ is \emph{isometric} if $d'(f(x),f(y)) = d(x,y)$ for $x$, $y\in X$. If furthermore $f$ is surjective then we say that it is an \emph{isometry} and that $X$ and $X'$ are \emph{isometric}. For a constant $c \geq 0$ a map $f:X \rightarrow X'$ is defined to be \emph{$c$-roughly isometric} if $d'(f(x),f(y)) \doteq_{c} d(x,y)$. The map $f$ is \emph{$L$-Lipschitz} for a constant $L \geq 0$ if $d'(f(x),f(y)) \leq Ld(x,y)$, and it is  \emph{$L$-biLipschitz} for a constant $L \geq 1$ if $d'(f(x),f(y)) \asymp_{L} d(x,y)$. As usual we will not mention the exact value of the constants if they are unimportant. 

When dealing with Gromov hyperbolic spaces $X$ in this paper we will use the generic distance notation $|xy|:=d(x,y)$ for the distance between $x$ and $y$ in $X$, except for cases where this could cause confusion. We will often use the generic notation $xy$ for a geodesic connecting two points $x,y \in X$, even when this geodesic is not unique; in these cases there will be no ambiguity regarding the geodesic that we are referring to. A \emph{geodesic triangle} $\Delta$ in $X$ is a collection of three points $x,y,z \in X$ together with geodesics $xy$, $xz$, and $yz$ joining these points, which we will refer to as the \emph{edges} of $\Delta$. We will also alternatively write $xyz = \Delta$ for a geodesic triangle with vertices $x$, $y$ and $z$. 

For $x,y,z \in X$ the \emph{Gromov product} of $x$ and $y$ based at $z$ is defined by
\begin{equation}\label{Gromov product}
(x|y)_{z} = \frac{1}{2}(|xz|+|yz|-|xy|). 
\end{equation}
We note the basepoint change inequality for $x,y,z,p \in X$,
\begin{equation}\label{basepoint change}
|(x|y)_{z}-(x|y)_{p}| \leq |zp|,
\end{equation}
which follows from the triangle inequality. 

By \cite[Chapitre 2, Proposition 21]{GdH90} we have two key consequences of $\delta$-hyperbolicity for a metric space $X$ regarding Gromov products. The first is that for every $x,y,z,p \in X$ we have
\begin{equation}\label{delta inequality}
(x|z)_{p} \geq \min \{(x|y)_{p},(y|z)_{p}\} - 4\delta. 
\end{equation}
We refer to \eqref{delta inequality} as the \emph{$4\delta$-inequality}.

The second is that for any geodesic triangle $xyz$ in $X$ we have that if $p \in xy$, $q \in xz$ are points with $|xp| = |xq| \leq (y|z)_{x}$ then $|pq| \leq 4\delta$. Here $xy$ and $xz$ are referring to the corresponding geodesics in the triangle $\Delta$. We will refer to this as the \emph{$4\delta$-tripod condition}.

Both inequality \eqref{delta inequality} and the tripod condition can be taken as equivalent definitions of hyperbolicity. By \cite[Chapitre 2, Proposition 21]{GdH90} all of these definitions are equivalent up to a factor of $4$. We note that the definition using inequality \eqref{delta inequality} does not use the fact that $X$ is geodesic, and is therefore used as a definition of $\delta$-hyperbolicity for general metric spaces. We will be citing several basic results from \cite{BS07} in which inequality \eqref{delta inequality} is used as the definition of $\delta$-hyperbolicity (with $\delta$ in place of $4\delta$). Wherever necessary we have multiplied the constants used in their results by $4$ in order to account for this discrepancy.

Let $X$ be a geodesic Gromov hyperbolic space and fix $p \in X$. A sequence $\{x_{n}\}$ \emph{converges to infinity} if we have $(x_{n}|x_{m})_{p} \rightarrow \infty$ as $m,n \rightarrow \infty$. The \emph{Gromov boundary} $\p_{G} X$ of a Gromov hyperbolic space $X$ is defined to be the set of all equivalence classes of sequences $\{x_{n}\} \subset X$ converging to infinity, with the equivalence relation $\{x_{n}\} \sim \{y_{n}\}$ if $(x_{n}|y_{n})_{p} \rightarrow \infty$ as $n \rightarrow \infty$. Inequality \eqref{basepoint change} shows that these notions do not depend on the choice of basepoint $p$. 

A second boundary that we can associate to $X$ is the \emph{geodesic boundary} $\p^{g}X$, which is defined as equivalence classes of geodesic rays $\gamma: [0,\infty) \rightarrow X$, with two geodesic rays $\gamma$ and $\sigma$ being equivalent if there is a constant $c \geq 0$ such that $|\gamma(t)\sigma(t)| \leq c$ for $t \geq 0$. There is a natural inclusion $\p^{g}X \subseteq \p_{G}X$ given by sending a geodesic ray $\gamma$ to the sequence $\{\gamma(n)\}_{n \in \N}$. This inclusion need not be surjective in general. However, it is always surjective if $X$ is \emph{proper}, meaning that closed balls in $X$ are compact. 

For a point $\omega \in \p_{G}X$ and a sequence $\{x_{n}\}$ converging to infinity we will write $\{x_{n}\} \in \omega$ or $x_{n} \rightarrow \omega$ if $\{x_{n}\}$ belongs to the equivalence class of $\omega$. For a geodesic ray $\gamma:[a,\infty) \rightarrow X$, $a \in \R$, and a point $\omega \in \p_{G}X$ we will write $\gamma \in \omega$ if $\{\gamma(n)\}_{n \geq a} \in \omega$, $n \in \N$. We will sometimes also consider geodesic rays $\gamma: (-\infty,a] \rightarrow X$ with a reversely oriented parametrization, for which we write $\gamma \in \omega$ if $\{\gamma(-n)\}_{n \geq -a} \in \omega$.

For the rest of this paper we will be using the standard notation $\p X:=\p_{G} X$ for the Gromov boundary of a Gromov hyperbolic space $X$. While this notation does conflict with the notation $\p \Omega = \bar{\Omega} \backslash \Omega$ introduced prior to Definition \ref{def:uniform}, the meaning of the notation will always be clear from context since we will never use it in the sense of Definition \ref{def:uniform} in the context of Gromov hyperbolic spaces.

We now extend some notions regarding geodesic triangles to the Gromov boundary.   For a point $x \in X$ and a point $\xi \in \p X$ we will often write $x\xi$ for a geodesic ray $\gamma: [0,\infty) \rightarrow X$ with $\gamma(0)= x$ and $\gamma \in \xi$, provided such a geodesic ray exists. Similarly, for $\zeta,\xi \in \p X$ we will write $\zeta\xi$ for a geodesic line $\gamma: \R \rightarrow X$ with $\gamma|_{(-\infty,0]} \in \zeta$ and $\gamma|_{[0,\infty)} \in \xi$, provided such a geodesic line exists. Such geodesic lines and rays always exist when $X$ is proper, but not necessarily in general. We extend the definition of geodesic triangles $\Delta$ in $X$ to allow for vertices in $\p X$: a geodesic triangle $xyz = \Delta$ in $X$ is a collection of three points $x,y,z \in X \cup \p X$ together with geodesics $xy$, $xz$, $yz$ connecting them in the sense described above. 

\begin{rem}\label{distinct}
It is easy to see from the definitions that there is no geodesic $\gamma: \R \rightarrow X$ such that $\gamma|_{[0,\infty)}$ and $\gamma|_{(-\infty,0]}$ belong to the same equivalence class in the Gromov boundary $\p X$. Hence, for a geodesic triangle $\Delta$, all vertices of $\Delta$ on $\p X$ must be distinct. 
\end{rem}

Gromov products based at points $p \in X$ can be extended to points of $\p X$ by defining the Gromov product of equivalence classes $\xi$, $\zeta \in \p X$ based at $p$ to be
\[
(\xi |\zeta)_{p} = \inf \liminf_{n \rightarrow \infty}(x_{n}|y_{n})_{p},
\] 
with the infimum taken over all sequences $\{x_{n}\} \in \xi$, $\{y_{n}\} \in \zeta$. By \cite[Lemma 2.2.2]{BS07},  if $X$ is $\delta$-hyperbolic then for any choices of sequences  $\{x_{n}\} \in \xi$, $\{y_{n}\} \in \zeta$ we have
\begin{equation}\label{sequence approximation}
(\xi |\zeta)_{p}  \leq \liminf_{n \rightarrow \infty}(x_{n}|y_{n})_{p} \leq \limsup_{n \rightarrow \infty}(x_{n}|y_{n})_{p} \leq (\xi |\zeta)_{p} + 8\delta.
\end{equation}
We also have the $4\delta$-inequality for $\xi$, $\zeta$, $\omega \in \p X$ and $p \in X$,
\begin{equation}\label{boundary delta inequality}
(\xi |\omega)_{p} \geq \min \{(\xi | \zeta)_{p},(\zeta | \omega)_{p}\} - 4\delta. 
\end{equation} 
For $x \in X$, $\xi \in \p X$ the Gromov product is defined analogously as 
\[
(x |\xi)_{p} = \inf \liminf_{n \rightarrow \infty}(x|x_{n})_{p},
\] 
with the infimum taken over $\{x_{n}\} \in \xi$, and the analogues of \eqref{sequence approximation} and \eqref{boundary delta inequality} hold as well.

We next observe that geodesic triangles $\Delta$ with vertices  in $X \cup \p X$ are $10\delta$-thin, in the precise sense that if $u \in \Delta$ is any given point then there is a point $v \in \Delta$ satisfying $|uv| \leq 10\delta$ that does not belong to the same edge of $\Delta$ as $u$. When $X$ is proper this can be easily deduced from the $\delta$-thin triangles property for triangles in $X$ by a limiting argument. Without the properness hypothesis this result can also be obtained with a larger thinness constant $200\delta$ as a consequence of work of V\"ais\"al\"a \cite[Theorem 6.24]{V05}; we note that he uses \eqref{delta inequality} as the definition of $\delta$-hyperbolicity so we have to multiply the constant he obtains by $4$. As V\"ais\"al\"a works in the more general context of Gromov hyperbolic spaces that are not necessarily geodesic (which greatly complicates the proofs), we prefer to give a simpler direct proof of $10\delta$-thinness here. 

\begin{lem}\label{infinite thin}
Let $\Delta$ be a geodesic triangle in $X$ with vertices in $X \cup \p X$. Then $\Delta$ is $10\delta$-thin. 
\end{lem}

\begin{proof}
Let $x,y,z \in X \cup \p X$ be the vertices of $\Delta$. Let $u \in \Delta$ be given. Since $X$ has $\delta$-thin triangles, we may assume that $\Delta$ has at least one vertex on $\p X$. We first consider the case in which $\Delta$ has exactly one vertex on $\p X$, which by relabeling we can assume is $z$.  We first assume that $u \in xy$. Let $\{z_{n}\} \subset xz$ and $\{z_{n}'\} \subset yz$ be sequences such that $z_{n} \rightarrow z$ and $z_{n}' \rightarrow z$. For each $n$ we let $\Delta_{n} = xyz_{n}$ be the geodesic triangle sharing the edge $xy$ with $\Delta$, having a second edge be the subsegment $xz_{n}$ of $xz$, and having a third edge be any choice of geodesic $yz_{n}$. Then $\Delta_{n}$ is $\delta$-thin, so we have for each $n$ that either $\mathrm{dist}(u,xz_{n}) \leq \delta$ or $\mathrm{dist}(u,yz_{n}) \leq \delta$ (or both). In the first case we are done since $xz_{n} \subset xz$, so we can assume that $\mathrm{dist}(u,yz_{n}) \leq \delta$. Let $v_{n} \in yz_{n}$ be such that $|uv_{n}| \leq \delta$. Then $|v_{n}y| \leq \delta + |uy|$. 

Since both $z_{n}$ and $z_{n}'$ converge to $z$, for sufficiently large $n$ we will have $(z_{n}|z_{n}')_{y} \geq \delta + |uy|$, which implies in particular that $|z_{n}'y| \geq |v_{n}y|$. The $4\delta$-tripod condition applied to $y$, $z_{n}$, and $z_{n}'$ then implies that if $w_{n} \in yz_{n}'$ is the unique point such that $|yw_{n}| = |yv_{n}|$ then $|v_{n}w_{n}| \leq 4\delta$, from which it follows that $|uw_{n}| \leq 5\delta$ for all sufficiently large $n$. Since $w_{n} \in yz$ this completes the proof of this case. 

The other cases are $u \in xz$ and $u \in yz$. By symmetry it suffices to prove the case $u \in xz$. We define the sequences $\{z_{n}\}$ and $\{z_{n}'\}$ and the triangle $\Delta_{n}$ as before. As in the case $u \in xy$ we can assume that $\mathrm{dist}(u,yz_{n}) \leq \delta$ for all $n$, as otherwise by the $\delta$-thin triangles property we have $\mathrm{dist}(u,xy) \leq \delta$ and we are done. We let $v_{n} \in yz_{n}$ be such that $|uv_{n}| \leq \delta$, note that $|v_{n}y| \leq \delta + |uy|$ as before, and choose $n$ large enough that $(z_{n}|z_{n}')_{y} \geq \delta + |uy|$. As before the $4\delta$-tripod condition then supplies a point $w_{n} \in yz_{n}'$ such that $|w_{n}v_{n}| \leq 4\delta$ and we conclude that $\mathrm{dist}(u,yz) \leq 5\delta$. 

We can now handle the case in which potentially two or three vertices of $\Delta$ belong to $\p X$. By symmetry it suffices to show for a point $u \in xy$ that $u$ is $10\delta$-close to either $xz$ or $yz$. Let $\{x_{n}\} \subset xy$ and $\{y_{n}\} \subset xy$ be sequences such that $x_{n} \rightarrow x$ and $y_{n} \rightarrow y$; if $x \in X$ then we set $x_{n} = x$ for all $n$ and similarly if $y \in Y$ then we set $y_{n} = y$ for all $n$.  Let $\Delta_{n} = x_{n}y_{n}z$ be a geodesic triangle with one edge the subsegment $x_{n}y_{n}$ of $xy$. Then $\Delta_{n}$ has at most one vertex $z$ on $\p X$. We conclude from the previous case that $u$ is $5\delta$-close to either $x_{n}z$ or $y_{n}z$. By switching the roles of $x$ and $y$ if necessary, we can then assume that there is $v \in x_{n}z$ such that $|uv| \leq 5\delta$. If $x \in X$ then $x_{n} = x$ and we are done. Thus we can assume that $x \in \p X$.

Fix any point $w \in xz$ and let $x_{n}' \in wx$ be defined such that $|wx_{n}'| = |ux_{n}|$. Since the geodesic rays $wx$ and $ux$ define the same point $x$ of the Gromov boundary, there is a constant $c \geq 0$ such that $|x_{n}x_{n}'| \leq c$ for all $n$. We apply the previous case again to a triangle $\Delta_{n}' = x_{n}x_{n}'z$ with edges the segment $x_{n}'z$, the segment $x_{n}z$, and a choice of geodesic $x_{n}x_{n}'$, obtaining that $v$ is $5\delta$-close to either $x_{n}'z$ or $x_{n}x_{n}'$. If $v$ is $5\delta$-close to $x_{n}x_{n}'$ for all $n$ then 
\[
|x_{n}u| \leq |uv| + \mathrm{dist}(v,x_{n}x_{n}') + |x_{n}x_{n}'| \leq 10\delta + c, 
\]
contradicting that $|x_{n}u| \rightarrow \infty$ as $n \rightarrow \infty$. We conclude that $v$ is $5\delta$-close to $x_{n}'z \subset xz$ for all sufficiently large $n$, which implies that $\mathrm{dist}(u,xz) \leq 10\delta$ as desired.  
\end{proof}

We can also now formally define rough starlikeness from points of $X \cup \p X$. We recall that for $\omega \in \p X$ we write $\p_{\omega} X = \p X \backslash \{\omega\}$ for the Gromov boundary of $X$ relative to $\omega$. The definition is slightly different for points of $X$ and points of $\p X$, so we handle these two cases separately. 

\begin{defn}\label{def:rough star}
Let $X$ be a geodesic Gromov hyperbolic space. Let $z \in X$ and $K \geq 0$ be given. We say that $X$ is \emph{$K$-roughly starlike} from $z$ if 
\begin{enumerate}
\item For each $x \in X$ there is a geodesic ray $\gamma: [0,\infty) \rightarrow X$ such that $\gamma(0) = z$ and $\dist(x,\gamma) \leq K$.
\item For each $\xi \in \p X$ there is a geodesic ray $\gamma: [0,\infty) \rightarrow X$ such that $\gamma(0) = z$ and $\gamma \in \xi$.
\end{enumerate}

For a point $\omega \in \p X$ we say that $X$ is \emph{$K$-roughly starlike} from $\omega$ if 
\begin{enumerate}
\item For each $x \in X$ there is a geodesic line $\gamma: \R \rightarrow X$ such that $\dist(x,\gamma) \leq K$ and  $\gamma|_{(-\infty,0]}\in \omega$. 
\item For each $\xi \in \p_{\omega} X$ there is a geodesic line $\gamma: \R \rightarrow X$ such that $\gamma|_{[0,\infty)} \in \xi$ and $\gamma|_{(-\infty,0]} \in \omega$. 
\end{enumerate}
\end{defn} 

Part (2) of Definition \ref{def:rough star} implies in both cases that $\p^{g}X = \p X$, i.e., the geodesic boundary and the Gromov boundary coincide. It will be used as a replacement for the properness hypothesis in the main theorem of \cite{BHK}. We note that Property (2) of Definition \ref{def:rough star} automatically holds for any $\omega \in X \cup \p X$ when $X$ is proper, since in this case any two points of $X \cup \p X$ can be joined by a geodesic. We also remark that if $\p X$ consists of a single point $\omega$ then $X$ cannot be roughly starlike from $\omega$, since no geodesic line $\gamma: \R \rightarrow X$ can exist in this case. Similarly if $\p X$ is empty then $X$ cannot be roughly starlike from any of its points. 

\subsection{Busemann functions} 
In this section we closely follow \cite[Chapter 3]{BS07}. Throughout much of the paper we will need to work with Gromov products based at a point $\omega \in \p X$. These will be defined through the use of Busemann functions. In order to use the results from \cite[Chapter 3]{BS07} we have to show, for a geodesic ray $\gamma \in \omega$, that $b_{\gamma}$ is a Busemann function based at $\omega$ in their sense. The definition of a Busemann function given there starts with the function 
\begin{equation}\label{bs busemann}
b_{\omega,p}(x) = (\omega|p)_{x}-(\omega|x)_{p},
\end{equation}
for $x,p \in X$ and $\omega \in \p X$ and defines a Busemann function based at $\omega$ to be any function $b: X \rightarrow \R$ satisfying $b \doteq_{8\delta} b_{\omega,p}+s$ for some $p \in X$ and $s \in \R$ (recall that we are multiplying all of their constants by $4$ due to differing definitions of hyperbolicity). Note that this alternative definition \eqref{bs busemann} makes sense even for points in the Gromov boundary that do not belong to the geodesic boundary. 

\begin{lem}\label{equivalence busemann}
Let $\omega \in \p X$, let $p \in X$, and let $\gamma \in \omega$ be a geodesic ray with $\gamma(0) = p$. Then we have $b_{\omega,p} \doteq_{24\delta} b_{\gamma}$.
\end{lem}

\begin{proof}
By \cite[Example 3.1.4]{BS07} we have for all $x \in X$ that
\[
b_{\omega,p}(x) \doteq_{8\delta} |xp|-(\omega|x)_{p}. 
\]
By inequality \eqref{sequence approximation} we have $(\gamma(n)|x)_{p} \doteq_{8\delta} (\omega|x)_{p}$ for $n \in \N$ sufficiently large. Since $p = \gamma(0)$ we have
\[
(\gamma(n)|x)_{p} = \frac{1}{2}(n + |xp|-|\gamma(n)x|). 
\]
Then
\begin{equation}\label{final line equivalence}
|xp|-2(\omega|x)_{p} \doteq_{16\delta} |\gamma(n)x|-n.
\end{equation}
Since the right side converges to $b_{\gamma}(x)$ as $n \rightarrow \infty$, the result follows.  
\end{proof}

We recall our definition of a Busemann function from \eqref{extension busemann definition}: a Busemann function $b \in \mathcal{B}(X)$ is any function $b: X \rightarrow \R$ such that $b = b_{\gamma}+s$ for some geodesic ray $\gamma$ in $X$ and some $s \in \R$. For such a Busemann function $b$ we let $\omega = \omega_{b} = [\gamma]$ denote its basepoint in $\p X$. Then by Lemma \ref{equivalence busemann} $b$ is a Busemann function based at $\omega$ in the sense of \cite[Chapter 3]{BS07} as well, provided that we use a cutoff of $b \doteq_{24\delta} b_{\omega,p} + s$ instead of the $8\delta$-cutoff used there. This only has the effect of further multiplying constants by 3 in the claims of that chapter. An easy consequence of Lemma \ref{equivalence busemann} is the following. 

\begin{lem}\label{identify busemann}
Let $\omega \in \p X$ and let $\gamma,\sigma:[0,\infty) \rightarrow X$ be geodesic rays with $\gamma,\sigma \in \omega$. Then there is a constant $s \in \R$ such that $b_{\sigma} \doteq_{72\delta} b_{\gamma} + s$. The constant $s$ depends only on the starting points $\gamma(0)$ and $\sigma(0)$ of the rays and satisfies $s = 0$ if $\gamma(0) = \sigma(0)$. 

Consequently if $b$ is any Busemann function based at $\omega$ and $\sigma \in \omega$ is any geodesic ray then there is a constant $s \in \R$ such that $b \doteq_{72\delta} b_{\sigma}+s$. 
\end{lem}

\begin{proof}
By \cite[Lemma 3.1.2]{BS07}, for each $p,q, x \in X$ we have
\[
b_{\omega,p}(x) \doteq_{24\delta} b_{\omega,q}(x)+b_{\omega,q}(p). 
\]
Setting $p = \gamma(0)$, $q = \sigma(0)$, and applying Lemma \ref{equivalence busemann} gives
\[
b_{\gamma}(x) \doteq_{72\delta} b_{\sigma}(x) + b_{\omega,\sigma(0)}(\gamma(0)). 
\]
This gives the first claim of the lemma with $c = b_{\omega,\sigma(0)}(\gamma(0))$. The claim that $s = 0$ if $\gamma(0) = \sigma(0)$ follows from the fact that $b_{\omega,p}(p) = 0$ for any $p \in X$. The final claim follows immediately since for any Busemann function $b$ based at $\omega$ there is some geodesic ray $\gamma \in \omega$ such that $b = b_{\gamma}+s'$ for some $s' \in \R$. 
\end{proof}

We will usually use the following lemma to perform computations with Busemann functions in practice.  Note that the geodesics are parametrized as starting from the basepoint $\omega \in \p X$ instead of ending there. The notation $(-\infty,a]$ below should be interpreted as $(-\infty,a] = \R$ when $a = \infty$. 

\begin{lem}\label{geodesic busemann}
Let $b$ be a Busemann function on $X$ based at $\omega \in \p X$. Let $a \in \R \cup \{\infty\}$ and let $\gamma:(-\infty,a] \rightarrow X$ be a geodesic with $\gamma(t) \rightarrow \omega$ as $t \rightarrow -\infty$. 
\begin{enumerate}
\item For any $s, t \in (-\infty,a]$ (or any $s,t \in \R$ in the case $a = \infty$) we have
\begin{equation}\label{geodesic busemann equation}
b(\gamma(t)) - b(\gamma(s)) \doteq_{144\delta} t - s.
\end{equation}
\item For any constant $u \in \R$ there is an arclength reparametrization $\t{\gamma}: (-\infty,\t{a}] \rightarrow X$ of $\gamma$ such that $b(\t{\gamma}(t)) \doteq_{144\delta} t+u$ for $t \in (-\infty,\t{a}]$.
\end{enumerate} 
\end{lem}

\begin{proof}
Let $s \in (-\infty,a]$ be given and let $\sigma_{s}:[s-a,\infty) \rightarrow X$ be defined by $\sigma_{s}(t) = \gamma(s-t)$. It's easily checked from the definition \eqref{first busemann definition} that $b_{\sigma_{s}}(\sigma_{s}(t)) = -t$ for $t \in [s-a,\infty)$. Lemma \ref{identify busemann} shows that there is a constant $c \in \R$ such that $b \doteq_{72\delta} b_{\sigma_{s}} + c$. It follows that for any $t \in (-\infty,a]$, 
\[
b(\gamma(t))-b(\gamma(s)) \doteq_{144\delta} b_{\sigma_{s}}(\sigma_{s}(s-t))-b_{\sigma_{s}}(\sigma_{s}(0)) = t-s,
\]
for $t \in [-a,\infty)$. This proves (1).

For the second claim we fix an $s \in (-\infty,a)$ and define $\t{\gamma}(t) = \gamma(t-b(\gamma(s))+s+u)$ for $t \in (-\infty,\t{a}]$, $\t{a} = a-s+b(\gamma(s))-u$ (if $a = \infty$ we take $\t{a} = \infty$).  Then by \eqref{geodesic busemann equation},
\begin{align*}
b(\t{\gamma}(t)) &= b(\gamma(t-b(\gamma(s))+s+u)) \\
&\doteq_{144\delta} (t-b(\gamma(s))+s+u)-s+b(\gamma(s)) \\
&= t+u. 
\end{align*}
\end{proof}

For $x$, $y \in X$ and $b \in \hat{\mathcal{B}}(X)$ the Gromov product based at $b$ is defined by 
\[
(x|y)_{b} = \frac{1}{2}(b(x) + b(y) - |xy|). 
\]
Since $b$ is $1$-Lipschitz we have the useful inequality
\begin{equation}\label{both busemann}
(x|y)_{b} \leq \min\{b(x),b(y)\}. 
\end{equation}
For $b \in \mathcal{D}(X)$ this notion essentially reduces to the standard Gromov product: if $b(x) = |xp| +s$ for some $p \in X$ and $s \in \R$ then $(x|y)_{b} = (x|y)_{p} +s$. The analogues of all the results below then follow from the discussion in the previous section. We will thus focus on the case of Busemann functions $b \in \mathcal{B}(X)$.
 
Let $b \in \mathcal{B}(X)$ and let $\omega = \omega_{b}$ be its basepoint. The Gromov product based at $b$ is extended to $\p X$ by, for $(\xi,\zeta) \neq (\omega,\omega)$, 
\[
(\xi | \zeta)_{b} = \inf \liminf_{n \rightarrow \infty}(x_{n}|y_{n})_{b}
\]
with the infimum taken over $\{x_{n}\} \in \xi$, $\{y_{n}\} \in \zeta$ as before, and similarly for $x \in X$ and $\xi \in \p X$ we define
\[
(x| \xi)_{b} = \inf \liminf_{n \rightarrow \infty}(x | x_{n})_{b},
\]
with the infimum taken over $\{x_{n}\} \in \xi$. The next lemma extends the $4\delta$-inequality to Gromov products based at $b$. It follows from \cite[Lemma 3.2.4]{BS07}. Recall that we have multiplied their additive constants by a total of $12$ due to the differing definition of hyperbolicity and larger cutoff in defining Busemann functions; we then round up to $600\delta$ afterward. The corresponding additive constant in \cite[Lemma 3.2.4]{BS07} below is $44\delta$. 

\begin{lem}\label{busemann inequality}
Let $b$ be a Busemann function based at $\omega \in \p X$. Then 
\begin{enumerate}
\item For any $\xi$, $\zeta \in \p X \backslash \{\omega\}$ and any $\{x_{n}\} \in \xi$, $\{y_{n}\} \in \zeta$ we have 
\[
(\xi |\zeta)_{b} \leq \liminf_{n \rightarrow \infty}(x_{n}|y_{n})_{b} \leq \limsup_{n \rightarrow \infty}(x_{n}|y_{n})_{b} \leq (\xi |\zeta)_{b} + 600\delta,
\]
and the same holds if we replace $\zeta$ with $x \in X$.
\item For any $\xi,\zeta,\la \in X \cup \p_{\omega}X$ we have
\[
(\xi |\la)_{b} \geq \min \{(\xi | \zeta)_{b},(\zeta | \la)_{b}\} - 600\delta. 
\]
\end{enumerate}
\end{lem} 

Combining (1) of Lemma \ref{busemann inequality} with inequality \eqref{both busemann} gives for all $x,y \in X \cup \p X$ with $(x,y) \neq (\omega,\omega)$, 
\begin{equation}\label{both busemann boundary}
(x|y)_{b} \leq \min\{b(x),b(y)\}+600\delta, 
\end{equation}
where we set $b(\omega) = -\infty$ and $b(\xi) = \infty$ for $\xi \in \p_{\omega} X$.

For a point $\omega \in \p^{g} X$ belonging to the geodesic boundary, a sequence $\{x_{n}\}$ \emph{converges to infinity with respect to $\omega$} if for some Busemann function $b$ based at $\omega$ we have $(x_{m}|x_{n})_{b} \rightarrow \infty$ as $m,n \rightarrow \infty$. Two sequences $\{x_{n}\}$, $\{y_{n}\}$ converging to infinity with respect to $\omega$ are \emph{equivalent with respect to $\omega$} if $(x_{n}|y_{n})_{b} \rightarrow \infty$. These notions do not depend on the choice of Busemann function $b$ based at $\omega$ by Lemma \ref{identify busemann}. One then defines the \emph{Gromov boundary relative to $\omega$} as the set of all equivalence classes of sequences converging to infinity with respect to $\omega$. We will denote this by $\p_{\omega} X$. As our past use of the notation $\p_{\omega}X = \p X \backslash \{\omega\}$ suggests we have the following, which is \cite[Proposition 3.4.1]{BS07}. 

\begin{prop}\label{convergence Busemann}
A sequence  $\{x_{n}\}$ converges to infinity with respect to $\omega$ if and only if it converges to a point $\xi \in \p X \backslash \{\omega\}$. This correspondence defines a canonical identification of $\p_{\omega} X$ and $\p X \backslash \{\omega\}$.
\end{prop}

We recall that for $\omega \in X$ we will often abuse terminology and also refer to $\p_{\omega}X = \p X$ as the Gromov boundary relative to $\omega$.

\subsection{Visual metrics}\label{subsec:visual} Let $K \geq 1$ and let $Z$ be a set. A function $\theta: Z \times Z \rightarrow [0,\infty)$ is a \emph{$K$-quasi-metric} if the following holds for any $z,z',z'' \in Z$,
\begin{enumerate}
\item $\theta(z,z') = 0$ if and only if $z = z'$, 
\item $\theta(z,z')= \theta(z',z)$,
\item $\theta(z,z'') \leq K \max\{\theta(z,z'),\theta(z',z'')\}$. 
\end{enumerate} 
By a standard construction (see \cite[Lemma 2.2.5]{BS07}) a $K$-quasi-metric with $K \leq 2$ is always $4$-biLipschitz to a metric on $Z$. Since for $\e > 0$ we have that $\theta^{\e}$ is a $K^{\e}$ quasi-metric if $\theta$ is a $K$-quasi-metric, for any quasi-metric $\theta$ we always have that $\theta^{\e}$ is $4$-biLipschitz to a metric $d$ on $Z$ (by the identity map on $Z$) whenever $\e$ is small enough that $K^{\e} \leq 2$.

Let $X$ be a geodesic $\delta$-hyperbolic space. For $x \in X$ and $\e > 0$ we define for $\xi$, $\zeta \in \p X$,
\begin{equation}\label{visual base}
\theta_{\e,x}(\xi,\zeta) = e^{-\e (\xi|\zeta)_{x}},
\end{equation}
with the understanding that $e^{-\infty} = 0$. By \eqref{boundary delta inequality} the function $\theta_{\e,x}$ defines an $e^{8\delta \e}$-quasi-metric on $\p X$. We refer to any metric $\theta$ on $\p X$ that is $L$-biLipschitz to $\theta_{\e,x}$ as a \emph{visual metric} on $\p X$ based at $x$ with parameter $\e$; we call $L$ the \emph{comparison constant} to the model quasi-metric $\theta_{\e,x}$. A visual metric always exists once $\e$ is small enough that $e^{8\delta \e} \leq 2$. We give $\p X$ the topology induced by any visual metric. Equipped with a visual metric with respect to any basepoint $x \in X$ and any parameter $\e > 0$ the set $\p X$ is a complete bounded metric space. The basepoint change inequality \eqref{basepoint change} combined with inequality \eqref{sequence approximation} shows that the notion of a visual metric does not actually depend on the choice of basepoint $x \in X$, however the comparison constant to the quasi-metric \eqref{visual base} will depend on the basepoint. For $b \in \mathcal{D}(X)$ of the form $b(y) = d(x,y) + s$ for some $s \in \R$ we then define
\begin{equation}\label{D extension}
\theta_{\e,b}(\xi,\zeta) = e^{-\e (\xi|\zeta)_{b}} = e^{-\e s}\theta_{\e,x}(\xi,\zeta).
\end{equation}
 
Let $\omega \in \p^{g} X$ be a point of the geodesic boundary and let $b \in \mathcal{B}(X)$ be a Busemann function based at $\omega$. We define for $\e > 0$ and $\xi, \zeta \in \p_{\omega} X$, 
\begin{equation}\label{visual quasi}
\theta_{\e,b}(\xi,\zeta) = e^{-\e (\xi|\zeta)_{b}}.
\end{equation}
Then $\theta_{\e,b}$ defines an $e^{600\delta \e}$-quasi-metric on $\p_{\omega} X$ by Lemma \ref{busemann inequality}. A \emph{visual metric} based at $\omega$ with parameter $\e$ is defined to be any metric $\theta$ on $\p_{\omega} X$ that is $L$-biLipschitz to $\theta_{\e,b}$, and as before we will call $L$ the comparison constant to the model quasi-metric $\theta_{\e,b}$. Since all Busemann functions associated to $\omega$ differ from each other by a constant, up to a bounded error (by Lemma \ref{identify busemann}), the notion of a visual metric based at $\omega$ does not depend on the choice of Busemann function $b$ based at $\omega$. Equipped with any visual metric based at $\omega$ the metric space $\p_{\omega}X$ is complete. It is bounded if and only if $\omega$ is an isolated point in $\p X$. 


\begin{rem}\label{CAT visual}
For CAT$(-1)$ spaces the quasi-metric $\theta_{1,b}$ for $b \in \hat{\mathcal{B}}(X)$ with basepoint $\omega$ defines a distinguished visual metric on $\p_{\omega} X$ with parameter $\e = 1$. This metric is known as a \emph{Bourdon metric} when $b \in \mathcal{D}(X)$ and a \emph{Hamenst\"adt metric} when $b \in \mathcal{B}(X)$. The basic properties of the Bourdon metric for CAT$(-1)$ spaces were established by Bourdon in \cite{Bou95}. The Hamenst\"adt metric was introduced by Hamenst\"adt in the setting of Hadamard manifolds with sectional curvatures $\leq -1$ in \cite{Ham89} through a slightly different construction. The formulation for CAT$(-1)$ spaces using Gromov products based at $b$ is due to Foertsch-Schroeder \cite{FS11}. 
\end{rem}

\section{Tripod maps and Busemann functions}\label{sec:tripod}

In this section we let $X$ be a geodesic $\delta$-hyperbolic space for a given parameter $\delta \geq 0$. We will be establishing some standard claims regarding geodesic triangles in $X$ that have vertices on the Gromov boundary $\p X$. We will then use these claims regarding geodesic triangles in $X$ to evaluate  Busemann functions on geodesics in $X$ in Proposition \ref{compute Busemann} and Lemma \ref{star parametrize}.  When $X$ is proper these claims can be obtained via limiting arguments from the corresponding claims for geodesic triangles in \cite[Chapitre 2]{GdH90}. Without the properness hypothesis they may be obtained (with larger constants) by careful examination and specialization of the results of V\"ais\"al\"a on roads and biroads in $\delta$-hyperbolic space \cite[Section 6]{V05}.  We will provide more direct proofs of these results here, as we will also need to use some particular corollaries of the proofs that cannot be found in \cite{V05}. Providing our own proofs also allows us to organize the results in a manner that is convenient for our applications. 

\subsection{Tripod maps} We start with a definition. The terminology is taken from \cite[Chapter 2]{BS07}. Compare \cite[Chapitre 2, D\'efinition 18]{GdH90}. 

\begin{defn}\label{equiradial definition}
Let $\Delta$ be a geodesic triangle in $X$ with vertices $x,y,z \in X \cup \p X$ and let $\chi \geq 0$ be given. A collection of points $\hat{x} \in yz$, $\hat{y} \in xz$, $\hat{z} \in xy$ is \emph{$\chi$-equiradial} if 
\[
\mathrm{diam}\{\hat{x},\hat{y},\hat{z}\} = \max\{|\hat{x}\hat{y}|,|\hat{y}\hat{z}|,|\hat{x}\hat{z}|\} \leq \chi. 
\]
We then refer to $\hat{x}$, $\hat{y}$, $\hat{z}$ as \emph{$\chi$-equiradial points} for $\Delta$. 
\end{defn}

\begin{rem}\label{alternate}
For $x,y,z \in X$ Definition \ref{equiradial definition} makes sense in any geodesic metric space $X$. Taking $\chi = \delta$ gives yet another quantitatively equivalent definition of $\delta$-hyperbolicity for $X$. See \cite[Chapitre 2, Proposition 21]{GdH90}. 
\end{rem}

When $x,y,z \in X$, the $4\delta$-tripod condition directly provides us with a set of $4\delta$-equiradial points $\hat{x}$, $\hat{y}$, $\hat{z}$ defined by the system of equalities $|x\hat{y}| = |x\hat{z}| = (y|z)_{x}$, $|y\hat{x}| = |y\hat{z}| = (x|z)_{y}$, and $|z\hat{x}| = |z\hat{y}| = (x|y)_{z}$. We will often refer to these points as the \emph{canonical equiradial points} for $\Delta$, since they are uniquely determined. The following definition encodes a convenient hypothesis to make on equiradial points of a geodesic triangle $\Delta$ that partially generalizes the notion of canonical equiradial points to the case that some of the vertices of $\Delta$ belong to $\p X$. 

We adopt the notation convention for $x,y \in X \cup \p X$ that $|xy| = \infty$ if $x \neq y$ and one of $x$ or $y$ belongs to $\p X$ and $|xy| = 0$ if $x = y$.

\begin{defn}\label{calibrated}
Let $\Delta$ be a geodesic triangle in $X$ with vertices $x,y,z \in X \cup \p X$, let $\chi \geq 0$ be given, and let $(\hat{x},\hat{y},\hat{z})$ be a collection of $\chi$-equiradial points for $\Delta$. We say that this collection is \emph{calibrated} if we have $|\hat{x}z| = |\hat{y}z|$, $|\hat{y}x| = |\hat{z}x|$, and $|\hat{z}y| = |\hat{x}y|$.
\end{defn}

This condition is trivially satisfied when all vertices of $\Delta$ belong to $\p X$, since all of the subsegments involved have infinite length.

We let $\Upsilon$ be the tripod geodesic metric space composed of three copies $L_{1}$, $L_{2}$, and $L_{3}$ of the closed half-line $[0,\infty)$ identified at $0$.  This identification point will be denoted by $o$ and will be referred to as the \emph{core} of the tripod $\Upsilon$. The space $\Upsilon$ is clearly $0$-hyperbolic. The Gromov boundary $\p \Upsilon$ consists of three points $\zeta_{i}$, $i = 1,2,3$, corresponding to the half-lines $L_{i}$ thought of as geodesic rays starting from the core $o$. 

For a geodesic triangle $\Delta$ with a calibrated ordered triple of $\chi$-equiradial points $(\hat{x}, \hat{y}, \hat{z})$ as in Definitions \ref{equiradial definition} and \ref{calibrated}, we define the associated \emph{tripod map} $T: \Delta \rightarrow \Upsilon$ to be the map that sends the sides $xz$, $yz$, and $xy$ isometrically into $L_{1} \cup L_{3}$, $L_{2} \cup L_{3}$, and $L_{1} \cup L_{2}$ respectively in the unique way that satisfies $T(x) \in L_{1} \cup \{\zeta_{1}\}$, $T(y) \in L_{2} \cup \{\zeta_{2}\}$,$T(z) \in L_{3} \cup \{\zeta_{3}\}$, and $T(\hat{x}) = T(\hat{y}) = T(\hat{z}) = o$. To be more precise for boundary points, when $x \in \p X$ we mean here that $T(x) = \zeta_{1}$, i.e., $T$ maps the geodesic rays $\hat{y}x$ and $\hat{z}x$ isometrically onto $L_{1}$. A choice of ordering of the equiradial points is required to define the map $T$ but is not important, as changing the ordering simply corresponds to permuting the rays $L_{i}$ in $\Upsilon$ while keeping the core $o$ fixed. 


We first obtain the following direct consequence of Lemma \ref{infinite thin}.

\begin{lem}\label{infinite equiradial}
Let $\Delta$ be a geodesic triangle with vertices $x,y,z \in X \cup \p X$. Then there is a calibrated $60\delta$-equiradial collection of points $\hat{x} \in yz$, $\hat{y} \in xz$, $\hat{z} \in xy$. 
\end{lem}

\begin{proof}
If all vertices of $\Delta$ belong to $X$ then the canonical equiradial points give a calibrated $4\delta$-equiradial collection for $\Delta$, so we can assume that at least one vertex of $\Delta$ belongs to $\p X$. Thus we can assume without loss of generality that $z \in \p X$. 

Parametrize the side $xy$ by arclength as $\gamma: I \rightarrow X$ for an interval $I \subset \R$, oriented from $x$ to $y$. Let $E_{x} \subset I$ be the collection of times $t$ such that $\mathrm{dist}(\gamma(t),xz) \leq 10\delta$ and $E_{y} \subset I$ the collection of times $t$ such that $\mathrm{dist}(\gamma(t),yz) \leq 10\delta$. Each of the sets $E_{x}$ and $E_{y}$ are closed and we have $E_{x} \cup E_{y} = I$ by Lemma \ref{infinite thin}. We claim that both $E_{x}$ and $E_{y}$ are always nonempty. For this we can assume without loss of generality that $E_{x}$ is nonempty since $E_{x} \cup E_{y} = I$. 

If $E_{y} = \emptyset$ then $E_{x} = I$. For each $t \in I$ we let $x_{t} \in xz$ be a point such that $|x_{t}\gamma(t)| \leq 10\delta$. For $n \in \N$ the sequence $\{\gamma(n)\}$ converges to $y$, which implies that the sequence $\{x_{n}\}$ converges to $y$ since these sequences are a bounded distance from one another. However any sequence of points converging to infinity in $xz$ can only possibly converge to $x$ or $z$, which is a contradiction. Thus $E_{y}$ must also be nonempty. 

By the connectedness of $I$ we then conclude that $E_{x} \cap E_{y} \neq \emptyset$. Letting $s \in E_{x} \cap E_{y}$, setting $w:=\gamma(s)$, and selecting points $u \in xz$, $v \in yz$ such that $|wu| \leq 10\delta$ and $|wv| \leq 10\delta$, we conclude that $\{w,u,v\}$ is a $20\delta$-equiradial collection of points for $\Delta$. 

Lastly we need to produce a calibrated collection of equiradial points from the collection $\{w,u,v\}$. If all vertices of $\Delta$ belong to $\p X$ then the collection is trivially calibrated, so we can assume at least one vertex of $\Delta$ belongs to $X$. By relabeling the vertices we can then assume that either  $x \in X$ and $y \in \p X$ or $x \in X$ and $y \in X$. In both cases we can find $u' \in xz$ such that $|xu'| = |xw|$ since $|xz| = \infty$. Then 
\begin{equation}\label{u u}
|uu'| = ||xu|-|xu'|| = ||xu|-|xw|| \leq |uw| \leq 20\delta. 
\end{equation}
It follows that the collection $\{w,u',v\}$ is $40\delta$-equiradial. If $y \in \p X$ then this collection is also calibrated and we are done. 

If $y \in X$ then we repeat this argument again by using the fact that $|yz| = \infty$ to find $v' \in yz$ such that $|yv'| = |yw|$. The calculation \eqref{u u} then shows that $|vv'| \leq 20\delta$ as well. We can then conclude that the collection $\{w,u',v'\}$ is calibrated and $60\delta$-equiradial, as desired. 
\end{proof}

Our next goal will be to prove that the tripod map $T: \Delta \rightarrow \Upsilon$ associated to the calibrated collection of equiradial points produced by Lemma \ref{infinite equiradial} is roughly isometric. We will require the following simple lemma.  

\begin{lem}\label{simple}
Let $X$ be a metric space and let $x,y,z \in X$ with $|xz| \leq |yz|$. Suppose that we are given geodesics $xz$ and $yz$ joining $x$ to $z$ and $y$ to $z$ respectively. Let $u \in xz$, $v \in yz$ be given points that satisfy $|xu| = |yv|$ and let $w \in yz$ be the unique point satisfying $|wz| = |uz|$. Then $w \in vz$ and $|wv| \leq |xy|$. 
\end{lem}

\begin{proof}
The point $w$ must belong to the subsegment $vz$ of $yz$, as if $w \in yv$ and $w \neq v$ then 
\begin{align*}
|yz| &= |yv| + |wz| - |wv| \\
&= |xu| + |uz| - |wv| \\
 &= |xz|-|wv| \\
 & < |xz|,
\end{align*}
contradicting that $|yz| \geq |xz|$. Since $w \in vz$ we then have
\begin{align*}
|yz| &= |yv| + |vw| + |wz| \\
&= |xu| + |vw| + |uz| \\
&= |xz| + |vw|,
\end{align*}
which implies by the triangle inequality that $|vw| \leq |xy|$. 
\end{proof}

We now apply Lemma \ref{simple} to the setting of a $\delta$-hyperbolic space $X$.

\begin{lem}\label{infinite triangle}
Let $x,y \in X$, let $z \in X \cup \p X$, and let $\bar{x} \in xz$, $\bar{y} \in yz$ satisfy $|x\bar{x}| = |y\bar{y}|$. Then we have
\begin{equation}\label{combined inequality}
|\bar{x}\bar{y}| \leq 3|xy|+8\delta. 
\end{equation}
\end{lem}

\begin{proof}
Set $t = |x\bar{x}| = |y\bar{y}|$. If $t \leq |xy|$ then
\[
|\bar{x}\bar{y}| \leq |\bar{x}x| + |xy| + |\bar{y}y| \leq 3|xy|,
\]
which verifies inequality \eqref{combined inequality}. We can thus assume that $t > |xy|$.

We first assume that $z \in X$. We can then assume without loss of generality that $|xz|\leq |yz|$. We consider a geodesic triangle $\Delta = xyz$ with sides the given geodesics $xz$ and $yz$, as well as a geodesic $xy$ from $x$ to $y$. Let $w \in yz$ be the unique point such that $|wz| = |\bar{x}z|$. Lemma \ref{simple} shows that $w \in \bar{y}z$ and $|w\bar{y}| \leq |xy|$. 

Let $x' \in xz$ and $y' \in yz$ be the canonical equiradial points for $\Delta$ on these edges. These points must satisfy $\max\{|x'x|,|y'y|\} \leq |xy|$ since $xy$ is an edge of $\Delta$. The assumption $t > |xy|$ then implies that $\bar{x} \in x'z$ and $\bar{y} \in y'z$. Thus $w \in y'z$. The $4\delta$-tripod condition then implies that $|w\bar{x}| \leq 4\delta$, from which it follows that $|\bar{x}\bar{y}| \leq |xy| + 4\delta$. This proves \eqref{combined inequality} in this case. 

We now consider the case $z \in \p X$. For each $s \geq 0$  we define $x_{s} \in xz$, $y_{s} \in yz$ to be the points such that $|x x_{s}| = s$ and $|yy_{s}| = s$. Since the geodesics $xz$ and $yz$ have the same endpoint $z \in \p X$, we must have $(x_{s}|y_{s})_{x} \rightarrow \infty$ as $s \rightarrow \infty$ and the same for $(x_{s}|y_{s})_{y}$. We choose $s$ large enough that $(x_{s}|y_{s})_{x} \geq t$ and $(x_{s}|y_{s})_{y} \geq t$. We consider a geodesic triangle $\Delta_{1} = xx_{s}y_{s}$ with edges the subsegment $xx_{s}$ of the given geodesic $xz$ as well as geodesics $x_{s}y_{s}$ and $xy_{s}$, and a triangle $\Delta_{2} = xy y_{s}$ with edges the subsegment $yy_{s}$ of the given geodesic $yz$, the edge $xy_{s}$ of $\Delta_{1}$, and a geodesic $xy$. Then $\bar{x} \in xx_{s}$ and $\bar{y} \in yy_{s}$ by our choice of $s$.

Since $(x_{s}|y_{s})_{x} \geq t$, we must have $|xy_{s}| \geq t$. Therefore there is a unique point $w \in xy_{s}$ such that $|xw| = |x\bar{x}| = t$. The $4\delta$-tripod condition applied to the triangle $\Delta_{1}$ then implies that $|\bar{x}w| \leq 4\delta$. If $|xy_{s}| \leq |yy_{s}|$ then we let $u \in yy_{s}$ be the unique point such that $|uy_{s}| = |wy_{s}|$. By applying Lemma \ref{simple} we then conclude that $u \in \bar{y}y_{s}$ and $|u\bar{y}| \leq |xy|$. Since $xy$ is an edge of the triangle $\Delta_{2}$, the canonical equiradial points of this triangle on the edges $xy_{s}$ and $yy_{s}$ can be at most a distance $|xy| \leq t$ from the vertices $x$ and $y$ respectively. We thus conclude from the $4\delta$-tripod condition that $|uw| \leq 4\delta$. Combining these inequalities together gives
\begin{equation}\label{calculate}
|\bar{x}\bar{y}| \leq |\bar{x}w| + |wu| + |u\bar{y}| \leq |xy| + 8\delta,
\end{equation}
which proves \eqref{combined inequality}. The case $|xy_{s}| \geq |yy_{s}|$ is similar: we let $v \in xy_{s}$ be the point such that $|vy_{s}| = |\bar{y}y_{s}|$, apply Lemma \ref{simple} to obtain $|vw| \leq |xy|$ and $v \in wy_{s}$, then apply the $4\delta$-tripod condition to obtain $|v\bar{y}| \leq 4\delta$. This gives inequality \eqref{combined inequality} through the same calculation as \eqref{calculate}.
\end{proof}

\begin{rem}\label{sharper}
The proof of Lemma \ref{infinite triangle} shows that we have the sharper inequality $|\bar{x}\bar{y}| \leq |xy| + 8\delta$ when $|xy| > |x\bar{x}| = |y\bar{y}|$. 
\end{rem}

We will use Lemma \ref{infinite triangle} to show that the tripod map associated to a collection of calibrated equiradial points for a geodesic triangle $\Delta$ is roughly isometric. 

\begin{prop}\label{rough tripod}
Let $x,y,z \in X \cup \p X$ be given vertices of a geodesic triangle $\Delta$ in $X$. Let $\hat{x} \in yz$, $\hat{y} \in xz$, $\hat{z} \in xy$ be points such that $(\hat{x},\hat{y},\hat{z})$ is a calibrated ordered triple of $\chi$-equiradial points for $\Delta$ for a given $\chi \geq 0$. Let $T: \Delta \rightarrow \Upsilon$ be the tripod map associated to this triple. Then $T$ is $(6\chi+16\delta)$-roughly isometric. 

In particular if $(\hat{x},\hat{y},\hat{z})$ is the calibrated $60\delta$-equiradial triple produced in Lemma \ref{infinite equiradial} then $T$ is $400\delta$-roughly isometric.
\end{prop}

\begin{proof}
By symmetry (permuting the vertices $x$, $y$, and $z$), to estimate $|T(p)T(q)|$ for $p,q \in \Delta$ it suffices to restrict to the case $p \in \hat{y}z$ and then consider the possible locations of $q$. By construction we have $|T(p)T(q)| = |pq|$ if $p$ and $q$ belong to the same edge of $\Delta$, since the tripod map is isometric on the edges of $\Delta$. This handles the case that $q$ belongs to the same edge as $p$, i.e., that $q \in xz$. 

We next consider the case $q \in yz$. Since $|\hat{y}z| = |\hat{x}z|$, we can find a point $u \in \hat{x}z$ such that $|\hat{x}u| = |\hat{y}p|$. Then $|T(p)T(q)| = |uq|$. Applying Lemma \ref{infinite triangle} yields
\[
|up| \leq 3|\hat{y}\hat{x}| + 8\delta \leq 3\chi + 8\delta,
\]
so that 
\[
||uq|-|pq|| \leq |up| \leq 3\chi + 8\delta,
\]
which gives the desired estimate in this case. 

Lastly we must consider the case $q \in xy$, which we subdivide into the cases $q \in x\hat{z}$ and $q \in \hat{z}y$. When $q \in x\hat{z}$ we can use the condition $|x\hat{z}| = |x\hat{y}|$ to find a point $v \in x\hat{y}$ such that $|q\hat{z}| = |v\hat{y}|$. Then $|T(p)T(q)| = |vp|$. Similarly to the previous case, Lemma \ref{infinite triangle} gives us the estimate $|vq| \leq 3\chi + 8\delta$ which implies that
\[
||vp|-|pq|| \leq |vq| \leq 3\chi + 8\delta,
\]
as desired. When $q \in \hat{z}y$ we use the equality $|\hat{z}y| = |\hat{x}y|$ to find $w \in \hat{x}y$ such that $|\hat{z}q| = |\hat{x}w|$, and we use the equality $|\hat{x}z| = |\hat{y}z|$ to find $s \in \hat{x}z$ such that $|s\hat{x}| = |p\hat{y}|$. Then $|T(p)T(q)| = |sw|$. Lemma \ref{infinite triangle} gives us the estimate
\[
\max\{|sp|,|wq|\} \leq 3\chi + 8\delta,
\]
which implies by the triangle inequality that
\[
||sw|-|pq|| \leq ||sw|-|wp|| + ||wp|-|qp|| \leq |sp| + |wq| \leq 6\chi + 16\delta. 
\]
This completes the proof of the main claim. The final assertion follows by substituting $\chi = 60\delta$ and rounding up. 
\end{proof}

\begin{rem}\label{shorthand}
Throughout this paper we will often suppress the exact choice of calibrated equiradial points used to define a tripod map $T: \Delta \rightarrow \Upsilon$. To make this more formal, for a geodesic triangle $\Delta$ in $X$ we will refer to a tripod map $T: \Delta \rightarrow \Upsilon$ associated to $\Delta$  as being any tripod map $T$ for $\Delta$ associated to an ordered triple $(\hat{x},\hat{y},\hat{z})$ of calibrated $60\delta$-equiradial points for $\Delta$ obtained from Lemma \ref{infinite equiradial}. We will also abuse terminology and say that $\hat{x}$, $\hat{y}$ and $\hat{z}$ are equiradial points for the tripod map $T$ (as opposed to for the triangle $\Delta$). 
\end{rem}

\subsection{Calculating Busemann functions} We recall that the Gromov boundary $\p \Upsilon$ of the tripod $\Upsilon$ is a disjoint union of three points $\zeta_{i}$, $i=1,2,3$, corresponding to the geodesic rays $\gamma_{i}:[0,\infty) \rightarrow \Upsilon$ that parametrize the half-lines $L_{i}$ starting from $o$ for $i =1,2,3$. Set $b_{\Upsilon}:=b_{\gamma_{1}}$ to be the Busemann function associated to the geodesic ray $\gamma_{1}$. A straightforward calculation shows that $b_{\Upsilon}$ is given by $b_{\Upsilon}(s) = -s$ for $s \in L_{1}$ and $b_{\Upsilon}(s) = s$ for $s \in L_{2}$ or $s \in L_{3}$, when we consider each of these rays as identified with $[0,\infty)$. 

In this next proposition we consider a geodesic triangle $\Delta$ in $X$ with a distinguished vertex $\omega \in \p X$ together with a Busemann function $b$ based at $\omega$. We will not keep track of exact constants in the proof of this lemma so we will not produce an explicit value for $\kappa = \kappa(\delta)$ below. If one does careful bookkeeping in the proof it is possible to show that $\kappa = 2000\delta$ works. 

\begin{prop}\label{compute Busemann}
Let $\Delta = \omega xy$ be a geodesic triangle in $X$ with $\omega \in \p X$ and $x,y\in X \cup \p_{\omega} X$. There is a constant $\kappa = \kappa(\delta)$ such that the following holds: let $\hat{\omega} \in xy$, $\hat{x} \in \omega y$, and $\hat{y} \in \omega x$ be a calibrated set of $60\delta$-equiradial points on $\Delta$ provided by Lemma \ref{infinite equiradial} and let $b$ be a Busemann function based at $\omega$. Let $T: \Delta \rightarrow \Upsilon$ be the tripod map associated to the triple $(\hat{\omega},\hat{x},\hat{y})$. Then for each $p \in \Delta$ we have
\begin{equation}\label{tripod image}
b(p) \doteq_{\kappa} b_{\Upsilon}(T(p)) + (x|y)_{b}. 
\end{equation}
Consequently we have $b(p) \doteq_{\kappa} (x|y)_{b}$ for $p \in \{\hat{\omega},\hat{x},\hat{y}\}$ and 
\begin{equation}\label{tripod minimum}
(x|y)_{b} \doteq_{\kappa} \inf_{p \in xy} b(p). 
\end{equation}
\end{prop}

\begin{proof}
Since we will not be keeping track of the exact value of the final constant $\kappa = \kappa(\delta)$ in the proof, we will let $\doteq$ denote any equality up to an additive error depending only on $\delta$. Set $u = b(\hat{\omega})$. Then $u \doteq b(\hat{x})$ and $u \doteq b(\hat{y})$ since $b$ is 1-Lipschitz. We will prove the rough equality \eqref{tripod image} with $u$ in place of $(x|y)_{b}$ and use this to deduce that $u \doteq (x|y)_{b}$. Thus we will first show that for $p \in \Delta$ we have
\begin{equation}\label{modified tripod}
b(p) \doteq b_{\Upsilon}(T(p)) + u. 
\end{equation}

We first handle the case in which $p \in \omega x$ or $p \in \omega y$. Since the roles of $x$ and $y$ are symmetric, we can assume without loss of generality that $p \in \omega x$. Let $\gamma: (-\infty,a] \rightarrow X$ be an arclength parametrization of $\omega x$ with $\gamma(0) = \hat{x}$ and $\gamma(t) \rightarrow \omega$ as $t \rightarrow -\infty$. If we define $s \in (-\infty,a]$ such that $\gamma(s) = p$ then it follows from the construction of the tripod map that $b_{\Upsilon}(T(\gamma(s)) = s$. Applying Lemma \ref{geodesic busemann} gives
\[
b(p) - b(\hat{x}) \doteq s = b_{\Upsilon}(T(\gamma(s)),
\]
which gives \eqref{modified tripod} since $b(\hat{x}) \doteq u$. 

The remaining case is when $p \in xy$. By the symmetric roles of $x$ and $y$ we can assume that $p \in x\hat{\omega}$. As in the proof of Proposition \ref{rough tripod}, since $|x\hat{\omega}| = |x\hat{y}|$ we can find $q \in x\hat{y}$ such that $|p\hat{\omega}| = |q\hat{y}|$. Then by Lemma \ref{infinite triangle} we have
\[
|pq| \leq 3|\hat{y}\hat{\omega}| + 10\delta \leq 190\delta. 
\]
Thus $b(p) \doteq b(q)$ since $b$ is 1-Lipschitz. It then follows, from the rough equality \eqref{modified tripod} for $q \in \omega x$ that we established above, that 
\[
b(p) \doteq b(q) \doteq |q\hat{y}| + u = |p\hat{\omega}| + u,
\]
which gives \eqref{modified tripod} in this case.  

We next show that $u \doteq (x|y)_{b}$. By Lemma \ref{busemann inequality} we have for any sequences $x_{n} \rightarrow x$ and $y_{n} \rightarrow y$ that $(x_{n}|y_{n})_{b} \doteq_{600\delta} (x|y)_{b}$ for sufficiently large $n$; if $x \in X$ then we can just set $x_{n} = x$ for all $n$ and the same goes for $y$. We choose sequences $\{x_{n}\}$ and $\{y_{n}\}$ that belong to $xy$ and consider only those $n$ large enough that $(x_{n}|y_{n})_{b} \doteq_{600\delta} (x|y)_{b}$ and $x_{n} \in \hat{\omega} x$, $y_{n} \in \hat{\omega} y$. Then applying \eqref{modified tripod}, 
\begin{align*}
(x|y)_{b}&\doteq (x_{n}|y_{n})_{b} \\
&= \frac{1}{2}(b(x_{n}) + b(y_{n}) - |x_{n}y_{n}|) \\
&\doteq  \frac{1}{2}(|x_{n}\hat{\omega}|+ |y_{n}\hat{\omega}| + 2u - |x_{n}y_{n}|) \\
&= u. 
\end{align*}
Thus we can substitute in $(x|y)_{b}$ for $u$ in \eqref{modified tripod} at the cost of an additional additive constant depending only on $\delta$. The main claim \eqref{tripod image} follows. The assertion that $b(p) \doteq_{\kappa} (x|y)_{b}$ for $p \in \{\hat{\omega},\hat{x},\hat{y}\}$ follows from \eqref{tripod image} since each point of $\{\hat{\omega},\hat{x},\hat{y}\}$ has image $o \in \Upsilon$ under the tripod map $T$ and $b_{\Upsilon}(o) = 0$. The rough equality \eqref{tripod minimum} also follows directly from \eqref{tripod image} since the image of $xy$ under $T$ is contained in $L_{2} \cup L_{3}$ and $b_{\Upsilon}$ is nonnegative on this subset of $\Upsilon$.
\end{proof}

Proposition \ref{compute Busemann} leads to the following important definition, which is useful for calculations. We recall our convention that $\p_{\omega} X = \p X$ when $\omega \in X$. 

\begin{defn}\label{adapted busemann}
Let $X$ be a geodesic $\delta$-hyperbolic space and let $b \in \hat{\mathcal{B}}(X)$ be given with basepoint $\omega$. Let $x,y \in X \cup \p_{\omega} X$ and let $c \geq 0$ be a given constant. Suppose that $xy$ is a geodesic joining $x$ to $y$. We say that a parametrization $\gamma:I \rightarrow X$, $I \subseteq \R$, of $xy$ by arclength is \emph{$c$-adapted to $b$} if $0 \in I$ and
\begin{equation}\label{adapted equation}
b(\gamma(t)) \doteq_{c} |t| + (x|y)_{b}, 
\end{equation}
for $t \in I$.  
\end{defn}

The inclusion of $0$ in the domain of $\gamma$ will be vital for our applications. When the value of $c$ is implied by context we will often shorten the terminology to just saying that the parametrization $\gamma$ is adapted to $b$. 

For $b \in \hat{\mathcal{B}}(X)$ with basepoint $\omega$ we will construct adapted parametrizations for geodesics joining any two points $x,y \in X \cup \p_{\omega} X$ under an assumption similar to the rough starlikeness hypothesis of Theorem \ref{unbounded uniformization}. We emphasize that the points $x$ and $y$ in the lemma need not always be the vertices of a geodesic triangle $\Delta$ with a third vertex at $\omega$. 

\begin{lem}\label{star parametrize}
Let $X$ be a geodesic $\delta$-hyperbolic space and let $b \in \hat{\mathcal{B}}(X)$ be given with basepoint $\omega$. Let $x,y \in X \cup \p_{\omega} X$ and let $xy$ be a given geodesic from $x$ to $y$. Suppose that we are given $K \geq 0$ and points $x',y' \in X \cup \p_{\omega}X$ and geodesics $\omega x',\omega y'$ joining $\omega$ to $x'$ and $y'$ respectively such that $\max\{|xx'|,|yy'|\} \leq K$. Then there is a constant $c = c(\delta,K)$ depending only on $\delta$ and $K$ such that there is a parametrization $\gamma: I \rightarrow X$ of $xy$ that is $c$-adapted to $b$.
\end{lem}

\begin{proof}
We first consider the case that $x = x'$ and $y = y'$, so that we can take $K = 0$. Let $\Delta = \omega x y$ be the geodesic triangle formed by the geodesics $\omega x$, $\omega y$, and $xy$. Let $T: \Delta \rightarrow \Upsilon$ be a $400\delta$-roughly isometric tripod map associated to $\Delta$ such that $T(\omega) \in L_{1} \cup\{\zeta_{1}\}$, $T(x) \in L_{2} \cup \{\zeta_{2}\}$, and $T(y) \in L_{3} \cup \{\zeta_{3}\}$, as given by Proposition \ref{rough tripod}. We identify the union $L_{2} \cup L_{3}$ of geodesic rays in $\Upsilon$ with $\R$ by identifying $L_{2}$ with $(-\infty,0]$ and $L_{3}$ with $[0,\infty)$, sending the core $o$ of $\Upsilon$ to the origin in $\R$. We let $I = T(xy) \subset L_{2} \cup L_{3}$ denote the image of $xy$ under $T$ and consider $I$ as a subinterval $I \subset \R$ under the identification of $L_{2} \cup L_{3}$ with $\R$. Since the tripod map $T$ is isometric when restricted to $xy$, we can then construct an arclength parametrization $\gamma: I \rightarrow X$ of $xy$ by inverting the restriction of $T$ to $xy$. By the construction of $T$ we have $0 \in I$ since the core $o$ is contained in the image $T(xy)$ of $xy$. 

When $b \in \mathcal{B}(X)$ the rough equality \eqref{tripod image} directly implies the $c$-adapted condition \eqref{adapted busemann} for $\gamma$ with $c = c(\delta)$, since for $z \in L_{2} \cup L_{3}$ we have $b_{\Upsilon}(z) = |zo|$. For $b \in \mathcal{D}(X)$ it is easy to see that it suffices to verify \eqref{adapted equation} for $b$ of the form $b(x) = |x\omega|$, $\omega \in X$. We then have to show that there is a constant $c = c(\delta)$ such that for $t \in I$ we have
\begin{equation}\label{distance adapated}
d(\gamma(t),\omega) \doteq_{c} |t| + (x|y)_{\omega}. 
\end{equation}
By \eqref{sequence approximation} we can find points $\bar{x} \in \omega x$, $\bar{y} \in \omega y$ such that $(\bar{x}|\bar{y})_{\omega} \doteq_{c(\delta)} (x|y)_{\omega}$ and $T(\bar{x}) \in L_{2}$, $T(\bar{y}) \in L_{3}$. Then $(T(\bar{x})|T(\bar{y}))_{T(\omega)} \doteq_{c(\delta)} (\bar{x}|\bar{y})_{\omega}$ since $T$ is $c(\delta)$-roughly isometric. Since $T(\bar{y}) \in L_{2}$ and $T(\bar{z}) \in L_{3}$, a quick calculation then shows that 
\[
(T(\bar{x})|T(\bar{y}))_{\omega} = |T(\omega)o|,
\] 
and therefore $(x|y)_{\omega} \doteq_{c(\delta)} |T(\omega)o|$. Thus for $t \in I$ we have
\[
|t| + (x|y)_{\omega} \doteq_{c(\delta)} |t| + |T(\omega)o| = |T(\omega)T(\gamma(t))|, 
\]
with the second equality following from the construction of $\gamma$. The rough equality \eqref{distance adapated} with $c = c(\delta)$ then directly follows from the fact that $T$ is $c(\delta)$-roughly isometric. 

We now consider the general case in which we are given points $x',y' \in X \cup \p_{\omega}X$ and $K \geq 0$ such that $\max\{|xx'|,|yy'|\} \leq K$. If $x$ and $y$ both belong to $\p_{\omega} X$ then our conventions imply that $x = x'$ and $y = y'$, hence this case reduces to the case $K = 0$ considered previously.  

If $x$ and $y$ both belong to $X$ then we apply the $K = 0$ case to the points $x'$ and $y'$ to obtain a $c'$-adapted parametrization $\sigma: I' \rightarrow X$ of $x'y'$ oriented from $x'$ to $y'$, $I' = [t_{-}',t_{+}']$ with $c' = c'(\delta)$. Since $0 \in I'$ we have $t_{-}' \leq 0$ and $t_{+}' \geq 0$. Let $\eta: I \rightarrow X$, $I = [t_{-}',t_{+}]$, be the unique arclength parametrization of $xy$ that is oriented from $x$ to $y$ and starts from the same time parameter $t_{-}'$ as $\sigma$. The piecewise geodesic curve $xx' \cup x'y' \cup yy'$ joining $x$ to $y$ can be parametrized as a $4K$-roughly isometric map $\beta: J \rightarrow X$ for an appropriate interval $J \subset \R$. By the stability of geodesics in Gromov hyperbolic spaces \cite[Theorem 1.3.2]{BS07} this implies that there is a constant $c_{0} = c_{0}(\delta,K)$ such that the given geodesic $xy$ is contained in a $c_{0}$-neighborhood of the curve $\beta$. Since the segments $xx'$ and $yy'$ of $\beta$ are each contained in a $K$-neighborhood of $\sigma$, by increasing $c_{0}$ by an amount depending only on $K$ we can assume that $xy$ is contained in a $c_{0}$-neighborhood of $\sigma$.

Now let $t \in I$ be given and let $s \in I'$ be such that $|\eta(t)\sigma(s)| \leq c_{0}$. Since $b$ is 1-Lipschitz it follows that 
\[
b(\eta(t)) \doteq_{c_{0}} b(\sigma(s)) \doteq_{c'} |s|+(x|y)_{b}. 
\] 
Thus it suffices to show that $t \doteq_{c''} s$ for a constant $c'' = c''(\delta,K)$. Since $t-t_{-}' = |\eta(t)x|$ and $s-t_{-}' = |\sigma(s)x'|$, we have
\begin{align*}
|t-s| &= |(t-t_{-}')-(s-t_{-}')| \\
&= ||\eta(t)x|-|\sigma(s)x'|| \\
&\leq ||\eta(t)x|-|\eta(t)x'|| + ||\eta(t)x'|-|\sigma(s)x'|| \\
&\leq |xx'| + |\eta(t)\sigma(s)| \\
&\leq K+ c_{0}, 
\end{align*}
so that we can set $c'' = K+c_{0}$. It follows that $\eta$ satisfies \eqref{adapted equation} with constant $c = c(\delta,K)$ depending only on $\delta$ and $K$.

If $0 \in I$ then $\eta$ gives a parametrization of $xy$ that is $c$-adapted to $b$ and we are done. We can therefore assume that $0 \notin I$ which implies that $t_{+} < 0$ since $t_{-}' \leq 0$. We then note that $|x'y'| \doteq_{2K} |xy|$ and $t_{+}'-t_{-}' = |x'y'|$, $t_{+}-t_{-}' = |xy|$, which implies that $t_{+} \doteq_{2K} t_{+}'$. Since $t_{+}' \geq 0$ and $t_{+} \leq 0$, we conclude that $|t_{+}| \leq 2K$. We set $I'' = [t_{-}'-t_{+},0]$ and set $\gamma(t) = \eta(t+t_{+})$ for $t \in I''$. Then $0 \in I''$ by construction and this arclength parametrization $\gamma$ of $xy$ still satisfies \eqref{adapted equation} with $c=  c(\delta,K)$ since $b$ is $1$-Lipschitz and $|t_{+}| \leq 2K$. Thus $\gamma$ gives the desired adapted parametrization. 

Lastly we consider the case in which one of $x$ or $y$ belong to $\p_{\omega} X$, but not both. Without loss of generality we can assume that $x \in X$ and $y \in \p_{\omega} X$. Let $\{y_{n}\} \subset xy$ be the sequence of points with $|xy_{n}| = n$ for each $n \in \N$. Let $\eta_{n}: I_{n} \rightarrow X$ be the arclength parametrizations of $xy_{n}$ for each $n$ that were constructed in the previous case, $I_{n} = [s_{n},t_{n}]$. Since $0 \in I_{n}$ for each $n$ we have $s_{n} \leq 0$ for each $n$. Since $\eta_{n}(s_{n}) = x$ for each $n$, we have from the condition that $\eta_{n}$ is $c$-adapted to $b$,
\[
b(x) \doteq_{c} |s_{n}|+(x|y)_{b} = -s_{n} + (x|y)_{b} 
\]
with $c = c(\delta,K)$. It follows that $s_{m} \doteq_{c} s_{n}$ for each $m, n \in \N$.  Thus, by replacing $\eta_{n}$ with the parametrization $\gamma_{n}$ defined by $\gamma_{n}(t) = \eta_{n}(t-s_{1}+s_{n})$ on the domain $I_{n}' = [s_{1},t_{n}+s_{1}-s_{n}]$, we can assume that $s_{n} = s_{1}:= s$ for all $n \in \N$. Note also that, since $t_{n} \rightarrow \infty$ as $n \rightarrow \infty$ and $s \leq 0$, we have $0 \in I_{n}$ for all large enough $n$. It follows that the resulting parametrization $\gamma_{n}$ will be $c$-adapted to $b$ for $n$ large enough that $t_{n} \geq 0$ since $b$ is 1-Lipschitz, with $c = c(\delta,K)$. 

With these modifications the parametrizations $\gamma_{n}$ now have the same starting point $s \leq 0$. Since these are parametrizations of $xy_{n}$ by arclength and the sequence $\{y_{n}\}$ defines progressively longer subsegments $xy_{n}$ of $xy$ that exhaust $xy$, the maps $\gamma_{n}$ coincide wherever their domains overlap and can therefore be used to define a parametrization $\gamma: [s,\infty) \rightarrow X$ of $xy$ that is $c$-adapted to $b$ by construction.  
\end{proof}

\section{Uniformization}\label{sec:uniformize}

Our task in this section will be to prove Theorems \ref{unbounded uniformization}, \ref{identification theorem}, and \ref{CAT theorem}. Section \ref{subsec:estimate uniform} establishes some general estimates for the uniformized distances $d_{\e,b}$. Section \ref{subsec:busemann uniformize} proves the theorems in the case that $b \in \mathcal{B}(X)$ (i.e., $b$ is a Busemann function). Section \ref{subsec:distance uniformize} then uses a special construction (Definition \ref{ray augment}) to deduce the case $b \in \mathcal{D}(X)$ from the case $b \in \mathcal{B}(X)$. Since Theorem \ref{CAT theorem} follows from Theorems \ref{unbounded uniformization} and \ref{identification theorem} once we've shown that $\rho_{1,b}$ is a GH-density in this case in  Proposition \ref{strong hyperbolic admissible}, we will focus our efforts on proving Theorems \ref{unbounded uniformization} and \ref{identification theorem} after that point. 

\subsection{Estimates for the uniformized distance}\label{subsec:estimate uniform} In this section we will derive some estimates for the conformal deformation  $(X_{\e,b},d_{\e,b})$ of a geodesic $\delta$-hyperbolic space $X$ by the density $\rho_{\e,b}(x) = e^{-\e b(x)}$ for a given $\e > 0$ and $b \in \hat{\mathcal{B}}(X)$ using the tripod maps we built in Section \ref{sec:tripod}. For now we will not be assuming that $\rho_{\e,b}$ is a GH-density (using the terminology of Definition \ref{conformal factor}). Hence we can use these results to establish that $\rho_{\e,b}$ is a GH-density in certain important cases. To simplify notation we will drop the function $b$ from the notation for objects associated to the conformal deformation and write $\rho_{\e}:=\rho_{\e,b}$, $X_{\e} :=X_{\e,b}$, etc. For a curve $\gamma: I \rightarrow X_{\e}$ we will write $\l_{\e}(\gamma):=\l_{\rho_{\e}}(\gamma)$ for its length measured in the metric $d_{\e}$. We let $\l(\gamma)$ denote the length of $\gamma$ measured in $X$ instead. 

\begin{rem}\label{appendix interlude} Throughout the rest of this paper we will be using \cite[Proposition A.7]{BHK}, which for a geodesic metric space $X$ and a continuous function $\rho:X \rightarrow (0,\infty)$  allows us to compute the lengths $\l_{\rho}(\gamma)$ in the conformal deformation $X_{\rho}$ of curves $\gamma: I \rightarrow X$ parametrized by arclength in $X$ as
\begin{equation}\label{appendix prop}
\l_{\rho}(\gamma) = \int_{I}\rho \circ \gamma \, ds,
\end{equation}
with $ds$ denoting the standard length element in $\R$.
\end{rem}

Since $b$ is $1$-Lipschitz we have the \emph{Harnack type inequality} for $x,y \in X$,
\begin{equation}\label{Harnack}
e^{-\e|xy|} \leq \frac{\rho_{\e}(x)}{\rho_{\e}(y)} \leq e^{\e|xy|},
\end{equation}
which made its first appearance in the statement of Theorem \ref{Gehring-Hayman} earlier.  

The metric spaces $X_{\e}$ and $X$ are biLipschitz on bounded subsets of $X$ by inequality \eqref{Harnack}. A more precise estimate for this is given in the lemma below.

\begin{lem}\label{arc Harnack}
For any $x,y \in X$ we have
\[
\rho_{\e}(x)\e^{-1}(1-e^{-\e |xy|}) \leq d_{\e}(x,y) \leq \rho_{\e}(x)\e^{-1}(e^{\e |xy|}-1).
\]
\end{lem}

\begin{proof}
For the upper bound we let $xy$ be a geodesic joining $x$ to $y$. Then, using \eqref{Harnack}, 
\begin{align*}
d_{\e}(x,y) &\leq \int_{xy} \rho_{e}\, ds \\
&\leq \rho_{\e}(x)\int_{0}^{|xy|} e^{\e t}\, dt \\
&= \rho_{\e}(x)\e^{-1}(e^{\e |xy|}-1).
\end{align*}
For the lower bound we consider a rectifiable curve $\gamma$ joining $x$ to $y$ in $X$, which we can assume is parametrized by arclength as $\gamma:[0,\l(\gamma)] \rightarrow X$ with $\l(\gamma)$ denoting the length of $\gamma$ in $X$.  With this parametrization $\gamma$ defines a $1$-Lipschitz function from $[0,\l(\gamma)]$ to $X$, so that in particular we have $|x\gamma(t)| \leq t$ for each $t \in [0,\l(\gamma)]$. Then by \eqref{Harnack},
\begin{align*}
\l_{\e}(\gamma) &\geq \rho_{\e}(x)\int_{0}^{\l(\gamma)} e^{-\e |x\gamma(t)|}\, dt \\
&\geq \rho_{\e}(x)\int_{0}^{\l(\gamma)} e^{-\e t}\, dt \\
&= \rho_{\e}(x)\e^{-1}(1-e^{-\e \l(\gamma)}) \\
&\geq \rho_{\e}(x)\e^{-1}(1-e^{-\e |xy|}),
\end{align*}
where in the final line we used that $\l(\gamma) \geq |xy|$. 
\end{proof}

Lemma \ref{arc Harnack} can be rewritten in the following useful form when $|xy| \leq 1$. 

\begin{lem}\label{rephrased arc Harnack}
For any $x,y \in X$ with $|xy| \leq 1$ we have
\begin{equation}\label{rephrased inequality}
d_{\e}(x,y) \asymp_{C(\e)} e^{-\e(x|y)_{b}}|xy|.
\end{equation}
\end{lem}

\begin{proof}
For $0 \leq t \leq 1$ we have the inequalities 
\[
1-e^{-\e t} \geq \e e^{-\e}t,
\]
and 
\[
e^{\e t}-1 \leq \e e^{\e} t,
\]
as can be verified by noting that equality holds at $t = 0$ and differentiating each side. Thus for $|xy| \leq 1$ the inequality of Lemma \ref{arc Harnack} implies that
\begin{equation}\label{unbounded inequality}
d_{\e}(x,y) \asymp_{C(\e)} \rho_{\e}(x)|xy| \asymp_{C(\e)} e^{-\e(x|y)_{b}}|xy|,
\end{equation}
with the final estimate following from
\[
\rho_{\e}(x) = e^{-\e b(x)} \asymp_{e^{\e}} e^{-\e(x|y)_{b}},
\]
since $|xy| \leq 1$ and $b$ is $1$-Lipschitz.
\end{proof}

The comparison \eqref{unbounded inequality} in Lemma \ref{rephrased arc Harnack} has the following important consequence, which proves the last claim of Theorem \ref{unbounded uniformization}.

\begin{prop}\label{bounded equivalence}
 $X_{\e}$ is bounded if and only if $b \in \mathcal{D}(X)$. 
\end{prop}

\begin{proof}
We first suppose that $b \in \mathcal{D}(X)$. For $x \in X$ we can then write $b(x) = |xz| + s$ for some $z \in X$ and $s \in \R$. We let $x \in X$ be given and let $\gamma$ be a geodesic joining $z$ to $x$. Then 
\begin{align*}
d_{\e}(x,z) &\leq \int_{\gamma} \rho_{\e}\, dt \\
&= e^{-s}\int_{0}^{|xz|}e^{-\e t}\,dt \\
&=\e^{-1}e^{-s}(1-e^{-\e |xz|}) \\
&\leq \e^{-1}e^{-s}. 
\end{align*}
It follows that $X_{\e}$ is bounded with $\mathrm{diam} \, X_{\e} \leq 2\e^{-1}e^{-s}$.

Now suppose that $b \in \mathcal{B}(X)$. Then we can find a geodesic ray $\gamma:[0,\infty) \rightarrow X$ and a constant $s \in \R$ such that $b = b_{\gamma}+s$. For each $t \geq 0$ we apply the comparison \eqref{unbounded inequality} with $x = \gamma(t)$ and $y = \gamma(t+1)$ to obtain
\[
d_{\e}(\gamma(t),\gamma(t+1)) \asymp_{C(\e)} e^{-\e b(\gamma(t))} = e^{\e(t-s)},
\]
since $b(\gamma(t)) = -t+s$. Thus as $t \rightarrow \infty$ we have $d_{\e}(\gamma(t),\gamma(t+1)) \rightarrow \infty$. It follows that $X_{\e}$ is unbounded. 
\end{proof}

We next use adapted parametrizations to estimate the length in $X_{\e}$ of geodesics in $X$. Below we write $\omega = \omega_{b}$ for the basepoint of $b$. 

\begin{lem}\label{epsilon geodesic}
Let $x,y \in X$ be given and let $\gamma$ be a geodesic in $X$ joining $x$ to $y$. Suppose that we are given $K \geq 0$ and points $x',y' \in X$ and geodesics $\omega x',\omega y'$ joining $\omega$ to $x$ and $y$ such that $\max\{|xx'|,|yy'|\} \leq K$. Then
\begin{equation}\label{epsilon geodesic estimate}
\l_{\e}(\gamma) \asymp_{C(\delta,K,\e)} e^{-\e(x|y)_{b}}\min\{1,|xy|\}.
\end{equation}
\end{lem}

\begin{proof}
Throughout this proof all additive constants $c \geq 0$ and $C \geq 1$ will depend only on $\delta$, $K$, and $\e$; we write $\doteq$ and $\asymp$ for $\doteq_{c}$ and $\asymp_{C}$ respectively.  We consider an arclength parametrization $\gamma: I \rightarrow X$ of $\gamma$ that is $c$-adapted to $b$ with $c = c(\delta,K)$ as constructed in Lemma \ref{star parametrize}. We assume that $\gamma$ is oriented from $x$ to $y$ and set $w = \gamma(0)$. By \eqref{adapted equation} we then have $b(w) \doteq (x|y)_{b}$. 

When $|xy| \leq 1$ we observe that $|zw| \leq |xy| \leq 1$ for all $z \in xy$. Inequality \eqref{Harnack} then implies that
\begin{equation}\label{small comparison}
\rho_{\e}(z) \asymp_{e^{\e}} \rho_{\e}(w).
\end{equation}
By integrating the comparison \eqref{small comparison} over $\gamma$ we obtain
\[
\l_{\e}(\gamma) \asymp \rho_{\e}(w)|xy| \asymp e^{-\e (x|y)_{b}}|xy|.
\]
This gives the estimate \eqref{epsilon geodesic estimate} when $|xy| \leq 1$. 

We now suppose that $|xy| \geq 1$. Let $\gamma_{1}:[-|xw|,0] \rightarrow X$ and $\gamma_{2}:[0,|yw|] \rightarrow X$ denote the parametrizations of the subsegments of $\gamma$ from $x$ to $w$ and from $w$ to $y$ respectively. Then, using \eqref{adapted equation} and $b(w) \doteq (x|y)_{b}$, we have
\begin{align*}
\l_{\e}(\gamma) &= \l_{\e}(\gamma_{1}) + \l_{\e}(\gamma_{2}) \\
&= \int_{\gamma_{1}}\rho_{\e} \, dt + \int_{\gamma_{2}}\rho_{\e} \, dt \\
&\asymp e^{-\e (x|y)_{b}} \left(\int_{0}^{|xw|}e^{-\e t}\, dt + \int_{0}^{|yw|}e^{-\e t}\, dt\right) \\
&= \e^{-1}e^{-\e (x|y)_{b}}(2-e^{-\e |xw|} - e^{-\e |yw|}).
\end{align*}
It follows immediately that
\[
\l_{\e}(\gamma) \leq Ce^{-\e (x|y)_{b}},
\]
with $C = C(\delta,K,\e)$. This gives the upper bound in \eqref{epsilon geodesic estimate} when $|xy| \geq 1$. For the lower bound we note that since $|xy| \geq 1$ and $|xw| + |yw| = |xy|$, we must have $\min\{|xw|,|yw|\} \geq \frac{1}{2}$. Therefore 
\[
\e^{-1}e^{-\e (x|y)_{b}}(2-e^{-\e |xw|} - e^{-\e |yw|}) \geq \e^{-1}e^{-(x|y)_{b}}(1-e^{-\frac{\e}{2}}). 
\]
This gives the lower bound in \eqref{epsilon geodesic estimate} when $|xy| \geq 1$. 
\end{proof}

In connection with Lemma \ref{epsilon geodesic} it is helpful to formulate the following definition. 

\begin{defn}\label{roughly geodesic}
Let $\omega \in X \cup \p X$ be given. For a constant $K \geq 0$ we say that $X$ is \emph{$K$-roughly geodesic from $\omega$} if for each $x \in X$ there exists $x' \in X$ and a geodesic $\omega x'$ joining $\omega$ to $x'$ such that $|x x'| \leq K$. 
\end{defn}

When $X$ is $K$-roughly geodesic from $\omega$ we can apply Lemma \ref{epsilon geodesic} freely to any geodesic $\gamma$ in $X$ with this constant $K$. If $\omega \in X$ then $X$ is $0$-roughly geodesic from $\omega$ since $X$ is geodesic. For the case $\omega \in \p X$ we note that if $X$ is $K$-roughly starlike from $\omega$ then $X$ is clearly also $K$-roughly geodesic from $\omega$. Note however that $X$ can be roughly geodesic from $\omega \in \p X$ without being roughly starlike from $\omega$; this happens for instance when $X$ is a tree with arbitrarily long finite branches. We also remark that $X$ is always $0$-roughly geodesic from any point of $\p X$ when it is proper. 

When $\rho_{\e}$ is a GH-density we thus obtain the following corollary of Lemma \ref{epsilon geodesic} using the GH-inequality \eqref{first GH}.

\begin{lem}\label{lem:proto estimate both}
Suppose that $X$ is $K$-roughly geodesic from $\omega$ and that $\rho_{\e}$ is a GH-density with constant $M$. Then for any $x,y \in X$ we have
\begin{equation}\label{proto estimate both}
d_{\e}(x,y) \asymp_{C(\delta,K,\e,M)} e^{-\e (x|y)_{b}}\min\{1,|xy|\}.
\end{equation} 
\end{lem}

Lemma \ref{epsilon geodesic} has the following immediate corollary when combined with Lemma \ref{rephrased arc Harnack}.

\begin{cor}\label{criterion}
Suppose that $X$ is $K$-roughly geodesic from $\omega$ and that there is a constant $M_{0} \geq 1$ such that for any $x,y \in X$ with $|xy| > 1$ we have
\begin{equation}\label{criterion estimate}
e^{-\e(x|y)_{b}} \leq M_{0}d_{\e}(x,y).
\end{equation}
Then $\rho_{\e}$ is a GH-density with constant $M = M(\delta,K,\e,M_{0})$. 
\end{cor}

\begin{proof}
If $x,y \in X$ with $|xy| \leq 1$ then combining Lemmas \ref{rephrased arc Harnack} and \ref{epsilon geodesic} establishes the GH-inequality \eqref{first GH} with $M = M(\delta,K,\e)$. If we instead have $|xy| > 1$ then the inequality \eqref{criterion estimate} implies the GH-inequality \eqref{first GH} with $M = M(\delta,K,\e,M_{0})$ by the estimate \eqref{epsilon geodesic estimate} and the fact that $d_{e}(x,y) \leq \l_{\e}(\gamma)$ for any geodesic $\gamma$ joining $x$ to $y$. 
\end{proof}

We will use Corollary \ref{criterion} to show for a CAT$(-1)$ space $X$ that  $\rho_{1} = \rho_{1,b}$ is a GH-density for any $b \in \hat{\mathcal{B}}(X)$ with a universal constant $M \geq 1$. Our proof will be based on the following four point inequality for CAT$(-1)$ spaces. 

\begin{prop}\label{prop:strong hyp}\cite[Proposition 3.3.4]{DSU17}
Let $X$ be a CAT$(-1)$ space. Then for any four points $x,y,z,w \in X$ we have
\begin{equation}\label{strong hyp inequality}
e^{-(x|z)_{w}} \leq e^{-(x|y)_{w}} + e^{-(y|z)_{w}}. 
\end{equation}
\end{prop}
 
 Metric spaces  satisfying the inequality \eqref{strong hyp inequality} are called \emph{strongly hyperbolic} in \cite{DSU17}.

The inequality \eqref{strong hyp inequality} can easily be improved to hold for Gromov products based at any function $b \in \hat{\mathcal{B}}(X)$. 

\begin{lem}\label{improve strong hyp}
Let $X$ be a CAT$(-1)$ space. Then for any $x,y,z \in X$ and $b \in \hat{\mathcal{B}}(X)$ we have
\begin{equation}\label{improve strong inequality}
e^{-(x|z)_{b}} \leq e^{-(x|y)_{b}} + e^{-(y|z)_{b}}. 
\end{equation}
\end{lem}

\begin{proof}
If $b \in \mathcal{D}(X)$ has the form $b(x) = d(x,w)$ for some $w \in X$ then \eqref{improve strong inequality} is just a restatement of \eqref{strong hyp inequality}. The inequality for $b$ of the form $b(x) = d(x,w) + s$ for some $w \in X$ and $s \in \R$ can then be obtained by multiplying the inequality \eqref{strong hyp inequality} through by $e^{-s}$. 

Now suppose that $b \in \mathcal{B}(X)$. Then we can find a geodesic ray $\gamma:[0,\infty) \rightarrow X$ and an $s \in \R$ such that $b = b_{\gamma} + s$. By multiplying the target inequality \eqref{improve strong inequality} through by $e^{s}$ we see that it suffices to consider the case that $b = b_{\gamma}$. Let $x,y,z \in X$ be given. For each $t \in [0,\infty)$ we have from \eqref{strong hyp inequality} that
\[
e^{-(x|z)_{\gamma(t)}} \leq e^{-(x|y)_{\gamma(t)}} + e^{-(y|z)_{\gamma(t)}}.
\]
After multiplying each side by $e^{t}$ and expanding the Gromov products we obtain that
\[
e^{-\frac{1}{2}(|x\gamma(t)|-t + |y\gamma(t)|-t -|xy|)} \leq e^{-\frac{1}{2}(|x\gamma(t)|-t + |z\gamma(t)|-t -|xz|)} + e^{-\frac{1}{2}(|y\gamma(t)|-t + |z\gamma(t)|-t -|yz|)}.
\]
Letting $t \rightarrow \infty$ then gives inequality \eqref{improve strong inequality}. 
\end{proof}

By combining Lemma \ref{improve strong hyp} with Corollary \ref{criterion} we can show that the density $\rho_{1}$ for $b \in \hat{\mathcal{B}}(X)$ is a GH-density when $X$ is a complete CAT$(-1)$ space. The completeness hypothesis is only used in the case $b \in \mathcal{B}(X)$.  

\begin{prop}\label{strong hyperbolic admissible}
There is a constant $M \geq 1$ such that for any complete CAT$(-1)$ space $X$ and any $b \in \hat{\mathcal{B}}(X)$ we have that $\rho_{1} = \rho_{1,b}$ is a GH-density with constant $M$. 
\end{prop}

\begin{proof}
We first observe that $X$ is always $0$-roughly geodesic from $\omega$. For $\omega \in X$ this is trivial, while for $\omega \in \p X$  this can be deduced from the completeness hypothesis together with the CAT$(-1)$ condition on $X$ \cite[Proposition 4.4.4]{DSU17}. Thus we can apply Corollary \ref{criterion} with $K = 0$. Let $x,y \in X$ be given and let $\omega \in X \cup \p X$ denote the basepoint of $b$.  Since $X$ is $\delta$-hyperbolic with $\delta = \delta(\mathbb{H}^{2})$ and we are restricting to the case $\e = 1$, it suffices by Corollary \ref{criterion} to produce a universal constant $M_{0} \geq 1$ such that for any  $x,y \in X$  with $|xy| > 1$ we have
\begin{equation}\label{admissible target}
e^{-(x|y)_{b}} \leq M_{0}d_{1}(x,y). 
\end{equation}

We consider a rectifiable curve $\eta: I \rightarrow X$ joining $x$ to $y$ that is parametrized by arclength (in $X$) and oriented from $x$ to $y$. We define a finite sequence of points $t_{0},\dots,t_{n}$ in $I$ inductively as follows: we set $t_{0}$ to be the left endpoint of $I$ and for each applicable $k > 0$ we set $t_{k}$ to be the supremum of all points $t \in I$ such that $t \geq t_{k-1}$ and $|\eta(s)\eta(t_{k-1})| < 1$ for each $t_{k-1} \leq s \leq t$. The finite length of $\eta$ in $X$ ensures that this sequence is finite, so that we have as a consequence that this process terminates at the right endpoint $t_{n}$ of $I$. The assumption $|xy| > 1$ implies that $n \geq 2$. By construction we then have $|\eta(t_{k})\eta(t_{k+1})| = 1$ for $0 \leq k \leq n-2$ and $|\eta(t_{n-1})\eta(t_{n})| \leq 1$. Since $b$ is $1$-Lipschitz it follows that
\[
e^{-(x|y)_{b}} \leq Ce^{-(\eta(t_{0})|\eta(t_{n-1}))_{b}},
\]
with $C = e$, recalling that $\eta(t_{0}) = x$. Hence to prove \eqref{admissible target} it suffices to show for any rectifiable curve $\eta$ joining $x$ to $y$ that we have  
\begin{equation}\label{reduce target}
e^{-(\eta(t_{0})|\eta(t_{n-1}))_{b}} \leq C\l_{1}(\eta),
\end{equation}
for some constant $C \geq 1$. Repeated application of the inequality \eqref{improve strong inequality} based at $b$ gives
\[
e^{-(\eta(t_{0})|\eta(t_{n-1}))_{b}} \leq \sum_{k=0}^{n-2}e^{-(\eta(t_{k})|\eta(t_{k+1}))_{b}} \leq C\sum_{k=0}^{n-2}e^{-b(\eta(t_{k}))},
\]
with the second inequality holding for $C = e$ since $b$ is $1$-Lipschitz and $|\eta(t_{k})\eta(t_{k+1})| = 1$. For $0 \leq k \leq n-2$ we have by construction that any $t \in [t_{k},t_{k+1}]$ satisfies $|\eta(t_{k})\eta(t)| \leq 1$. Hence the Harnack inequality \eqref{Harnack} implies that 
\[
\l_{1}(\eta|_{[t_{k},t_{k+1}]}) \asymp_{e} e^{-b(\eta(t_{k}))}\l(\eta|_{[t_{k},t_{k+1}]}) \geq e^{-b(\eta(t_{k}))}, 
\]
since the length of $\eta|_{[t_{k},t_{k+1}]}$ in $X$ is at least the distance $1$ between its endpoints. This shows that
\[
\sum_{k=0}^{n-2}e^{-b(\eta(t_{k}))} \leq C\sum_{k=0}^{n-2}\l_{1}(\eta|_{[t_{k},t_{k+1}]}) \leq C\l_{1}(\eta),
\]
with $C = e$. We conclude that the desired estimate \eqref{reduce target} holds, so that as a consequence $\rho_{1}$ is a GH-density with a universal constant $M \geq 1$. 
\end{proof}

By repurposing the proof of Proposition \ref{strong hyperbolic admissible} we are also able to show that $\rho_{\e}$ being a GH-density implies that $\rho_{\e'}$ is also a GH-density for each $0 < \e' \leq \e$, with a constant independent of $\e'$. 

\begin{prop}\label{inheritance}
Suppose that $X$ is $K$-roughly geodesic from $\omega$ and that $\rho_{\e}$ is a GH-density with constant $M$. Then there is a constant $M' = M'(\delta,K,\e,M)$ such that $\rho_{\e'}$ is a GH-density with constant $M'$ for any $0 < \e' \leq \e$. 
\end{prop}

\begin{proof}
Let $\e_{0} = \e_{0}(\delta)$ be the threshold determined by the Harnack inequality \eqref{Harnack} and Theorem \ref{Gehring-Hayman} such that $\rho_{\e}$ is a GH-density with constant $M = 20$ for $0 < \e \leq \e_{0}$. For the purpose of proving the proposition we can then assume that $\e > \e_{0}$ and $\e' \in [\e_{0},\e]$. We will first produce a constant $M_{0} = M_{0}(\delta,K,\e,\e',M)$ such that for any $x,y \in X$ with $|xy| > 1$ we have
\begin{equation}\label{admissible second target}
e^{-\e'(x|y)_{b}} \leq M_{0}d_{\e'}(x,y). 
\end{equation}
As in the proof of Proposition \ref{strong hyperbolic admissible}, we let $\eta: I \rightarrow X$ be a rectifiable curve joining $x$ to $y$ that is parametrized by arclength (in $X$) and oriented from $x$ to $y$. We then construct the finite sequence of points $t_{0},\dots,t_{n}$ in $I$ exactly as in the proof of Proposition \ref{strong hyperbolic admissible}, with $n \geq 2$ since $|xy| > 1$. Since $b$ is $1$-Lipschitz, we conclude as in that proof that it suffices to establish the estimate
\begin{equation}\label{reduce second target}
e^{-\e'(\eta(t_{0})|\eta(t_{n-1}))_{b}} \leq M_{1}\l_{\e'}(\eta),
\end{equation}
for any rectifiable curve $\eta$ joining $x$ to $y$, with $M_{1} = M_{1}(\delta,K,\e,\e',M)$.

We set $\beta = \e'/\e$ and observe that $\e_{0}/\e \leq \beta \leq 1$ by hypothesis. By the triangle inequality for $d_{\e}$ we have
\[
d_{\e}(\eta(t_{0}),\eta(t_{n-1})) \leq \sum_{k=0}^{n-2} d_{\e}(\eta(t_{k}),\eta(t_{k+1})),
\]
which implies that
\[
d_{\e}(\eta(t_{0}),\eta(t_{n-1}))^{\beta} \leq \sum_{k=0}^{n-2} d_{\e}(\eta(t_{k}),\eta(t_{k+1}))^{\beta},
\]
since $0 < \beta \leq 1$. By Lemma \ref{lem:proto estimate both} we then conclude that
\begin{align*}
e^{-\e'(\eta(t_{0})|\eta(t_{n-1}))_{b}} &\leq C\sum_{k=0}^{n-2}e^{-\e'(\eta(t_{k})|\eta(t_{k+1}))_{b}} \\
&\leq C\sum_{k=0}^{n-2}e^{-\e'b(\eta(t_{k}))},
\end{align*}
with $C = C(\delta,K,\e,\e',M)$, where the second inequality follows from the fact that $|\eta(t_{k})\eta(t_{k+1})| = 1$ by construction. Since $|\eta(t_{k})\eta(t)| \leq 1$ for each $t \in [t_{k},t_{k+1}]$, we have by the Harnack inequality \eqref{Harnack},
\[
\l_{\e'}(\eta|_{[t_{k},t_{k+1}]}) \asymp_{e^{\e'}} e^{-\e'b(\eta(t_{k}))}\l(\eta|_{[t_{k},t_{k+1}]}) \geq e^{-\e'b(\eta(t_{k}))}, 
\]
since the length of $\eta|_{[t_{k},t_{k+1}]}$ in $X$ is at least the distance $1$ between its endpoints. It follows that 
\[
\sum_{k=0}^{n-2}e^{-\e'b(\eta(t_{k}))} \leq C(\e')\sum_{k=0}^{n-2}\l_{\e'}(\eta|_{[t_{k},t_{k+1}]}) \leq C(\e')\l_{\e'}(\eta),
\]
with $C(\e') = e^{\e'}$. This proves the desired inequality \eqref{reduce second target}. 

By direct inspection of the calculations in this proof, as well as in the proofs of our previous lemmas in this section, one can verify that if $\e' \in [\e_{0},\e]$ then the constants can always be chosen to depend only on $\e_{0}$ and $\e$. This shows that the GH-constant $M'$ for $\rho_{\e'}$ can be chosen such that $M' = M'(\delta,K,\e,\e_{0},M)$. Since $\e_{0} = \e_{0}(\delta)$ we in fact obtain that $M' = M'(\delta,K,\e,M)$, as desired. 
\end{proof}

We recall that $\omega = \omega_{b}$ denotes the basepoint of $b$. We end this section by using Lemma \ref{epsilon geodesic} to construct a map 
\begin{equation}\label{construct identify}
\varphi_{\e} = \varphi_{\e,b}: \p_{\omega} X \rightarrow \p X_{\e},
\end{equation}
when $X$ is roughly geodesic from $\omega$ and $\p_{\omega} X \neq \emptyset$. We will be using the following corollary of the estimate \eqref{epsilon geodesic estimate} for $x,y \in X$, 
\begin{equation}\label{epsilon geodesic corollary}
d_{\e}(x,y) \leq C(\delta,\e,K) e^{-\e(x|y)_{b}}\min\{1,|xy|\}.
\end{equation}
 The inequality \eqref{epsilon geodesic corollary} implies that if $\{x_{n}\}$ is a sequence in $X$ that converges to infinity with respect to $\omega$ (recall this means that $(x_{m}|x_{n})_{b} \rightarrow \infty$ as $m,n \rightarrow \infty$) then $\{x_{n}\}$ is a Cauchy sequence in $X_{\e}$. Since $X$ and $X_{\e}$ are biLipschitz on bounded subsets of $X$ by Lemma \ref{arc Harnack}, it is easy to see that $\{x_{n}\}$ cannot converge to a point of $X_{\e}$. It follows that $X_{\e}$ is incomplete as long as $\p_{\omega}X$ contains at least one point; the only exceptions are when either $X$ is bounded or $\omega$ is the only point of $\p X$. Furthermore a second application of \eqref{epsilon geodesic corollary} shows that sequences $\{x_{n}\}$ and $\{y_{n}\}$ converging to infinity with respect to $\omega$ that are equivalent with respect to $\omega$ are equivalent as Cauchy sequences in $X_{\e}$, i.e., $d_{\e}(x_{n},y_{n}) \rightarrow 0$ as $n \rightarrow \infty$. Setting $\p X_{\e} = \bar{X}_{\e} \backslash X_{\e}$ to be the complement of $X_{\e}$ in its completion, we thus have a well-defined map $\varphi_{\e}: \p_{\omega} X \rightarrow \p X_{\e}$ given by sending a sequence $\{x_{n}\}$ converging to infinity with respect to $\omega$ to its limit in $\p X_{\e}$ inside of $\bar{X}_{\e}$. In the next section we will show that $\varphi_{\e}$ is a bijection when $X$ is roughly starlike from $\omega$ and $\rho_{\e}$ is a GH-density. We remark that the map $\varphi_{\e}$ need not be a bijection in general, see \cite[Proposition 4.1]{BBS21}.

\subsection{Uniformizing by Busemann functions}\label{subsec:busemann uniformize} For this section we let $X$ be a complete geodesic $\delta$-hyperbolic space. We let $b \in \mathcal{B}(X)$ be given with basepoint $\omega \in \p X$ and suppose that $X$ is $K$-roughly starlike from $\omega$. For a given $\e > 0$ we suppose $\rho_{\e,b}$ is a GH-density with constant $M$. As in Section \ref{subsec:estimate uniform}, to simplify notation we will drop $b$ from the notation for objects associated to the uniformization and write $\rho_{\e}:=\rho_{\e,b}$, $X_{\e} :=X_{\e,b}$, etc. As before, for a curve $\eta: I \rightarrow X_{\e}$ we will write $\l_{\e}(\eta):=\l_{\rho_{\e}}(\eta)$ for its length measured in the metric $d_{\e}$. For brevity, throughout this section we write $\doteq$ for equality up to an additive that depends only on $\delta$, $K$, $\e$, and $M$, and write $\asymp$ for equality up to a multiplicative constant that depends only on those same parameters. We write $c \geq 0$ and $C \geq 1$ for additive and multiplicative constants depending only on these parameters. 

The rough starlikeness of $X$ from $\omega$ implies that $\p_{\omega}X$ contains at least one point, so as a consequence the space $X_{\e}$ is incomplete by the discussion at the conclusion of the previous section. As before we write $\bar{X}_{\e}$ for the completion of the uniformization $X_{\e}$ and $\p X_{\e} = \bar{X}_{\e} \backslash X_{\e}$ for the boundary of $X_{\e}$ inside its completion. We will continue to write $d_{\e}$ for the canonical extension of this metric on $X_{\e}$ to the completion $\bar{X}_{\e}$. We write $d_{\e}(x) = \dist(x,\p X_{\e})$ for the distance to the boundary in $X_{\e}$. We let $\varphi_{\e}: \p_{\omega}X \rightarrow \p X_{\e}$ be the map constructed in \eqref{construct identify} at the end of the previous section by sending a sequence $\{x_{n}\}$ converging to infinity with respect to $\omega$ to its limit in $\p X_{\e}$ as a Cauchy sequence in $X_{\e}$. We formally extend the distance function $d_{\e}$ to $\p_{\omega}X$ by setting $d_{\e}(x,y):=d_{\e}(\varphi_{\e}(x),\varphi_{\e}(y))$ for $x,y \in X \cup \p_{\omega}X$, where we set $\varphi_{\e}(x) = x$ for $x \in X$. 

Our first lemma extends the estimate \eqref{proto estimate both} to hold for $x,y \in X \cup \p_{\omega} X$ with our formal extension of $d_{\e}$ to $\p_{\omega}X$. We recall our convention that for $x,y \in X \cup \p X$ we define $|xy| = \infty$ if $x \neq y$ and either $x \in \p X$ or $y \in \p X$, and define $|xy| = 0$ if $x = y \in \p X$.  

\begin{lem}\label{lem:estimate both}
For any $x,y \in X \cup \p_{\omega} X$ we have
\begin{equation}\label{estimate both}
d_{\e}(x,y) \asymp e^{-\e (x|y)_{b}}\min\{1,|xy|\}.
\end{equation} 
\end{lem}

\begin{proof}
The case in which $x,y \in X$ is an immediate consequence of Lemma \ref{lem:proto estimate both} since $X$ is $K$-roughly geodesic from $\omega$ (because $X$ is $K$-roughly starlike from $\omega$). We thus only need to consider the case in which at least one of the points belongs to $\p_{\omega}X$. Since \eqref{estimate both} holds trivially when $x = y$ we can assume that $x \neq y$. We can then assume without loss of generality that $x \in \p_{\omega}X$. We then need to show that $d_{\e}(x,y) \asymp e^{-\e (x|y)_{b}}$. We let $\{x_{n}\}$ and $\{y_{n}\}$ be sequences converging to infinity with respect to $\omega$ in $X$ that represent the points $x$ and $y$ respectively; if $y \in X$ then we instead set $y_{n} = y$ for all $n$. Our definition of the extension of $d_{\e}$ to $\p_{\omega}X$ then implies that we have $\lim_{n \rightarrow \infty} d_{\e}(x_{n},y_{n}) = d_{\e}(x,y)$. For $n$ sufficiently large we will have $(x|y)_{b} \doteq_{600\delta} (x_{n}|y_{n})_{b}$ by Lemma \ref{busemann inequality} and we will have $|x_{n}y_{n}| \geq 1$. The comparison \eqref{estimate both} then follows from the corresponding comparison for $x_{n}$ and $y_{n}$ for sufficiently large $n$. 
\end{proof}


By using Lemma \ref{lem:estimate both} we are able to show that $\varphi_{\e}$ is a bijection. 

\begin{lem}\label{boundary identification}
The map $\varphi_{\e}: \p_{\omega} X \rightarrow \p X_{\e}$ is a bijection. 
\end{lem}

\begin{proof}
The injectivity of $\varphi_{\e}$ follows immediately from Lemma \ref{lem:estimate both} applied to $x \neq y \in \p_{\omega}X$. Thus our focus will be on showing that $\varphi_{\e}$ is surjective. Let $\{x_{n}\}$ be a Cauchy sequence in $X_{\e}$ that converges to a point $z \in \p X_{\e}$. We claim that the sequence $\{x_{n}\}$ cannot belong to a bounded subset of $X$. If it did then for a fixed $p \in X$ there would be an $r > 0$ such that $\{x_{n}\} \subset B(p,r)$ for all $n$, with $B(p,r)$ denoting the ball of radius $r$ centered at $p$ in $X$. Lemma \ref{arc Harnack} shows that the metrics on $X$ and $X_{\e}$ are biLipschitz to one another on $B(p,2r)$, which implies that $\{x_{n}\}$ is also a Cauchy sequence in $X$. Since $X$ is complete this Cauchy sequence must converge in $X$ to a point $y \in B(p,2r)$. However this means that $\{x_{n}\}$ also converges to $y$ in $X_{\e}$, contradicting that $\{x_{n}\}$ converges to a point of $\p X_{\e}$. 

Thus, by passing to a subsequence if necessary, we can assume that $|x_{m}x_{n}| \geq 1$ for $m \neq n$. It then follows from Lemma \ref{lem:estimate both} that for $m \neq n$, 
\[
d_{\e}(x_{m},x_{n}) \asymp e^{-\e (x_{m}|x_{n})_{b}}.
\]
Since $d_{\e}(x_{n},x_{m}) \rightarrow 0$ as $m,n \rightarrow \infty$, we conclude that $(x_{m}|x_{n})_{b} \rightarrow \infty$ as $m,n \rightarrow \infty$. Thus $\{x_{n}\}$ converges to infinity with respect to $\omega$. Letting $\xi \in \p_{\omega}X$ denote the point in the Gromov boundary relative to $\omega$ represented by the sequence $\{x_{n}\}$,  the construction of $\varphi_{\e}$ then shows that $\varphi_{\e}(\xi) = z$. We conclude that $\varphi_{\e}$ is surjective. 
\end{proof}

We can now prove Theorem \ref{identification theorem} in the case $b \in \mathcal{B}(X)$. We recall the definition \eqref{visual quasi} of the model visual quasi-metric $\theta_{\e,b}$ on $\p_{\omega}X$.

\begin{proof}[Proof of Theorem \ref{identification theorem}]
The fact that $\varphi_{\e}$ is a bijection follows from Lemma \ref{boundary identification}. To complete the proof of Theorem \ref{identification theorem} it then suffices to show that for any $\xi,\zeta \in \p_{\omega}X$ there is a constant  $L = L(\delta,K,\e,M)$ such that
\[
d_{\e}(\xi,\zeta) \asymp_{L} \theta_{\e,b}(\xi,\zeta) = e^{-\e(\xi|\zeta)_{b}}.
\] 
This desired comparison then follows from Lemma \ref{lem:estimate both}.
\end{proof} 

The next lemma shows that the distance to $\p X_{\e}$ can be computed in terms of the density $\rho_{\e}$.

\begin{lem}\label{compute distance}
For $x \in X$ we have
\begin{equation}\label{compute distance inequality}
d_{\e}(x) \asymp \rho_{\e}(x).
\end{equation}
\end{lem}

\begin{proof}
Let $x \in X$ be given. We first compute the upper bound in \eqref{compute distance inequality}. By the rough-starlikeness condition we can find a geodesic line $\gamma: \R \rightarrow X$ starting at $\omega$ and ending at some $\xi \in \p X$ with $\dist(x,\gamma) \leq K$. Using Lemma \ref{geodesic busemann} we can consider $\gamma$ as parametrized by arclength with $b(\gamma(t)) \doteq_{144\delta} t$ for each $t \in \R$. We let $s \in \R$ be such that $|x\gamma(s)| \leq K$. We then compute, 
\begin{align*}
\l_{\e}(\gamma|_{[s,\infty)}) &= \int_{s}^{\infty}e^{-\e b(\gamma(t))}\,dt \\
& \asymp \int_{s}^{\infty}e^{-\e t}\,dt \\
&= \e^{-1}e^{-\e s}.
\end{align*}
By Lemma \ref{arc Harnack} we then have
\begin{align*}
d_{\e}(x,\xi) &\leq d_{\e}(x,\gamma(s)) + d_{\e}(\gamma(s),\xi) \\
&\leq \e^{-1}\rho_{\e}(x)(e^{\e |x\gamma(s)|}-1) + \l_{\e}(\gamma|_{[s,\infty)}) \\
&\leq C\rho_{\e}(x).
\end{align*}
Since $d_{\e}(x) \leq d_{\e}(x,\xi)$ the upper bound follows. 

For the lower bound we let $\xi \in \p X_{\e}$ be a given point, which we can think of as a point in $\p_{\omega}X$ using Lemma \ref{boundary identification}. By rough starlikeness we can then find a geodesic line $\gamma: \R \rightarrow X$ starting at $\omega$ and ending at $\xi$. For $n \in \N$ we note that $|x\gamma(n)| \rightarrow \infty$ as $n \rightarrow \infty$, so we will have $|x\gamma(n)| \geq 1$ for all sufficiently large $n$. For sufficiently large $n$ we can then apply \eqref{estimate both} and Lemma \ref{busemann inequality} to obtain 
\[
d_{\e}(x,\gamma(n)) \asymp  e^{-\e (x|\gamma(n))_{b}} \asymp e^{-\e (x|\xi)_{b}}.
\]
By \eqref{both busemann boundary} we have $(x|\xi)_{b} \leq b(x) + 600\delta$. By combining this with the above we obtain that
\[
d_{\e}(x,\gamma(n)) \geq C^{-1}\rho_{\e}(x).
\]
This gives the lower bound since $d_{\e}(x,\gamma(n)) \rightarrow d_{\e}(x,\xi)$ as $n \rightarrow \infty$. 
\end{proof}

We can now complete the proof of Theorem \ref{unbounded uniformization} in the case $b \in \mathcal{B}(X)$. Since we have already shown that $X_{\e}$ is unbounded in Proposition \ref{bounded equivalence}, to prove Theorem \ref{unbounded uniformization} we only need to show that geodesics in $X$ are uniform curves in $X_{\e}$. For this we need to extend the definition of uniform curves to cover curves defined on arbitrary subintervals $I \subset \R$. 

As in Definition \ref{def:uniform}, we let $(\Omega,d)$ be an incomplete metric space, set $\p \Omega = \bar{\Omega} \backslash \Omega$, and set $d_{\Omega}(x)=\dist(x,\p \Omega)$. We consider a curve $\gamma:I \rightarrow \Omega$ defined on an arbitrary subinterval $I \subset \R$; we write $t_{-} \in [-\infty,\infty)$ and $t_{+} \in (-\infty,\infty]$ for the endpoints of $I$. If $\gamma$ has finite length, $\l(\gamma) < \infty$, then $\gamma$ has well-defined endpoints $\gamma_{-},\gamma_{+} \in \bar{\Omega}$ defined by the limits $\gamma(t_{-}) = \lim_{t \rightarrow t_{-}} \gamma(t)$ and $\gamma(t_{+}) = \lim_{t \rightarrow t_{+}} \gamma(t)$ in $\bar{\Omega}$, which exist because $\l(\gamma) < \infty$.  

\begin{defn}\label{def:extend uniform}For a constant $A \geq 1$ and an interval $I \subset \R$, a curve $\gamma: I \rightarrow \Omega$ with $\l(\gamma) < \infty$ is \emph{$A$-uniform} if 
\begin{equation}\label{extend uniform one}
\l(\gamma) \leq Ad(\gamma_{-},\gamma_{+}),
\end{equation}
and if for every $t \in I$ we have
\begin{equation}\label{extend uniform two}
\min\{\l(\gamma|_{I_{\leq t}}),\l(\gamma|_{I_{\geq t}})\} \leq A d_{\Omega}(\gamma(t)). 
\end{equation}
If $\l(\gamma) = \infty$ then we instead define $\gamma$ to be $A$-uniform if \eqref{extend uniform two} holds and if $d(\gamma(s),\gamma(t)) \rightarrow \infty$ as $s \rightarrow t_{-}$ and $t \rightarrow t_{+}$.
\end{defn}


\begin{prop}\label{uniformization prop}
There is an $A = A(\delta,K,\e,M)$ such that any geodesic in $X$ is an $A$-uniform curve in $X_{\e}$. Consequently $X_{\e}$ is $A$-uniform. 
\end{prop}

\begin{proof}
We first consider the case of a geodesic $xy$ in $X$ joining two points $x,y \in X \cup \p_{\omega} X$. Let $\gamma:I \rightarrow X$ be a parametrization of $xy$ that is $c$-adapted to $b$, $c = c(\delta,K)$, as in Lemma \ref{star parametrize}; we can always find points $x', y' \in X \cup \p_{\omega}X$ satisfying the hypotheses of the lemma by the $K$-rough starlikeness hypothesis from $\omega$. Let $t_{-} \in [-\infty,\infty)$, $t_{+} \in (-\infty,\infty]$ be the endpoints of $I$. Let $\{s_{n}\},\{t_{n}\} \subset I$ be sequences such that $s_{n} \rightarrow t_{-}$, $t_{n} \rightarrow t_{+}$, and $s_{n} < t_{n}$ for each $n$. Applying inequality \eqref{first GH} to the geodesic $\gamma|_{[s_{n},t_{n}]}$ joining $\gamma(s_{n})$ to $\gamma(t_{n})$ gives
\begin{equation}\label{sequential GH}
\l_{\e}(\gamma|_{[s_{n},t_{n}]}) \leq Md_{\e}(\gamma(s_{n}),\gamma(t_{n})).
\end{equation}
Letting $n \rightarrow \infty$, the left side of \eqref{sequential GH} converges to $\l_{\e}(\gamma)$. If $x \in X$ then the sequence $\{\gamma(s_{n})\}$ converges to $x$ in $X$, while if $x \in \p_{\omega}X$ then the sequence $\{\gamma(s_{n})\}$ belongs to the equivalence class of $X$ with respect to $\omega$. In both cases we then have that $\{\gamma(s_{n})\}$ converges to $x$ in $\bar{X}_{\e}$, with the second case following from the construction of the identification $\varphi_{\e}: \p_{\omega}X \rightarrow \p X_{\e}$. The same discussion applies to the sequence $\{\gamma(t_{n})\}$ in relation to $y$. It follows that $d_{\e}(\gamma(s_{n}),\gamma(t_{n})) \rightarrow d_{\e}(x,y)$ as $n \rightarrow \infty$. Consequently $\gamma$ has finite length in $X_{\e}$ with endpoints $\gamma_{-} = x$ and $\gamma_{+} = y$. The inequality \eqref{extend uniform one} then follows by letting $n \rightarrow \infty$ in \eqref{sequential GH}. 

We next verify inequality \eqref{extend uniform two}. It suffices to verify this inequality in the case that $s \in I_{\geq 0}$, since we can deduce the case $s \in I_{\leq 0}$ from this by reversing the roles of $x$ and $y$. We thus assume that $s \in I_{\geq 0}$. A straightforward calculation with \eqref{adapted equation} gives us that
\begin{equation}\label{near uniform}
\l_{\e}(\gamma|_{I_{\geq s}}) \asymp e^{-\e (x|y)_{b}}\int_{I_{\geq s}} e^{-\e t} \,dt \leq \e^{-1} e^{-\e (x|y)_{b}} e^{-\e s}. 
\end{equation}
Since $s + (x|y)_{b} \doteq b(\eta(s))$, it then follows from \eqref{near uniform} and Lemma \ref{compute distance} that
\[
\l_{\e}(\gamma|_{I_{\geq s}}) \leq C \rho_{\e}(\gamma(s)) \leq C d_{\e}(\gamma(s)),
\]
with $C = C(\delta,K,\e,M)$. We conclude that $\gamma$ is an $A$-uniform curve in $X_{\e}$ with  $A = A(\delta,K,\e,M)$. Since any two points $x,y \in X$ can be joined by a geodesic $xy$ in $X$, this implies that the metric space $X_{\e}$ is $A$-uniform. 

It remains to treat the case of a geodesic $\omega x$ joining the basepoint $\omega \in \p X$ of $b$ to a point $x \in X \cup \p_{\omega}X$. By applying (2) of Lemma \ref{geodesic busemann} with $u = 0$ we can find an arclength parametrization $\gamma:(-\infty,a] \rightarrow X$ of $\omega x$, $a \in (-\infty,\infty]$, with $b(\gamma(t)) \doteq_{144\delta} t$ for $t \in (-\infty,a]$. For $s \leq t \in (-\infty,a]$ we then have, by a computation similar to the one done in Lemma \ref{compute distance},
\begin{equation}\label{two endpoint}
\l_{\e}(\gamma|_{[s,t]}) \asymp e^{-\e s} - e^{-\e t}. 
\end{equation}
By letting $s \rightarrow -\infty$ we conclude that $\l_{\e}(\gamma) = \infty$. The comparison \eqref{two endpoint} together with the GH-inequality \eqref{first GH} then implies that $d_{\e}(\gamma(s),\gamma(t)) \rightarrow \infty$ as $s \rightarrow -\infty$ and $t \rightarrow a$. Thus our replacement condition in the infinite length case for \eqref{extend uniform one} is satisfied. Combining \eqref{two endpoint} and Lemma \ref{compute distance} also implies for each $s \in (-\infty,a]$ that 
\[
\l_{\e}(\gamma|_{I_{\geq s}}) \leq Ce^{-\e s} \leq C\rho_{\e}(\gamma(s)) \leq Cd_{\e}(\gamma(s)),
\]
with $C = C(\delta,K,\e,M)$ in each inequality. Thus \eqref{extend uniform two} also holds for $\gamma$ with $A = A(\delta,K,\e,M)$. We conclude that $\gamma$ is an $A$-uniform curve in $X_{\e}$ in this case as well.  
\end{proof}

\subsection{Uniformizing by distance functions}\label{subsec:distance uniformize}For this section we assume the same setup as in Section \ref{subsec:busemann uniformize}, with the exception that we will be assuming instead that $b \in \mathcal{D}(X)$ with basepoint $z \in X$.  We will reduce Theorems \ref{unbounded uniformization} and \ref{identification theorem} in this case to the case considered in Section \ref{subsec:busemann uniformize} using the following general construction. 

\begin{defn}\label{ray augment} For a metric space $(X,d_{X})$ and a point $z \in X$ we let $Y = X \cup_{z \sim 0} [0,\infty)$ be the metric space obtained by gluing the half-line $[0,\infty)$ to $X$ by identifying the point $z$ with $0 \in [0,\infty)$. The metric $d_{Y}$ on $Y$ is defined by setting $d_{Y}(x,y) = d_{X}(x,y)$ for $x,y \in X$, $d_{Y}(s,t) = |s-t|$ for $s,t \in [0,\infty)$, and $d_{Y}(x,s) = d_{X}(x,z) + s$ for $x \in X$, $s \in [0,\infty)$. The space $X$ then has a canonical isometric embedding into $Y$. We refer to the metric space $(Y,d_{Y})$ as the \emph{ray augmentation of $X$ based at $z$}.
\end{defn} 

The following trick will be the basis of many of the results we prove regarding the ray augmentation. 

\begin{lem}\label{shorten}
Let $(X,d_{X})$ be a metric space. Let $(Y,d_{Y})$ be the ray augmentation of $X$ based at $z \in X$. Then for each curve $\gamma:I \rightarrow Y$ in $Y$ there is a curve $\sigma:I \rightarrow X$ such that $\sigma(t) = \gamma(t)$ when $\gamma(t) \in X$ and $\sigma(t) = z$ when $\gamma(t) \notin X$. 
\end{lem}

\begin{proof}
Consider the retraction $r: Y \rightarrow X$ given by setting $r(x) = x$ for $x \in X$ and $r(t) = z$ for $t \in [0,\infty)$. It is easy to see that $r$ is $1$-Lipschitz; in particular $r$ is continuous. For a given curve $\gamma:I \rightarrow Y$ the curve $\sigma = r \circ \gamma:I \rightarrow X$ then has the desired properties.  
\end{proof}

We refer to the curve $\sigma$ constructed from $\gamma$ in Lemma \ref{shorten} as the \emph{shortening} of $\gamma$ to $X$. We apply Lemma \ref{shorten} to conformal deformations of the ray augmentation. 

\begin{lem}\label{transfer admissible}
Let $(X,d_{X})$ be a geodesic metric space and let $(Y,d_{Y})$ be the ray augmentation of $X$ based at $z \in X$. Let $\rho: X \rightarrow (0,\infty)$ be a continuous density on $X$ and let $\t{\rho}:Y \rightarrow (0,\infty)$ be a continuous density on $Y$ with $\t{\rho}|_{X} = \rho$. Then the isometric embedding $X \rightarrow Y$ induces an isometric embedding $X_{\rho} \rightarrow Y_{\t{\rho}}$ of the corresponding conformal deformations. Furthermore $\rho$ is a GH-density with constant $M \geq 1$ if and only if $\t{\rho}$ is a GH-density with the same constant $M$.
\end{lem} 

\begin{proof}
We write $d_{\rho}$ for the metric on $X_{\rho}$ and $d_{\t{\rho}}$ for the metric on $Y_{\t{\rho}}$. Let $x,y \in X$ be given. Let $\gamma:I \rightarrow Y$ be a rectifiable curve joining them and let $\sigma:I \rightarrow X$ be the shortening of $\gamma$ to $X$. Since $\t{\rho}|_{X} = \rho$ we clearly have
\[
\int_{\sigma} \t{\rho}\, ds = \int_{\sigma} \rho\, ds = \int_{\gamma^{-1}(X)} \rho\, ds \leq \int_{\gamma} \t{\rho}\, ds.
\]
It follows that $\l_{\t{\rho}}(\sigma) \leq \l_{\t{\rho}}(\gamma)$. Thus in computing $d_{\t{\rho}}(x,y)$ it suffices to minimize $\l_{\t{\rho}}(\gamma)$ over curves $\gamma: I \rightarrow X$ taking values only in $X$. Since for such curves we have $\l_{\t{\rho}}(\gamma) = \l_{\rho}(\gamma)$, it immediately follows that $d_{\t{\rho}}(x,y) = d_{\rho}(x,y)$. We conclude that the embedding $X_{\rho} \rightarrow Y_{\t{\rho}}$ is an isometry. 

It is obvious that $\rho$ is a GH-density with constant $M$ if $\t{\rho}$ is a GH-density with constant $M$, since $\t{\rho}$ restricts to $\rho$ on $X$ and geodesics in $X$ are also geodesics in $Y$. For the converse we assume that $\rho$ is a GH-density with constant $M$ and let $x,y \in Y$ be given points. We need to prove the GH-inequality \eqref{first GH} for any geodesic $\gamma$ joining $x$ to $y$ in $Y$. 

If $x,y \in X$ then any geodesic $\gamma$ joining $x$ to $y$ in $Y$ is in fact a geodesic joining $x$ to $y$ in $X$. The inequality \eqref{first GH} for $\t{\rho}$  then follows from the corresponding inequality for $\rho$. If $x,y \in [0,\infty) = Y \backslash (X \backslash \{z\})$ then there is only one geodesic $\gamma$ from $x$ to $y$ in $Y$, which is simply the interval connecting them in $[0,\infty)$. Since any curve $\sigma$ joining $x$ to $y$ in $Y$ must contain this interval we in fact have $\l_{\t{\rho}}(\gamma) = d_{\t{\rho}}(x,y)$. Thus inequality \eqref{first GH} holds in this case as well. 

The final case is that in which $x \in X$ and $y \in [0,\infty)$. Let $\gamma$ be a geodesic joining $x$ to $y$ in $Y$. Then we can write $\gamma = \sigma \cup \eta$, where $\sigma$ is a geodesic in $X$ joining $x$ to $z$ and $\eta$ is the geodesic in $[0,\infty)$ joining $0$ to $y$, which is just the interval connecting these points. Given a rectifiable curve $\alpha: I\rightarrow Y$ joining $x$ to $y$ we let $\beta$ be the shortening of $\alpha$ to $X$. Then by the inequality \eqref{first GH} for $\rho$ we have $\l_{\rho}(\sigma) \leq M\l_{\rho}(\beta)$. Since the intersection $\alpha(I) \cap [0,\infty)$ must contain the interval $\eta$ joining $0$ to $y$, we deduce from this that
\begin{align*}
\l_{\t{\rho}}(\gamma) &= \l_{\t{\rho}}(\sigma) + \l_{\t{\rho}}(\eta) \\
&\leq M\left(\l_{\t{\rho}}(\beta) + \int_{\alpha^{-1}([0,\infty))} \t{\rho}\, ds\right) \\
&= M \l_{\t{\rho}}(\alpha). 
\end{align*}
Minimizing over all rectifiable curves $\alpha$ joining $x$ to $y$ then gives inequality \eqref{first GH}.
\end{proof}

We assume now that $X$ is a geodesic $\delta$-hyperbolic space. We let $Y$ be the ray augmentation of $X$ based at some point $z \in X$. It is an easy exercise to see that $Y$ is also $\delta$-hyperbolic; recall that we have defined $\delta$-hyperbolicity using $\delta$-thin triangles. We also note that $Y$ is complete if $X$ is complete. We will continue to use the generic distance notation $|xy|$ for the distance between $x,y \in Y$, noting that there is no conflict with the distance notation for $X$ since $X$ sits isometrically inside of $Y$.

By definition the \emph{canonical ray} in $Y$ is the geodesic ray $\gamma:[0,\infty) \rightarrow Y$ corresponding to the canonical parametrization of the copy of $[0,\infty)$ that we glued onto $X$. The key property of the ray augmentation is that the Busemann function $b_{\gamma}$ associated to the canonical ray restricts on $X$ to the distance from the distinguished point $z$. 

\begin{lem}\label{augment}
Let $\gamma$ be the canonical ray in $Y$. Then for $x \in X$ we have $b_{\gamma}(x) = |xz|$. 
\end{lem}

\begin{proof}
For $x \in X$ and $t \geq 0$ we have
\[
|\gamma(t)x|-t = t+|xz|-t = |xz|,
\]
which implies upon taking $t \rightarrow \infty$ that $b_{\gamma}(x) = |xz|$.
\end{proof}

Let $\omega \in \p Y$ be the point in the Gromov boundary defined by the canonical ray. We next show that $\p_{\omega}Y = \p Y \backslash \{\omega\}$ identifies canonically with $\p X$ and rough starlikeness from $z$ in $X$ passes over to rough starlikeness from $\omega$ in $Y$. 

\begin{lem}\label{ray boundary}
We have $\p_{\omega} Y = \p X$. Furthermore any visual metric on $\p X$ based at $z$ with parameter $\e > 0$ also defines a visual metric on $\p_{\omega} Y$ based at $\omega$ with parameter $\e$, and the converse holds as well. If $X$ is $K$-roughly starlike from $z$ then $Y$ is $K$-roughly starlike from $\omega$. 
\end{lem}

\begin{proof}
For the first assertion it suffices to show that if $\{x_{n}\}$ is a sequence converging to infinity in $Y$ then there is an $N \in \N$ such that for $n \geq N$ the points $x_{n}$ either all belong to $X$ or all belong to $[0,\infty)$. Recall that $\{x_{n}\}$ converges to infinity if we have $(x_{n}|x_{m})_{z} \rightarrow \infty$ as $m,n \rightarrow \infty$.  If our assertion did not hold then we could find subsequences $\{y_{n}\}$ and $\{z_{n}\}$ of the sequence $\{x_{n}\}$ such that $\{y_{n}\} \subset X$, $\{z_{n}\} \subset [0,\infty)$, and $(y_{n}|z_{n})_{z} \rightarrow \infty$ as $n \rightarrow \infty$. But then $|y_{n}z_{n}| = |y_{n}z| + |z_{n}z|$ and therefore
\[
(y_{n}|z_{n})_{z} = \frac{1}{2}(|y_{n}z|+|z_{n}z|-|y_{n}z_{n}|) = 0,
\]
contradicting our assumption that $(y_{n}|z_{n})_{z} \rightarrow \infty$. Since all sequences $\{x_{n}\}$ converging to infinity in $Y$ that belong exclusively to $[0,\infty)$ must converge to $\omega$, it follows that we have a canonical identification of $\p X$ with $\p_{\omega} Y$. 

Let $b_{\gamma}$ be the Busemann function associated to the canonical ray $\gamma$ as in Lemma \ref{augment}. We observe that $b_{\gamma}(x) = |xz|$ for $x \in X$ implies that $(x|y)_{b_{\gamma}} = (x|y)_{z}$ for $x,y \in X$. Since any sequence $\{x_{n}\}$ converging to infinity in $Y$ that does not converge to $\omega$ must eventually stay within $X$, it follows that $(\xi|\zeta)_{b_{\gamma}} = (\xi|\zeta)_{z}$ for $\xi,\zeta \in \p X$. Thus through our identification $\p_{\omega} Y = \p X$ we have a canonical identification between visual metrics on $\p_{\omega} Y$ based at $\omega$ and visual metrics on $\p X$ based at $z$ for any parameter $\e > 0$. 

Observe that if $\sigma: [0,\infty) \rightarrow X$ is a geodesic ray starting at $z$ then the map $\t{\sigma}: \R \rightarrow Y$ defined by $\t{\sigma}(t) = \sigma(t)$ for $t \geq 0$ and $\t{\sigma}(t) = -t \in [0,\infty)$ for $t \leq 0$ defines a geodesic line in $Y$ that begins at $\omega$, coincides with $\sigma$ inside of $X$, and has the same endpoint in $\p X = \p_{\omega} Y$ as $\sigma$. This implies that if $X$ is $K$-roughly starlike from $z$ for some $K \geq 0$ then $Y$ is $K$-roughly starlike from $\omega$.
\end{proof}

We can now complete the proofs of Theorems \ref{unbounded uniformization} and \ref{identification theorem} by showing that they also hold for $b \in \mathcal{D}(X)$.

\begin{prop}\label{bound from unbound}
Theorems \ref{unbounded uniformization} and \ref{identification theorem} hold for $b \in \mathcal{D}(X)$.
\end{prop}

\begin{proof}
Let $X$ be a complete geodesic $\delta$-hyperbolic space that is $K$-roughly starlike from $z \in X$ and let $b \in \mathcal{D}(X)$ have the form $b(x) = |xz| + s$ for some $s \in \R$. We assume that $\e > 0$ is such that $\rho_{\e,b}$ is a GH-density with constant $M$. We let $Y = X \cup_{z \sim 0} [0,\infty)$ be the ray augmentation of $X$ based at $z$, let $\gamma$ be the canonical ray in $Y$, and set $\t{b} = b_{\gamma}+s$. Then $Y$ is a complete geodesic $\delta$-hyperbolic space that is $K$-roughly starlike from the endpoint $\omega \in \p Y$ of the canonical ray by Lemma \ref{ray boundary}. We have $\t{b}|_{X} = b$ by Lemma \ref{augment}. Thus by Lemma \ref{transfer admissible} the embedding $X_{\e,b} \rightarrow Y_{\e,\t{b}}$ is isometric and $\rho_{\e,\t{b}}$ is a GH-density with the same constant $M$. We can then apply Theorems \ref{unbounded uniformization} and \ref{identification theorem} to $Y$ equipped with the Busemann function $\t{b}$. 

For $t \in [0,\infty) \subset Y$ we have $\t{b}(t) = -t+s$. Thus a straightforward calculation shows that $\l_{\e}(\gamma) = \infty$. It follows from this and the GH-inequality \eqref{first GH} that the only boundary points of the metric space $Y_{\e,\t{b}}$ are the boundary points of $X_{\e,b}$, i.e., $\p X_{\e,b} = \p Y_{\e,\t{b}}$. By applying  Theorem \ref{identification theorem} to $Y$ and then using Lemma \ref{ray boundary}, we conclude that we have a canonical identification $\varphi_{\e}:\p X \rightarrow \p X_{\e,b}$. Since visual metrics on $\p X$ based at $z$ with parameter $\e$ correspond to visual metrics on $\p_{\omega}Y$ based at $\omega$ with the same parameter $\e$, it then follows from Theorem \ref{identification theorem} that the restriction of $d_{\e,b}$ to $\p X_{\e,b}$ defines a visual metric on $\p_{\omega}X$ with parameter $\e$ and comparison constant $L = L(\delta,K,\e,M)$. This completes the proof of Theorem \ref{identification theorem}. 

For Theorem \ref{unbounded uniformization} we observe that any geodesic $\sigma$ in $X$ is also a geodesic in $Y$. Since $X_{\e,b}$ sits isometrically inside of $Y_{\e,\t{b}}$ and $\p X_{\e,b} = \p Y_{\e,\t{b}}$, by applying Theorem \ref{unbounded uniformization} to $Y$ we conclude that $\sigma$ is an $A$-uniform curve in $Y_{\e,\t{b}}$ and therefore also an $A$-uniform curve in $X_{\e,b}$. This completes the proof of Theorem \ref{unbounded uniformization} in this case.
\end{proof}

\begin{rem}\label{transfer estimates}
The proof of Proposition \ref{bound from unbound} also shows that the estimates of Lemmas \ref{lem:estimate both} and \ref{compute distance} hold for $b \in \mathcal{D}(X)$ as well. In the notation of the proof, this is because $\t{b}|_{X} = b$, $X_{\e,b}$ isometrically embeds into $Y_{\e,\t{b}}$, and we have an identification of boundaries $\p X_{\e,b} = \p Y_{\e,\t{b}}$. 
\end{rem}

\section{Hyperbolic fillings}\label{sec:filling} 
Let $(Z,d)$ be a metric space and let $\alpha,\tau > 1$ be given parameters. We recall the construction of a hyperbolic filling $X$ of $Z$ with these parameters described prior to Theorem \ref{filling theorem}. For each $n \in \Z$ we select a maximal $\alpha^{-n}$-separated subset $S_{n}$ of $Z$. The existence of such a set is guaranteed by a standard application of Zorn's lemma. Then for each $n \in \Z$ the balls $B(z,\alpha^{-n})$, $z \in S_{n}$, cover $Z$.

The vertex set of $X$ has the form 
\[
V = \bigcup_{n \in \Z} V_{n}, \;\;\; V_{n} = \{(x,n):x \in S_{n}\}.
\]
To each vertex $v = (x,n)$ we associate the dilated ball $B(v) = B(x,\tau \alpha^{-n})$. We will often use $v$ to denote both a vertex in $X$ and its associated point in $Z$. We also define the \emph{height function} $h: V \rightarrow \Z$ by $h(x,n) = n$. By construction for each $z \in Z$ there is a $v \in V_{n}$ such that $\rho(v,z) < \alpha^{-n}$. 
 
We place an edge in $X$ between distinct vertices $v$ and $w$ if and only if $|h(v)-h(w)| \leq 1$ and $B(v) \cap B(w) \neq \emptyset$. Thus there is an edge between vertices if and only if they are of the same or adjacent height and there is a nonempty intersection of their associated balls. For vertices $v,w$ we write $v \sim w$ if there is an edge between $v$ and $w$.  Edges between vertices of the same height are referred to as \emph{horizontal}, and edges between vertices of different heights are called \emph{vertical}. We say that an edge path between two vertices is \emph{vertical} if it is composed exclusively of vertical edges. 

While we will allow any choice of $\alpha > 1$, we will need to place some constraints on the values of the parameter $\tau$ based on $\alpha$. We will require that
\begin{equation}\label{tau requirement}
\tau > \max\left\{3,\frac{\alpha}{\alpha-1}\right\}.
\end{equation}
We give each connected component of $X$ the unique geodesic metric in which all edges have unit length. The restriction \eqref{tau requirement} will be used in Proposition \ref{connected filling} to show that $X$ is actually connected and is therefore a geodesic metric space itself. For applications a standard choice of parameters satisfying \eqref{tau requirement} is given by $\alpha = 2$ and $\tau = 4$.

\begin{rem}\label{idk}We do not know whether the constraint \eqref{tau requirement} can be relaxed while preserving the properties of $X$ described below. In particular we do not know whether $X$ is always Gromov hyperbolic or even connected for all $\tau > 1$. However, by applying Lemma \ref{height connection} below it is easy to see that $X$  is connected for any $\tau > 1$ when $Z$ is bounded. We note that one cannot take $\tau = 1$ in the construction as it is possible for the resulting graph to fail to be Gromov hyperbolic even in the bounded case \cite[Example 8.8]{BBS21}.
\end{rem}

Since edges can only connect vertices of the same or adjacent heights, all vertical edge paths are geodesics in $X$. We will refer to these vertical paths as \emph{vertical geodesics}. We will use the generic distance notation $|xy|$ for the distance between $x,y \in X$. Thus for $v = (x,n),w=(y,n) \in V$ we will denote their distance in $X$ by $|vw|$ and their distance in $Z$ by $d(v,w):=d(x,y)$. Identifying an edge $g$ from a vertex $v$ to a vertex $w$ isometrically with $[0,1]$, we extend the height function $h$ to $g$ by $h(s) = sh(v)+(1-s)h(w)$. Then $h$ defines a function $h: X \rightarrow \R$ that is $1$-Lipschitz on the connected components of $X$.
 
We begin with a simple lemma. 

\begin{lem}\label{height connection}
Let $v,w \in V$ with $h(v) \neq h(w)$ and $B(v) \cap B(w) \neq \emptyset$. Then there is a vertical edge path from $v$ to $w$.
\end{lem}

\begin{proof}
Let $v = (x,m)$, $w = (y,n)$, and let $z \in B(v) \cap B(w)$. We can assume without loss of generality that $m < n$. For each integer $m \leq k \leq n$ we can find a vertex $v_{k} \in V_{k}$ with $z \in B(v_{k})$; we set $v_{m} = v$ and $v_{n} = w$. Then $v_{k} \sim v_{k+1}$ for each $m \leq k < n$ by the construction of the graph $X$.  It follows that $v$ is connected to $w$ by a vertical edge path passing through the vertices $v_{k}$. 
\end{proof}

The next lemma estimates the distance in $Z$ between vertices in $X$ that are connected by a vertical edge path. 

\begin{lem}\label{geometric series}
Let $v, w \in V$. Suppose that $v$ is joined to $w$ by a vertical edge path and $h(v) \leq h(w)$. Then
\[
d(v,w) \leq \frac{2\tau \alpha^{-h(v)+1}}{\alpha-1}.
\]
\end{lem}

\begin{proof}
We first derive a sharper inequality in the case $h(w) = h(v)+1$. Set $h(v) = m$. Let $x \in B(v) \cap B(w)$. Then 
\[
d(v,w) \leq d(x,v) + d(x,w) < \tau \alpha^{-m} + \tau \alpha^{-m-1} < 2\tau \alpha^{-m}.
\]
Now let $h(v) = m$, $h(w) = n$. For each $m \leq k \leq n$ we let $v_{k} \in V_{k}$ be the vertex at this height in the vertical edge path joining $v$ to $w$. Then by the ``$h(w) = h(v) + 1$" case we have
\[
d(v,w) \leq \sum_{k=m}^{n-1}d(v_{k},v_{k+1}) \leq 2\tau \alpha^{-m}\sum_{k=0}^{n-m-1}\alpha^{-k} \leq \frac{2\tau \alpha^{-m+1}}{\alpha-1},
\] 
with the final inequality following by summing the geometric series in $\alpha^{-1}$. 
\end{proof}

Following the hyperbolic filling construction in \cite{BS07}, we define a \emph{cone point} $u \in V$ for a pair of vertices $\{v,w\} \subseteq V$ to be a vertex that can be joined to both $v$ and $w$ by vertical geodesics and that satisfies $h(u) \leq \min\{h(v),h(w)\}$. A \emph{branch point} for $\{v,w\}$ is defined to be a cone point of maximal height. A branch point for $\{v,w\}$ always exists as long as there is at least one cone point for $\{v,w\}$. When $v = w$ the vertex $v$ is trivially a branch point for the set $\{v\}$. 

\begin{lem}\label{cone adjacent}
Let $v,w \in V_{n}$ be distinct vertices with $v \sim w$. Then there is a branch point $u \in V_{n-1}$ for the set $\{v,w\}$. 
\end{lem}

\begin{proof}
The assumptions imply that $B(v) \cap B(w) \neq \emptyset$. Let $z \in B(v)\cap B(w)$ be a point in this intersection. Since $V_{n-1}$ is a maximal $\alpha^{-n+1}$-separated set in $Z$ we can find $u \in V_{n-1}$ such that $d(u,z) < \alpha^{-n+1}$. We compute
\[
d(v,u) \leq d(v,z) + d(z,u) < \tau \alpha^{-n} + \alpha^{-n+1} < \tau \alpha^{-n+1},
\]
by inequality \eqref{tau requirement}, noting that the final inquality here is equivalent to
\[
\tau  + \alpha < \tau \alpha,
\]
which is implied by \eqref{tau requirement}. It follows that $v \in B(u)$ and therefore $B(v) \cap B(u) \neq \emptyset$. Thus $v$ is joined to $u$ by a vertical edge. Since the roles of $v$ and $w$ are symmetric, we conclude by the same calculation that $B(w) \cap B(u) \neq \emptyset$, i.e., $w$ is also joined to $u$ by a vertical edge. Thus $u$ is a cone point for $\{v,w\}$. Since a cone point for a pair of distinct vertices on an adjacent level is trivially maximal, we conclude that $u$ is a branch point for $\{v,w\}$. 
\end{proof}

We can now show that the graph $X$ is connected. 

\begin{prop}\label{connected filling}
For each $v,w \in V$ there is a branch point $u$ for the set $\{v,w\}$ that satisfies 
\begin{equation}\label{branch comparison}
\alpha^{-h(u)} \asymp_{C(\alpha,\tau)} d(v,w) + \alpha^{-\min\{h(v),h(w)\}}.
\end{equation}
Consequently the graph $X$ is connected.
\end{prop}

\begin{proof}
Let $v \in V_{m}$, $w \in V_{n}$ be given. We can assume without loss of generality that $m \leq n$. If $v = w$ then $v$ is a branch point for the set $\{v\}$ and the comparison \eqref{branch comparison} holds trivially. We can thus assume for the rest of the proof that $v \neq w$. We let $k \in \Z$ be any integer satisfying $\alpha^{-k} > d(v,w)$ and $k\leq m$; note that such an integer always exists since $\alpha^{-k} \rightarrow \infty$ as $k \rightarrow -\infty$. Let $p \in V_{k}$ be a vertex such that $d(v,p) < \alpha^{-k}$ and let $q \in V_{k}$ be a vertex such that $d(q,w) < \alpha^{-k}$. Then
\[
d(p,q) \leq d(v,p) + d(v,w) + d(w,q) < 3\alpha^{-k} < \tau \alpha^{-k},
\]
by \eqref{tau requirement}. Thus $q \in B(p)$, so $B(p) \cap B(q) \neq \emptyset$.  We conclude that $p \sim q$. By Lemma \ref{cone adjacent} we can then find a branch point $x \in V_{k-1}$ for the set $\{p,q\}$. Since $B(p) \cap B(v) \neq \emptyset$ and $B(q) \cap B(w) \neq \emptyset$, Lemma \ref{height connection} shows that $p$ and $q$ are connected to $v$ and $w$ respectively by vertical edge paths, and the requirement $k \leq m$ implies that $\max\{h(p),h(q)\} \leq \min\{h(v),h(w)\}$. Since $p$ and $q$ are each connected to $x$ by a vertical edge, we conclude that $x$ is a cone point for the set $\{v,w\}$. 

It follows that there is a branch point $u$ for the set $\{v,w\}$. Since $u$ is joined to $v$ and $w$ by vertical edge paths, the triangle inequality and Lemma \ref{geometric series} implies that
\[
d(v,w) \leq 2\max\{d(v,u),d(w,u)\} \leq C(\alpha,\tau)\alpha^{-h(u)}.  
\]
Since $h(u) \leq m$, we have $\alpha^{-m} \leq \alpha^{-h(u)}$ and therefore
\[
d(v,w) + \alpha^{-m}  \leq C(\alpha,\tau)\alpha^{-h(u)} + \alpha^{-m}  \leq C(\alpha,\tau)\alpha^{-h(u)},
\]
which gives the lower bound in the comparison \eqref{branch comparison}. 

For the upper bound in \eqref{branch comparison} we split into two cases. The first case is that in which $B(v) \cap B(w) \neq \emptyset$. If $h(v) = h(w)$ then this implies that $v \sim w$ and Lemma \ref{cone adjacent} implies that $h(u) = h(v)-1$. The upper bound follows immediately from this, as we then have
\[
d(v,w) + \alpha^{-\min\{h(v),h(w)\}} \geq \alpha^{-h(v)} = \alpha^{-1}\alpha^{-h(u)}. 
\]
If $h(v) \neq h(w)$ then by Lemma \ref{height connection} $v$ can be joined to $w$ by a vertical edge path. In this case $v$ is a branch point for the set $\{v,w\}$ and the inequality 
\[
d(v,w) +\alpha^{-h(v)} \geq \alpha^{-h(u)},
\]
holds trivially for $u = v$. 

The second case is that in which have $B(v) \cap B(w) = \emptyset$. This implies in particular that we must have $w \notin B(v)$. Thus $d(v,w) \geq \tau \alpha^{-m} > 0$. Let $k \in \Z$ be the maximal integer such that $k \leq m$ and $\alpha^{-k} > d(v,w)$. Then either $k = m$ or $d(v,w) \geq \alpha^{-k-1}$. Since $\alpha^{-k} > d(v,w)$ and $d(v,w) \geq \tau \alpha^{-m}$, we conclude in both cases that $d(v,w) \asymp_{C(\alpha,\tau)} \alpha^{-k}$. Making this choice of $k$ in the construction of $x$ above, we can thus construct a cone point $x$ for the set $\{v,w\}$ with $h(x) = k-1$ and therefore
\[
\alpha^{-h(x)} \asymp_{C(\alpha,\tau)} d(v,w). 
\]
Since the branch point $u$ satisfies $h(u) \geq h(x)$ it follows that
\[
\alpha^{-h(u)} \leq C(\alpha,\tau)d(v,w) \leq C(\alpha,\tau) (d(v,w) + \alpha^{-m}). 
\]
The upper bound in \eqref{branch comparison} follows.

Lastly, since we can connect $v$ to $w$ through the branch point $u$, it follows that $v$ and $w$ can be connected by an edge path in the graph $X$. Since $v$ and $w$ were arbitrary we conclude that $X$ is connected. 
\end{proof}

Now that we've shown $X$ is connected, the metrics we put on its connected components give it the structure of a geodesic metric space in which all edges of $X$ have unit length. The height function then defines a $1$-Lipschitz function $h: X \rightarrow \R$. We formally define the Gromov product based at $h$ by, for $x,y \in X$,
\[
(x|y)_{h} = \frac{1}{2}(h(x)+h(y)-|xy|). 
\]
Since $h$ is $1$-Lipschitz we have
\begin{equation}\label{lip height}
(x|y)_{h} \leq \min\{h(x),h(y)\}.
\end{equation}
Our next lemma gives a key relation of the Gromov product based at $h$ to branch points. 

\begin{lem}\label{branch estimate}
Let $v,w \in V$ and let $u$ be a branch point for $\{v,w\}$. Then 
\[
h(u) \doteq_{c(\alpha,\tau)} (v|w)_{h},
\]
and therefore 
\[
\alpha^{-(v|w)_{h}} \asymp_{C(\alpha,\tau)} d(v,w)+\alpha^{-\min\{h(v),h(w)\}}. 
\]
\end{lem}

\begin{proof}
Since the claims of the lemma hold trivially when $v = w$ we can assume that $v \neq w$. Proposition \ref{connected filling} gives the existence of a branch point $u$ for $\{v,w\}$ satisfying \eqref{branch comparison}. The vertical edge path from $v$ to $u$ followed by the vertical edge path from $u$ to $w$ gives an edge path from $v$ to $w$, which shows that
\[
|vw| \leq |vu| + |uw| = h(v)-h(u) + h(w)-h(u) = h(v)+h(w)-2h(u).
\]
Rearranging this we obtain
\[
h(u) \leq \frac{1}{2}(h(v)+h(w)-|vw|) = (v|w)_{h}.
\]

To get a bound in the other direction, let $v = v_{0},v_{1},\dots,v_{k} = w$ be a sequence of vertices joined by edges that gives a geodesic $\gamma$ from $v$ to $w$. Then $|vw| = k$ and $k \geq 1$ since $v \neq w$. For $1 \leq i \leq k$ we have $B(v_{i-1}) \cap B(v_{i}) \neq \emptyset$ and therefore, using $|h(v_{i-1})-h(v_{i})| \leq 1$, 
\[
d(v_{i-1},v_{i}) < 2\tau \alpha^{-\min\{h(v_{i-1}),h(v_{i})\}} \leq 2\tau \alpha^{-h(v_{i-1})+1}.
\]
We can run the same argument viewing $\gamma$ as a geodesic from $w$ to $v$ instead, setting $w_{i} = v_{k-i}$ for $0 \leq i \leq k$. We see from this that we also have
\[
d(w_{i-1},w_{i}) < 2\tau  \alpha^{-h(w_{i-1})+1},
\]
for $1 \leq i \leq k$. For each $1 \leq l \leq k$ we thus obtain an estimate (using $h(v_{i-1}) \geq h(v)-i+1$ and $h(w_{i-1}) \geq h(w)- i+ 1$), 
\begin{align*}
d(v,w) &\leq 
\sum_{i=1}^{k}d(v_{i-1},v_{i}) \\
&\leq \sum_{i=1}^{l}d(v_{i-1},v_{i}) + \sum_{i=1}^{k-l+1}d(w_{i-1},w_{i}) \\
&< 2\tau \alpha^{-h(v)}\sum_{i=1}^{l} \alpha^{i} + 2\tau \alpha^{-h(w)}\sum_{i=1}^{k-l+1} \alpha^{i} \\
&\leq \frac{2\tau \alpha}{\alpha-1}(\alpha^{-h(v)}(\alpha^{l}-1) + \alpha^{-h(w)}(\alpha^{k-l+1}-1)) \\
&\leq \frac{2\tau \alpha}{\alpha-1}(\alpha^{l-h(v)}+ \alpha^{k-l+1-h(w)}). 
\end{align*}
We set $l = \lceil\frac{1}{2}(k-h(w)+h(v)) \rceil$ (the least integer greater than this quantity), observing that $1 \leq l \leq k$ since $|h(v)-h(w)| \leq k$. This gives, after some simplification,
\[
d(v,w) \leq C(\alpha,\tau)\alpha^{-\frac{1}{2}(h(v)+h(w)-k)} = C(\alpha,\tau) \alpha^{-(v|w)_{h}},
\]
recalling that $|vw| = k$. By Proposition \ref{connected filling} and inequality \eqref{lip height}, we then have
\[
\alpha^{-h(u)} \leq C(\alpha,\tau) \alpha^{-(v|w)_{h}},
\]
which implies upon taking logarithms that
\[
h(u) \geq (v|w)_{h}-c(\alpha,\tau). 
\]
This gives the desired lower bound of the first approximate equality of the lemma. The second comparison inequality follows by using Proposition \ref{connected filling} again. 
\end{proof}

We now prove an inequality similar to the $4\delta$-inequality \eqref{delta inequality} for our formal Gromov products based at $h$. 

\begin{lem}\label{delta inequality filling}
Let $u,v,w \in V$. Then 
\[
(u|w)_{h} \geq \min\{(u|v)_{h},(v|w)_{h}\} - c(\alpha,\tau).
\]
\end{lem}

\begin{proof}
Let $u,v,w \in V$ be vertices. By the triangle inequality in $Z$ we have 
\[
d(u,w) + \alpha^{-\min\{h(u),h(w)\}} \leq d(u,v) + \alpha^{-\min\{h(u),h(v)\}}  + d(v,w) + \alpha^{-\min\{h(v),h(w)\}},
\]
which becomes, upon applying Lemma \ref{branch estimate}, 
\begin{align*}
\alpha^{-(u|w)_{h}} &\leq C(\alpha,\tau) (\alpha^{-(u|v)_{h}} + \alpha^{-(v|w)_{h}}) \\
&\leq C(\alpha,\tau)\alpha^{-\min\{(u|v)_{h},(v|w)_{h}\}}. 
\end{align*}
Taking logarithms of each side gives the desired inequality. 
\end{proof}

We can now show that $X$ is Gromov hyperbolic. For this we use some terminology from \cite[Chapter 2]{BS07}: a \emph{$\delta$-triple} for $\delta \geq 0$ is a triple $(a,b,c)$ of real numbers $a,b,c$ such that the two smallest numbers differ by at most $\delta$. Observe that $(a,b,c)$ is a $\delta$-triple if and only if the inequality 
\begin{equation}\label{basic delta}
c \geq \min\{a,b\} - \delta,
\end{equation}
holds for all permutations of the roles of $a$, $b$, and $c$. We will also need the following standard claim \cite[Lemma 2.1.4]{BS07} which is called the \emph{Tetrahedron lemma}. 

\begin{lem}\label{tetrahedron}
Let $d_{12}$, $ d_{13}$, $d_{14}$, $d_{23}$, $d_{24}$, $d_{34}$ be six numbers such that the four triples $(d_{23},d_{24},d_{34})$, $(d_{13},d_{14},d_{34})$, $(d_{12},d_{14},d_{24})$, and $(d_{12},d_{13},d_{23})$ are $\delta$-triples. Then 
\[
(d_{12}+d_{34},d_{13}+d_{24},d_{14}+d_{23})
\]
is a $2\delta$-triple.
\end{lem}

\begin{prop}\label{hyperbolicity filling}
The space $X$ is $\delta$-hyperbolic with $\delta = \delta(\alpha,\tau)$. 
\end{prop}

\begin{proof}
We will use the \emph{cross-difference triple} defined in \cite[Chapter 2.4]{BS07}. For a quadruple of points $Q=(x,y,z,u) \in X$ and a fixed basepoint $o \in X$ this triple is defined by
\[
A_{o}(Q)= ((x|y)_{o} + (z|u)_{o},(x|z)_{o}+(y|u)_{o},(x|u)_{o}+(y|z)_{o}). 
\]
The triple $A_{o}(Q)$ has the same differences among its members as the triple 
\[
A_{h}(Q) = ((x|y)_{h} + (z|u)_{h}, (x|z)_{h}+(y|u)_{h}, (x|u)_{h} + (y|z)_{h}),
\]
as a routine calculation shows for instance that
\[
(x|y)_{o} + (z|u)_{o} - (x|z)_{o}-(y|u)_{o} = (x|y)_{h} + (z|u)_{h} - (x|z)_{h}-(y|u)_{h} ,
\]
with both expressions being equal to
\[
\frac{1}{2}(-|xy|-|zu|+|xz|+|yu|). 
\]
Similar calculations give equality for the other differences. Thus $A_{o}(Q)$ is a $\delta$-triple for a given $\delta \geq 0$ if and only if $A_{h}(Q)$ is a $\delta$-triple. 

Using Lemma \ref{delta inequality filling} we conclude that the six numbers $(x|y)_{h}$, $(z|u)_{h}$, $(x|z)_{h}$, $(y|u)_{h}$, $(x|u)_{h}$, $(y|z)_{h}$ together satisfy the hypotheses of Lemma \ref{tetrahedron} with parameter $\delta = \delta(\alpha,\tau)$. This implies that $A_{h}(Q)$ is a $2\delta$-triple and therefore that $A_{o}(Q)$ is a $2\delta$-triple. By \cite[Proposition 2.4.1]{BS07} this implies that inequality \eqref{delta inequality} holds for Gromov products based at $o$ in $X$ (with $2\delta$ replacing $4\delta$). By  \cite[Chapitre 2, Proposition 21]{GdH90} this implies that geodesic triangles in $X$ are $8\delta$-thin, i.e., $X$ is $8\delta$-hyperbolic. 
\end{proof}

We next show that any vertex in $V$ is part of a vertical geodesic line. We will in fact show something stronger. We let $\bar{Z}$ denote the completion of $Z$, and continue to write $d$ for the canonical extension of the metric on $Z$ to its completion. For $r > 0$ and a point $z \in \bar{Z}$ we will write $B'(z,r)$ for the open ball of radius $r$ centered at $z$ in the completion $\bar{Z}$. 

\begin{lem}\label{second height connection}
Let $z \in \bar{Z}$. Then there is a vertical geodesic $\gamma: \R \rightarrow X$ with $h(\gamma(t)) = t$ for $t \in \R$ such that, writing $\gamma(n) = (z_{n},n)$ for $n \in \Z$, we have $z \in B'(z_{n},\frac{\tau}{3}\alpha^{-n})$ for each $n \in \Z$. Furthermore if $v = (z,m)$ is a given vertex of $V$ then we can construct $\gamma$ such that $\gamma(m) = v$. 
\end{lem}

\begin{proof}
Since $\frac{\tau}{3} > 1$ by \eqref{tau requirement} and since for each $n  \in \Z$ the balls $B(y,\alpha^{-n})$ cover $Z$ for $y \in S_{n}$, it follows from the fact that $Z$ is dense in $\bar{Z}$ that the balls $B'(y,\frac{\tau}{3}\alpha^{-n})$ for $y \in S_{n}$ cover $\bar{Z}$. Thus, given $z \in \bar{Z}$, for each $n \in \Z$ we can find $z_{n} \in S_{n}$ such that $z \in B'(z_{n},\frac{\tau}{3}\alpha^{-n})$. 

Let $v_{n} = (z_{n},n)$ be the associated vertex in $V$. We claim that for each $n \in \Z$ we have $B(v_{n}) \cap B(v_{n+1}) \neq \emptyset$. Since $Z$ is dense in $\bar{Z}$ we can find $y \in Z$ such that $d(y,z) < \frac{\tau}{3}\alpha^{-n-1}$. Then 
\[
d(y,z_{n+1}) \leq d(y,z) + d(z,z_{n+1}) < \frac{\tau}{3}\alpha^{-n-1} + \frac{\tau}{3}\alpha^{-n-1} < \tau \alpha^{-n-1},
\]
which implies that $y \in B(v_{n+1})$. A similar calculation shows that $y \in B(v_{n})$ since $\alpha^{-n-1} < \alpha^{-n}$. Thus $B(v_{n}) \cap B(v_{n+1}) \neq \emptyset$ and therefore $v_{n} \sim v_{n+1}$. We can then find a vertical geodesic $\gamma: \R \rightarrow X$ through the sequence of vertices $\{v_{n}\}_{n \in \Z}$, which can be parametrized such that $h(\gamma(t)) = t$ for $t \in \R$. Finally, if $v = (z,m)$ is a vertex of $V$ then we can choose $z_{m} = z$ in our construction since we trivially have $z \in B'(z,\frac{\tau}{3}\alpha^{-m})$. 
\end{proof}

A \emph{descending geodesic ray} $\gamma:[0,\infty) \rightarrow X$ is a vertical geodesic ray such that $h(\gamma(t))$ is strictly decreasing as a function of $t$. In this case we have $h(\gamma(t)) = h(\gamma(0)) - t$ for each $t \geq 0$. Similarly an \emph{ascending geodesic ray} $\gamma:[0,\infty) \rightarrow X$ is a vertical geodesic ray such that $h(\gamma(t))$ is strictly increasing as a function of $t$. In this case we instead have that $h(\gamma(t)) = h(\gamma(0)) + t$ for each $t \geq 0$. A vertical geodesic $\gamma$ is \emph{anchored} at a point $z \in \bar{Z}$ if for each vertex $(z_{m},m)$ belonging to $\gamma$ we have $z \in B'(z_{m},\frac{\tau}{3}\alpha^{-m})$; when the point $z$ does not need to be referenced we will just say that $\gamma$ is \emph{anchored}. Lemma \ref{second height connection} gives the existence of ascending and descending geodesic rays in $X$ anchored at any point $z \in \bar{Z}$. 

We will next show that all anchored descending vertical geodesic rays in $X$ define the same point in the Gromov boundary $\p X$. 

\begin{lem}\label{bounded distance vertical}
Let $\gamma$, $\sigma: [0,\infty) \rightarrow X$ be two descending geodesic rays in $X$ starting at vertices $v = \gamma(0)$ and $w = \sigma(0)$ of $X$ respectively and anchored at $y,z \in \bar{Z}$ respectively. Let $k \in \Z$ be such that $k \leq \min\{h(v),h(w)\}$ and $\frac{\tau}{3}\alpha^{-k} > d(y,z)$. Let $v_{k} \in \gamma \cap V_{k}$, $w_{k} \in \sigma \cap V_{k}$ be the vertices on these geodesics at the height $k$. Then $|v_{k}w_{k}| \leq 1$. 
\end{lem}

\begin{proof}
By the anchoring condition we have $d(v_{k},y) < \frac{\tau}{3}\alpha^{-k}$ and $d(w_{k},z) < \frac{\tau}{3}\alpha^{-k}$. Hence
\[
d(v_{k},w_{k}) \leq d(v_{k},y) + d(y,z) + d(z,w_{k}) < \tau \alpha^{-k}.
\]
Thus $w_{k} \in B(v_{k})$ and therefore $B(v_{k}) \cap B(w_{k}) \neq \emptyset$, which implies that either $v_{k} = w_{k}$ or $v_{k} \sim w_{k}$. In both cases we conclude that $|v_{k}w_{k}| \leq 1$. 
\end{proof}

The Busemann functions associated to anchored descending geodesic rays have a particularly simple form. 

\begin{lem}\label{height busemann} 
Let $\gamma$ be an anchored descending geodesic ray in $X$ starting from a vertex $v \in V$. Then for all $x \in X$ we have 
\begin{equation}\label{height busemann equation}
b_{\gamma}(x) \doteq_{3} h(x)-h(\gamma(0)).
\end{equation}
\end{lem}

\begin{proof}
Since both $b_{\gamma}$ and $h$ are $1$-Lipschitz and the edges of $X$ have unit length, it suffices to prove the estimate \eqref{height busemann equation} on the vertices of $X$ with the constant $1$ instead of $3$. Let $z \in \bar{Z}$ be the anchoring point for $\gamma$. Let $w \in V$ be an arbitrary  vertex and let $\sigma: [0,\infty) \rightarrow X$ be an anchored descending geodesic ray in $X$ starting at $w$ and anchored at the point $y \in Z$ associated to $w$, as constructed in Lemma \ref{second height connection}.  The sequences of vertices $\{\gamma(n)\}_{n=0}^{\infty}$ on $\gamma$ satisfies $h(\gamma(n)) = h(\gamma(0)) - n$ since $\gamma$ is a descending geodesic ray, and the same holds for $\sigma$. We let $n$ be any integer large enough that $h(\gamma(n)) \leq h(w)$ and $\frac{\tau}{3}\alpha^{-h(\gamma(n))} > d(y,z)$ and observe that if we define $k_{n} = h(w)-h(\gamma(n))$ then $h(\sigma(k_{n})) = h(\gamma(n))$. Then by Lemma \ref{bounded distance vertical} we have $|\gamma(n)\sigma(k_{n})| \leq 1$. Since $\sigma(k_{n})$ is joined to $w$ by a vertical geodesic of length $k_{n}$, it follows immediately that
\[
h(w)-h(\gamma(n)) \leq |\gamma(n)w| \leq h(w)-h(\gamma(n))+1,
\]
and therefore, since $h(\gamma(n)) = h(\gamma(0)) -n$, 
\[
|\gamma(n)w|  \doteq_{1} h(w)+ n - h(\gamma(0)). 
\]
This implies that
\[
|\gamma(n)w|  - n \doteq_{1} h(w)-h(\gamma(0)).
\]
By letting $n \rightarrow \infty$ we conclude that
\[
b_{\gamma}(w) \doteq_{1} h(w) - h(\gamma(0)),
\]
which gives the desired result. 
\end{proof}

In particular, for an anchored descending geodesic ray $\gamma$ with $h(\gamma(0)) = 0$, Lemma \ref{height busemann} shows that $b_{\gamma} \doteq_{3} h$. We fix such a descending geodesic ray $\gamma$ for the remainder of this section and write $b:=b_{\gamma}$ for the associated Busemann function. Let $\omega \in \p X$ be the point corresponding to the equivalence class of $\gamma$ in the Gromov boundary of $X$; note that Lemma \ref{bounded distance vertical} shows that all anchored descending geodesic rays belong to the equivalence class $\omega$ defined by $\gamma$. Our final goal in this section is to show that the boundary $\p_{\omega}X$ of $X$ relative to $\omega$ can be canonically identified with the completion $\bar{Z}$ of $Z$ in such a way that the extension of the metric $d$ to $\bar{Z}$ is a visual metric on $\p_{\omega}X$ based at $\omega$ with parameter $\log \alpha$.

We remark that the rough equality $b \doteq_{3} h$ implies that $(x|y)_{b} \doteq_{3} (x|y)_{h}$ for all $x,y \in X$ as well, so that in particular the conclusions of Lemma \ref{connected filling} hold with $b$ replacing $h$ and $(v|w)_{b}$ replacing $(v|w)_{h}$ everywhere, at the cost of adding $6$ to the constant $c(\alpha,\tau)$ there and multiplying the constant $C(\alpha,\tau)$ by $\alpha^{6}$. We will use this observation without further comment below. 

For each point $z \in \bar{Z}$ we fix an ascending geodesic ray $\gamma_{z}:[0,\infty) \rightarrow X$ anchored at $z$, as given by Lemma \ref{second height connection}. We define a map $\psi: \bar{Z} \rightarrow \p_{\omega} X$ by setting $\psi(z) = [\gamma_{z}]$, i.e., $\psi(z)$ is the equivalence class in $\p_{\omega}X$ defined by the geodesic ray $\gamma_{z}$. Implicit in this definition is the fact that we cannot have $[\gamma_{z}] = \omega$ for any $z \in \bar{Z}$, as it is easy to see from the fact that $h$ is $1$-Lipschitz that ascending geodesic rays cannot be at bounded distance from descending geodesic rays. We also note that if $\gamma$ is any other ascending geodesic ray anchored at $z \in \bar{Z}$ then we must have $[\gamma] = [\gamma_{z}]$: for $k \in \Z$ sufficiently large the unique vertices $v_{k} \in \gamma \cap V_{k}$ and $w_{k} \in \gamma_{z} \cap V_{k}$ must satisfy $|v_{k}w_{k}| \leq 1$ since $z \in B(v_{k}) \cap B(w_{k})$, hence the geodesic rays $\gamma$ and $\gamma_{z}$ are at a bounded distance from one another. The map $\psi$ can thus equivalently be thought of as sending $z \in \bar{Z}$ to the equivalence class of all ascending geodesic rays anchored at $z$.

\begin{prop}\label{identify boundary}
The map $\psi: \bar{Z} \rightarrow \p_{\omega}X$ defines an identification of $\bar{Z}$ with $\p_{\omega}X$. Under this identification the metric $d$ on $\bar{Z}$ defines a visual metric on $\p_{\omega}X$ with parameter $\log \alpha$ and comparison constant depending only on $\alpha$ and $\tau$.  
\end{prop}

\begin{proof}
Let $x,y \in Z$ be given and let $\gamma_{x}$ and $\gamma_{y}$ be ascending geodesic rays anchored at $x$ and $y$ respectively. For $n \in \Z$ sufficiently large we let $x_{n} \in \gamma_{x} \cap V_{n}$ and $y_{n} \in \gamma_{y} \cap V_{n}$ be the unique vertices on these rays at height $n$. By Lemma \ref{branch estimate} we have
\begin{equation}\label{branch use}
\alpha^{-(x_{n}|y_{n})_{b}} \asymp_{C(\alpha,\tau)} d(x_{n},y_{n}) + \alpha^{-n}. 
\end{equation}
Since $x \in B(x_{n})$ and $y \in B(y_{n})$ we have $d(x,x_{n}) < \tau \alpha^{-n}$ and $d(y,y_{n}) < \tau \alpha^{-n}$. Hence, by letting $n \rightarrow \infty$ in \eqref{branch use} and using Lemma \ref{busemann inequality}, we conclude that 
\begin{equation}\label{phi visual}
\alpha^{-(\psi(x)|\psi(y))_{b}} \asymp_{C(\alpha,\tau)} d(x,y). 
\end{equation}
It follows immediately that $\psi: \bar{Z} \rightarrow \p_{\omega}X$ is injective. To complete the proof of the proposition it suffices to show that $\psi$ is surjective, as the estimate \eqref{phi visual} then shows that the metric $d$ on $\bar{Z}$ defines a visual metric on $\p_{\omega}X$ with parameter $\log \alpha$ and comparison constant depending only on $\alpha$ and $\tau$ when we use $\psi$ to identify $\bar{Z}$ with $\p_{\omega}X$.

We recall from Proposition \ref{convergence Busemann} that $\p_{\omega}X$ can be defined as equivalence classes of sequences $\{x_{n}\}$ in $X$ such that $(x_{m}|x_{n})_{b} \rightarrow \infty$ as $m,n \rightarrow \infty$, with two sequences $\{x_{n}\}$, $\{y_{n}\}$ being equivalent if $(x_{n}|y_{n})_{b} \rightarrow \infty$ as $n \rightarrow \infty$. Since $b$ is $1$-Lipschitz we can always choose these sequences to consist of vertices in $X$ by replacing $x_{n}$ with a nearest vertex $v_{n}$.

Thus let $\{v_{n}\}$ be a sequence of vertices defining a point $\xi$ of $\p_{\omega}X$. Let $\{z_{n}\}$ be the associated sequence of points in $Z$. By Lemma \ref{branch estimate} we have 
\[
\alpha^{-(v_{n}|v_{m})_{b}} \asymp_{C(\alpha,\tau)} d(z_{n},z_{m}) + \alpha^{-\min\{b(v_{n}),b(v_{m})\}}. 
\]
Since $(v_{n}|v_{m})_{b} \rightarrow \infty$ and $b(v_{n}) \rightarrow \infty$, it follows immediately that $\{z_{n}\}$ is a Cauchy sequence in $Z$ and therefore converges to a point $z \in \bar{Z}$. We claim that $\psi(z) = \xi$. 

Let $\gamma_{z}$ be an ascending geodesic ray anchored at $z$ and let $\{w_{n}\}$ be the sequence of vertices on $\gamma_{z}$ starting from its initial point. When considered as points of $Z$ this sequence of vertices must satisfy $d(w_{n},z) \rightarrow 0$ since $\gamma_{z}$ is ascending and anchored at $z$. This implies that $d(w_{n},z_{n}) \rightarrow 0$. Since by Lemma \ref{branch estimate} we have
\[
\alpha^{-(v_{n}|w_{n})_{b}} \asymp_{C(\alpha,\tau)} d(z_{n},w_{n}) + \alpha^{-\min\{b(v_{n}),b(w_{n})\}},
\]
we conclude that $(v_{n}|w_{n})_{b} \rightarrow \infty$ as $n \rightarrow \infty$. Hence $\{v_{n}\}$ and $\{w_{n}\}$ define the same point of $\p_{\omega}X$, i.e., $\psi(z) = \xi$. We conclude that $\psi$ is surjective. 
\end{proof}

\section{Uniformizing the hyperbolic filling}\label{sec:uniform filling}
This final section is devoted to proving Theorem \ref{filling theorem}. We retain all hypotheses and notation from the previous section. In particular we let $(Z,d)$ be a metric space and let $X$ be a hyperbolic filling of $Z$ with parameters $\alpha > 1$ and $\tau > \min\{3,\alpha/(\alpha-1)\}$ as in the previous section. We let $h: X \rightarrow \R$ be the height function and set $\rho(x) = \alpha^{-h(x)}$. We write $X_{\rho}$ for the conformal deformation of $X$ with conformal factor $\rho$, $d_{\rho}$ for the metric on $X_{\rho}$, and $\l_{\rho}$ for lengths of curves measured in the metric $d_{\rho}$. Since $h$ is $1$-Lipschitz the density $\rho$ satisfies the Harnack inequality \eqref{Harnack} with $\e = \log \alpha$. 

In the notation of Remark \ref{flexibility}, Lemma \ref{height busemann} shows that we have $h \in \mathcal{B}_{3}(X)$. Clearly $X$ is geodesic and complete, and Proposition \ref{hyperbolicity filling} shows that $X$ is $\delta$-hyperbolic with $\delta = \delta(\alpha,\tau)$. Thus to prove Theorem \ref{filling theorem} it suffices to show that $X$ is $\frac{1}{2}$-roughly starlike from $\omega$ and that $\rho$ is a GH-density with constant $M = M(\alpha,\tau)$, as it then follows for a Busemann function $b \in \mathcal{B}(X)$ with $b \doteq_{3} h$ that $\rho_{\e,b}$ is a GH-density with constant $M = M(\alpha,\tau)$ as well, where $\e = \log \alpha$. The conclusions of Theorem \ref{filling theorem} can then be derived from the fact that the metrics on $X_{\e,b}$ and $X_{\rho}$ are $\alpha^{3}$-biLipschitz to one another by the identity map on $X$.  

We first look at rough starlikeness from the distinguished point $\omega \in \p X$ corresponding to the equivalence class of all anchored descending geodesic rays in $X$.   

\begin{lem}\label{starlike filling}
The hyperbolic filling $X$ is $\frac{1}{2}$-roughly starlike from $\omega$. 
\end{lem}

\begin{proof}
Let $v \in V$ be a vertex of $X$ with associated point $z \in Z$. Let $\gamma:\R \rightarrow X$ be an ascending vertical geodesic line through $v$ that is anchored at $z$, as constructed in Lemma \ref{second height connection}, parametrized by arclength such that $\gamma(0) = v$. We put $\bar{\gamma}(t) = \gamma(-t)$ for $t \in \R$. Then $\bar{\gamma}|_{[0,\infty)}$ is an anchored descending geodesic ray and therefore belongs to the equivalence class $\omega$ by Lemma \ref{bounded distance vertical}. This shows that any vertex of $X$ lies on a geodesic line starting at $\omega$. Since any point in $X$ is within distance $\frac{1}{2}$ of some vertex, condition (1) of Definition \ref{def:rough star} follows. 

For condition (2) we use the identification of $\p_{\omega} X$ with $\bar{Z}$ from Proposition \ref{identify boundary}. Let $z \in \bar{Z}$ be given and let $\gamma: \R \rightarrow X$ be a vertical geodesic line anchored at $z$ and parametrized such that $h(\gamma(t)) = t$, as constructed in Lemma \ref{second height connection}. By the construction of the identification $\varphi$ in Proposition \ref{identify boundary} the geodesic ray $\gamma|_{[0,\infty)}$ then belongs to the equivalence class of $z$ when $z$ is considered as a point of $\p_{\omega}X$. Putting $\bar{\gamma}(t) = \gamma(-t)$ as above, we also have that $\bar{\gamma}|_{[0,\infty)}$ is a descending geodesic ray anchored at $z$ and therefore belongs to the equivalence class of $\omega$ by Lemma \ref{bounded distance vertical}. Since $z \in \bar{Z}$ was arbitrary, condition (2) follows.  
\end{proof}

We will now show that $\rho$ is a GH-density with constant $M = M(\alpha,\tau)$ depending only on $\alpha$ and $\tau$. We will do this by estimating the distance $d_{\rho}$ between points at sufficiently large scales in $X$ and then using Corollary \ref{criterion}.

\begin{lem}\label{large distance}
Let $x,y \in X$ with $|xy| \geq 1$. Then we have
\begin{equation}\label{filling large}
d_{\rho}(x,y) \asymp_{C(\alpha,\tau)} \alpha^{-(x|y)_{h}}. 
\end{equation}
Consequently $\rho$ is a GH-density with constant $M = M(\alpha,\tau)$. 
\end{lem}

\begin{proof}
The bound on $d_{\rho}(x,y)$ from above follows from Lemma \ref{epsilon geodesic} applied with $\e = \log \alpha$,  since there is a Busemann function $b \in \mathcal{B}(X)$ such that $b \doteq_{3} h$. Hence it suffices to establish the lower bound. Observe that for an edge $g$ of $X$, considered as a path between its endpoints $v$ and $w$ and assuming the orientation in which $h(v) \leq h(w)$, when $h(v) = h(w) = k$ we have 
\begin{equation}\label{no telescope}
\l_{\rho}(g) = \alpha^{-k}.
\end{equation}
On the other hand, since $B(v) \cap B(w) \neq \emptyset$ we have $d(v,w) < 2\tau \alpha^{-k}$. It follows that
\[
\l_{\rho}(g) > C(\alpha,\tau)^{-1}d(v,w)
\]
Similarly, when $h(v) = k$ and $h(w) = k+1$ we have
\begin{equation}\label{telescope}
\l_{\rho}(g) = \frac{1}{\log \alpha}(\alpha^{-k}-\alpha^{-k-1}) = \frac{1-\alpha^{-1}}{\log \alpha}\alpha^{- k},
\end{equation}
while $B(v) \cap B(w) \neq \emptyset$ implies again that $d(v,w) < 2\tau \alpha^{-k}$. Thus in this case we also have
\[
\l_{\rho}(g) > C(\alpha,\tau)^{-1}d(v,w).
\]

Now let $\gamma$ be a rectifiable curve joining $x$ to $y$. Let $v$ be the first vertex on $\gamma$ met traveling from $x$ to $y$ and let $w$ be the first vertex on $\gamma$ met traveling from $y$ to $x$. We first suppose that $v \neq w$. We then let $\sigma$ be the subcurve of $\gamma$ from $v$ to $w$ starting from this first occurrence of $v$ and ending at this last occurrence of $w$. Let $\{v_{i}\}_{i=0}^{l}$ be the sequence of vertices encountered along the path $\sigma$, noting that by assumption we have $l \geq 1$.  Then from our calculations above we have
\begin{equation}\label{one bound}
\l_{\rho}(\sigma) \geq C(\alpha,\tau)^{-1}\sum_{i=0}^{l-1}d(v_{i},v_{i+1}) \geq C(\alpha,\tau)^{-1}d(v,w).
\end{equation}
On the other hand, since $v \neq w$ the curve $\sigma$ must contain at least one full edge of $X$ with one vertex being $v$ and at least one full edge with one vertex being $w$ (these may be the same edge). Then it follows from \eqref{no telescope} and \eqref{telescope} applied to those edges that
\begin{equation}\label{two bound}
\l_{\rho}(\sigma) \geq C(\alpha,\tau)^{-1}\alpha^{-\min\{h(v),h(w)\}}. 
\end{equation}
By combining \eqref{one bound} and \eqref{two bound} and then using Lemma \ref{branch estimate}, we conclude that
\[
\l_{\rho}(\sigma) \geq C(\alpha,\tau)^{-1}\alpha^{-(v|w)_{h}}.
\]
Since $\sigma$ is a subcurve of $\gamma$ it follows that this inequality holds for $\gamma$ as well. 

Now suppose that $v = w$. Then, since $|xy| \geq 1$, either the initial segment of $\gamma$ from $x$ to $v$ or the final segment of $\gamma$ from $w$ to $y$ has length at least $\frac{1}{2}$ in $X$.  By reversing the roles of $x$ and $y$ if necessary we can assume that the initial segment $\eta$ of $\gamma$ from $x$ to $v$ has length at least $\frac{1}{2}$ in $X$. Since $\eta$ is contained entirely in a single edge of $X$ that has $v$ as a vertex, it follows that
\[
\l_{\rho}(\eta) \geq \alpha^{-h(v)-1}\l(\eta) \geq \frac{1}{2}\alpha^{-h(v)-1},
\]
with $\l(\eta)$ denoting the length of $\eta$ in $X$. Hence 
\[
\l_{\rho}(\gamma) \geq \l_{\rho}(\eta) \geq C(\alpha)^{-1}\alpha^{-(v|w)_{h}},
\]
using that $(v|w)_{h} = h(v)$ since $v = w$. This gives a similar lower bound on $\l_{\rho}(\gamma)$ in this case as well. Minimizing over all rectifiable paths $\gamma$ from $x$ to $y$ then gives in both cases that
\[
d_{\rho}(x,y) \geq C(\alpha,\tau)^{-1}\alpha^{-(v|w)_{h}} \geq C(\alpha,\tau)^{-1}\alpha^{-(x|y)_{h}},
\]
with the second inequality following from the fact that $h$ is $1$-Lipschitz and $|xv| \leq 1$, $|yw| \leq 1$. 

To conclude that $\rho$ is a GH-density we let $b$ be a Busemann function on $X$ such that $b \doteq_{3} h$ as in Lemma \ref{height busemann}. We let $d_{\e,b}$ be the distance obtained on $X$ through conformal deformation with the conformal factor
\[
\rho_{\e,b}(x) = e^{-\e b(x)} = \alpha^{-b(x)},
\] 
where $\e = \log \alpha$. Then $d_{\e,b}$ is $\alpha^{3}$-biLipschitz to $d_{\rho}$ by the identity map on $X$. Consequently the comparison \ref{filling large} holds with $d_{\e,b}$ replacing $d_{\rho}$. We can then apply Corollary \ref{criterion} to conclude that $\rho_{\e,b}$ is a GH-density with constant $M = M(\alpha,\tau)$, noting that $X$ is $\delta$-hyperbolic with $\delta = \delta(\alpha,\tau)$ and $\frac{1}{2}$-roughly starlike from $\omega$. The GH-inequality \eqref{first GH} for $\rho$ then follows immediately from the $\alpha^{3}$-biLipschitz comparison of $d_{\e,b}$ to $d_{\rho}$.
\end{proof}

This completes the proof of Theorem \ref{filling theorem} aside from the final assertion regarding the identification of $\p X_{\rho}$ with $\bar{Z}$ given by the combination of Lemma \ref{boundary identification} and Proposition \ref{identify boundary} is biLipschitz. This is shown below. 

\begin{prop}\label{bilip boundary}
The identification $\p X_{\rho} \cong \bar{Z}$ is biLipschitz with biLipschitz constant $L = L(\alpha,\tau)$ depending only on $\alpha$ and $\tau$. 
\end{prop}

\begin{proof}
We consider $\p_{\omega} X$ as equipped with the visual metric with parameter $\log \alpha$ defined by Proposition \ref{identify boundary}, which coincides with the extension of the metric $d$ on $Z$ to the completion $\bar{Z}$ under the identification $\psi: \bar{Z} \rightarrow \p_{\omega}X$ of that proposition. By Theorem \ref{identification theorem} applied with this visual metric and $\e = \log \alpha$, the identification $\varphi:\p_{\omega}X \rightarrow \p X_{\rho}$ is biLipschitz. Hence the induced identification $\bar{Z} \cong \p X_{\e}$ given by $\varphi \circ \psi$ is also biLipschitz. Furthermore all of the parameters involved in the biLipschitz constant can be taken to depend only on $\alpha$ and $\tau$ by the results of this section. 
\end{proof}

\begin{rem}\label{bounded filling}
For $k \in \Z$ there is a canonical correspondence between hyperbolic fillings with fixed parameters $\alpha,\tau > 1$ of the metric spaces $(Z,d)$ and $(Z,\alpha^{-k}d)$ given by considering $\alpha^{-n}$-separated sets in $(Z,d)$ as $\alpha^{-n-k}$-separated sets in $(Z,\alpha^{-k}d)$. Thus when $Z$ is bounded there is no harm in assuming that $\mathrm{diam}\, Z < 1$ by multiplying the metric by $\alpha^{-k}$ for $k$ sufficiently large. The hyperbolic filling $X$ can then be written as $X = X_{\geq 0} \cup X_{\leq 0}$, where $X_{\geq 0} = h^{-1}([0,\infty))$ is the set of all points of nonnegative height and $X_{\leq 0} = h^{-1}((-\infty,0])$ is the set of all points of nonpositive height. The condition $\mathrm{diam}\, Z < 1$ implies that the vertex sets $V_{n} = \{v_{n}\}$ for $n \leq 0$ consist only of a single point, and in particular $X_{\leq 0}$ is simply a descending geodesic ray starting from $v_{0}$. The space $X$ is isometric to the ray augmentation of $X_{\geq 0}$ based at $v_{0}$, in the language of Definition \ref{ray augment}. 

The graph $X_{\geq 0}$ is essentially the hyperbolic filling of $Z$ constructed in \cite{BBS21}, with the exceptions that they have a stricter condition for the placement of vertical edges and that they require an additional nesting condition $S_{m} \subset S_{n}$ for $m \leq n$. They uniformize this filling for \emph{all} $\tau > 1$ using the density $\rho_{\e,v_{0}}(x) = e^{-\e |xv_{0}|}$ for $0 < \e \leq \log \alpha$, for which it is easy to see that $\rho_{\e,v_{0}} = \rho_{\e}|_{X \geq 0}$, where $\rho_{\e}(x) = e^{-\e h(x)}$ and $h: X \rightarrow \R$ is the height function. When $\tau$ satisfies \eqref{tau requirement} we can use Theorem \ref{filling theorem} to deduce their results from ours, up to some minor differences in the definition of the hyperbolic filling. When $\tau$ is close to $1$ it is possible to realize trees as hyperbolic fillings \cite[Theorem 7.1]{BBS21}, whereas when $\tau$ satisfies \eqref{tau requirement} the hyperbolic filling is only a tree if $Z$ consists of a single point (by Lemma \ref{cone adjacent}). 
\end{rem}

\bibliographystyle{plain}
\bibliography{Uniform}

\end{document}